\newif\iffinal
\finalfalse	

\documentclass[letterpaper,11pt,reqno]{amsart}
\RequirePackage[utf8]{inputenc}
\usepackage[portrait,margin=2.5cm]{geometry}
\usepackage{mathrsfs}
\usepackage{verbatim}
\usepackage[colorlinks, linkcolor=blue, citecolor=violet, pdfborder={0 1 0}]{hyperref}
\usepackage[foot]{amsaddr}
\usepackage{amssymb,amsthm,amsfonts,amsbsy,latexsym,dsfont}
\usepackage[textsize=small]{todonotes}       
\usepackage{graphicx}
\usepackage[numeric,initials,nobysame]{amsrefs}
\usepackage{upref,setspace}
\usepackage{xcolor,colortbl}
\usepackage{enumerate}
\usepackage{graphicx}
\usepackage{subcaption}   
\usepackage{bm}
\usepackage{tikz}
\usepackage{bigints}

\usetikzlibrary{calc,intersections,through,backgrounds,shapes.geometric}
\usetikzlibrary{graphs}
\tikzset{every path/.style={line width=.07 cm}}
\usepackage[normalem]{ulem}

\allowdisplaybreaks

\definecolor{darkblue}{rgb}{0,0,0.6}

\newenvironment{enumeratei}{\begin{enumerate}[\upshape (i)]}{\end{enumerate}}
\newenvironment{enumeratea}{\begin{enumerate}[\upshape (a)]}{\end{enumerate}}

\usepackage{bbm}

\usepackage{paralist}

\newenvironment{inparaenumn}{\begin{inparaenum}[\upshape \bfseries (i) ]}{\end{inparaenum}}

\newenvironment{inparaenuma}{\begin{inparaenum}[\upshape \bfseries (a) ]}{\end{inparaenum}}

\numberwithin{equation}{section}
\numberwithin{figure}{section}
\numberwithin{table}{section}

\newtheorem{thm}{Theorem}[section]
\newtheorem{lem}[thm]{Lemma}

\newtheorem{cor}[thm]{Corollary}

\newtheorem{prop}[thm]{Proposition}
\newtheorem{defn}[thm]{Definition}

\newtheorem{ass}[thm]{Assumption}

\newtheorem{lemma}[thm]{Lemma}

\theoremstyle{definition}
\newtheorem{rem}{Remark}

\newcommand{\ra}{\rangle}
\newcommand{\la}{\langle}

\newcommand{\ind}{\mathds{1}}
\newcommand{\eps}{\varepsilon}

\newcommand{\norm}[1]{\left\Vert#1\right\Vert}
\newcommand{\abs}[1]{\left\vert#1\right\vert}
\newcommand{\set}[1]{\left\{#1\right\}}

\newcommand{\probc}{\stackrel{\mathrm{P}}{\longrightarrow}}

\newcommand{\weakc}{\stackrel{\mathrm{w}}{\longrightarrow}}

\newcommand{\at}[2][]{#1|_{#2}}
\newcommand{\mycomment}[1]{}

\def\qed{ \hfill $\blacksquare$}

\newcommand{\cA}{\mathcal{A}}\newcommand{\cC}{\mathcal{C}}
\newcommand{\cE}{\mathcal{E}}\newcommand{\cF}{\mathcal{F}}
\newcommand{\cG}{\mathcal{G}}\newcommand{\cI}{\mathcal{I}}
\newcommand{\cJ}{\mathcal{J}}\newcommand{\cK}{\mathcal{K}}\newcommand{\cL}{\mathcal{L}}
\newcommand{\cM}{\mathcal{M}}\newcommand{\cN}{\mathcal{N}}
\newcommand{\cP}{\mathcal{P}}
\newcommand{\cS}{\mathcal{S}}
\newcommand{\cV}{\mathcal{V}}\newcommand{\cW}{\mathcal{W}}

\newcommand{\vC}{\mathbf{C}}

\newcommand{\vI}{\mathbf{I}}

\newcommand{\vU}{\mathbf{U}}\newcommand{\vV}{\mathbf{V}}
\newcommand{\vX}{\mathbf{X}}
\newcommand{\vc}{\mathbf{c}}
\newcommand{\ve}{\mathbf{e}}\newcommand{\vf}{\mathbf{f}}

\newcommand{\vt}{\mathbf{t}}\newcommand{\vu}{\mathbf{u}}
\newcommand{\vv}{\mathbf{v}}
\newcommand{\vz}{\mathbf{z}}

\newcommand{\mv}[1]{\boldsymbol{#1}}\newcommand{\mvzero}{\boldsymbol{0}}
\newcommand{\mvone}{\boldsymbol{1}}

\newcommand{\mvA}{\boldsymbol{A}}\newcommand{\mvB}{\boldsymbol{B}}

\newcommand{\mvH}{\boldsymbol{H}}

\newcommand{\mvM}{\boldsymbol{M}}
\newcommand{\mvQ}{\boldsymbol{Q}}

\newcommand{\mvY}{\boldsymbol{Y}}

\newcommand{\mvk}{\boldsymbol{k}}\newcommand{\mvl}{\boldsymbol{l}}\newcommand{\mvm}{\boldsymbol{m}}
\newcommand{\mvn}{\boldsymbol{n}}

\newcommand{\mvu}{\boldsymbol{u}}\newcommand{\mvv}{\boldsymbol{v}}
\newcommand{\mvx}{\boldsymbol{x}}
\newcommand{\mvy}{\boldsymbol{y}}
\newcommand{\mvz}{\boldsymbol{z}}

\newcommand{\mvpi}{\boldsymbol{\pi}}

\newcommand{\mvxi}{\boldsymbol{\xi}}

\newcommand{\fF}{\mathfrak{F}}

\newcommand{\fL}{\mathfrak{L}}
\newcommand{\fN}{\mathfrak{N}}

\newcommand{\fT}{\mathfrak{T}}

\newcommand{\bb}[1]{\mathbb{#1}}
\newcommand{\bB}{\mathbb{B}}\newcommand{\bC}{\mathbb{C}}
\newcommand{\bD}{\mathbb{D}}

\newcommand{\bL}{\mathbb{L}}
\newcommand{\bN}{\mathbb{N}}
\newcommand{\bP}{\mathbb{P}}\newcommand{\bR}{\mathbb{R}}
\newcommand{\bT}{\mathbb{T}}

\DeclareMathOperator{\E}{\mathbb{E}}
\DeclareMathOperator{\pr}{\mathbb{P}}

\DeclareMathOperator{\var}{Var}
\DeclareMathOperator{\cov}{Cov}

 \DeclareMathOperator{\Pois}{Poisson}

 \DeclareMathOperator{\IRG}{IRG}

 \DeclareMathOperator{\BP}{BP}

\newcommand{\sss}{\scriptscriptstyle}

\newcommand{\erdos}{Erd\H{o}s-R\'enyi }

\newcommand{\ER}{{\text{ER}}} 
\newcommand{\convd}{\stackrel{d}{\longrightarrow}}
\newcommand{\convp}{\stackrel{\prob}{\longrightarrow}}

\usepackage{pdfsync}

\definecolor{aqua}{rgb}{0.0, 1.0, 1.0}
\definecolor{boo}{rgb}{1.0, 0.0, 1.0}
\definecolor{stred}{rgb}{1.0, 0.44, 0.37}
 \newcommand{\abb}{}
\newcommand{\ab}{}

\newcommand{\sa}{}

\newcommand{\MBP}{{\sf MBP}}
\newcommand{\dense}{{\sf dense}}
\newcommand{\inpr}[2]{\langle #1, #2 \rangle}

\usepackage{accents}

\newcommand{\bt}{\mathbf{t}}

\newcommand{\prob}{\mathbb{P}}

\DeclareMathAlphabet\mathbfcal{OMS}{cmsy}{b}{n}

\usepackage{accents}
\newcommand\thickbar[1]{\accentset{\rule{.6em}{1.3pt}}{#1}}

\def\beq{ \begin{equation} }
\def\eeq{ \end{equation} }

\newcommand{\thkappa}{\thickbar{\kappa}}

\newcommand{\fanz}{\mathbf{[0]}}
\newcommand{\fano}{\mathbf{[1]}}

\usepackage{tgpagella,eulervm}

\begin{document}

\title[Functional Central limit theorems for Inhomogeneous random graphs]{Functional Central limit theorems for microscopic and macroscopic functionals of inhomogeneous random graphs}

\date{}

\author[Bhamidi]{Shankar Bhamidi}
\author[Budhiraja]{Amarjit Budhiraja}
\address{Department of Statistics and Operations Research, 304 Hanes Hall, University of North Carolina, Chapel Hill, NC 27599}
\email{bhamidi@email.unc.edu, budhiraja@email.unc.edu, sakshay@unc.edu}
\author[Sakanaveeti]{Akshay Sakanaveeti }

\begin{abstract}
We study inhomogeneous random graphs with a finite type space. For a natural generalization of the model as a dynamic network-valued process, the paper establishes the following results: 
\begin{inparaenuma}
    \item   {
    Functional central limit theorems for the infinite vector of microscopic type-densities and characterizations of the limits as  infinite-dimensional \sa{ conditionally }Gaussian processes in a certain Banach space.}
    \item  Functional (joint) central limit theorems for macroscopic observables of the giant component in the supercritical regime including size,  surplus and number of vertices of various types in the giant component. As a corollary this provides  central limit theorems for the size of the largest connected component, its surplus, and its type vector, for percolation on dense graphs obtained from a finite type Graphon.
    \item  
   Central limit theorem for the {\it weight} of the minimum spanning tree with random {\it i.i.d.} Exponential edge weights on dense graph sequences driven by an underlying finite type graphon. 
\end{inparaenuma}

\end{abstract}

\maketitle

\section{Introduction}
\label{sec:int}
Driven  by mathematical questions on properties of large disordered media as well as applications arising from an explosion of data on real world systems, there has been significant impetus from a range of disciplines in formulating network models, understanding properties of these models either numerically or mathematically, and applying these models in relevant domains to gain scientific insight. Starting from the inception of the field one of the central questions of interest for a typical model, e.g. in the work of \erdos \cite{erdos1960evolution}, is connectivity, i.e.  regimes of the driving parameters of the network model where one has a ``giant component''. More precisely, writing $\set{\cG_n:n\geq 1}$ for a generic sequence of random graph models, where $\cG_n$ is a (random) network on vertex set $[n]:=\set{1,2,\ldots, n}$,   many such models have a parameter $t$ connected to the edge density for the sequence of networks and a model dependent critical time $t_c$ such that for $t< t_c$, the size (number of vertices) of the the largest connected component $\abs{\cC_{(1),n}(t)} = o_{\pr}(n)$ while for $t> t_c$, $\abs{\cC_{(1),n}(t)} \sim f(t) n$ for a strictly positive function $f(\cdot)$. The first regime is often referred to as the {\bf subcritical} regime while the second regime the {\bf supercritical} regime of the network model.

Foundational work on such questions include the original work of \erdos on the now eponymous and fundamental model \cite{erdos1960evolution}, the configuration model \cite{molloy1995critical}, and most relevant for this work, the inhomogeneous random graph model \cite{bollobas2007phase} and their dense cousins, the world of graphons \cite{lovasz-book}. We refer the interested reader to the books \cites{boll-book,durrett-book,janson-luczak-book,remco-book-1,remco-rgcn2,lovasz-book} and the references therein for the beautiful mathematics that has resulted from trying to understand such phenomenon. Also, there is much to say about  the the critical regime $t=t_c$ however, since this regime is only tangentially related to the goal of this paper, we defer further discussion to Section \ref{sec:disc}. 

Results on the largest component, $\cC_{(1),n}(t)$, noted above, give  laws of large number behavior of the form $|\cC_{(1),n}(t)|/n\convp f(t)$. These lead naturally to more refined questions such as: 
\begin{enumeratea}
 \item Laws of large numbers as well as fluctuations for {prevalence of different types of} microscopic components, for example the density of components of any fixed size $k\geq 1$ denoted by $\pi_n(k;t)$,  for fixed $t$ in the sub and supercritical regime; see for example \cite{pittel1990tree} for refined analysis in the setting of the \erdos random graph.  
\item Second order fluctuations in the supercritical regime,  typically phrased in terms of understanding asymptotics of the sequence $\set{(\abs{\cC_{(1),n}(t)} - nf(t))/\sqrt{n}: n\geq 1}$ for {\bf fixed } $t> t_c$. In the setting of the \erdos random graph, see e.g. \cites{pittel1990tree,bollobas2012asymptotic}; for related results in the configuration model, see e.g. \cites{barbour2019central,ball27asymptotic}. 
   
    \item First and second order behavior of other functionals of the giant component when $t> t_c$, for example the surplus of the giant, namely the number of extra edges in the giant compared to a tree of the same size, the size of the $2$-core of the giant and corresponding tree mantle; see e.g. \cite{pittel2005counting} for definitions of these objects and an analysis in the setting of the \erdos random graph.  
    \item {Functional limit theorems:} Many random graph models can naturally be viewed, via various constructions,  as a dynamic graph valued process $\set{\cG_n(t): t\ge 0}$. One construction in the  \erdos setting, for example,  assigns each of the ${n \choose 2} $ unordered pairs $\ve = \set{i,j}$ an independent and uniform $[0,1]$ valued weight $U_{\ve}$ and then for each fixed $t$, retains only those edges $\ve$ with $U_{\ve} \leq t$. This graph valued process starts from a collection of $n$ disconnected vertices at $t=0$ and terminates at the complete graph for $t=1$. {One can then look at, for example, the density of components of finite size $k$,  $\set{\pi_n(k;t):  k\geq 0}_{t\ge 0}$ and study  a functional central limit theorem for this infinite dimensional stochastic process.} The importance of such results in understanding for example, the behavior of the weight of minimal spanning tree on the complete graph equipped with random edge weights was explored in foundational papers by Frieze \cite{frieze1985value} and Janson \cite{janson1995minimal}. 
\end{enumeratea}

\subsection{Goals and contributions of this paper}
This paper considers the inhomogeneous random graph model \cite{bollobas2007phase}. In the finite type space setting, for the natural generalization of the model as a dynamic network valued process, this paper accomplishes the following major goals: 
\begin{enumeratea}
    \item {\bf Functional central limit theorems for  densities of microscopic components of various types:} We prove joint functional central limit theorems (FCLT) for the \emph{microscopic} type-density vector and characterize the limit as an  infinite dimensional (conditionally) Gaussian process in an appropriate Banach space \abb{given through the solution of a stochastic differential equation (SDE) driven by a cylindrical Brownian motion in $\ell_2$}. A special case of this result applied to the setting of the \erdos random graph answer the conjecture in Remark 1.3 of \cite{janson1995minimal}, while our general results make a start on \cite{aldous1999deterministic}*{Open Problem 9} regarding Gaussian fluctuations for coalescent systems for a specific setting of the {\em coagulation kernel}. We discuss these implications further in Section \ref{sec:disc}. 
    \item {\bf Functional central limit theorems for macroscopic functionals:} In the supercritical regime, the above results are leveraged to derive (joint) functional central limit theorems for macroscopic observables such as {the number of connected components,  the size and surplus of the giant component, and the number of different types of vertices in the giant.}  In the specific case of the \erdos random graph, this extends the work in \cite{pittel2005counting} that considers  fixed time instant central limit theorems for such objects. 
     \item {\bf Asymptotics for the weight of the MST on graphon modulated random graphs: } Using the above results as ingredients we then  prove  central limit theorems for the weight of the minimal spanning tree with random i.i.d. Exponential edge weights on {\bf dense} graph sequences driven by an underlying finite type kernel \cite{lovasz-book}, extending the famous work of Janson related to the fluctuations of the weight of the MST on the complete graph \cite{janson1995minimal} and strengthening the law of large number results for the weight recently derived in \cite{hladky2023random}.
     \end{enumeratea}
      \abb{Establishing the limit theorem in part (c) is one of the important motivations for studying (infinite dimensional) FCLT in parts (a) and (b). Indeed the weight of the MST can be represented as a  time integrated functional of the infinite sum of the type-density vector (see e.g. Lemma \ref{thm:mst-wn0-convg}) and one of the key technical challenges is to understand the behavior of the tail sums of the type-density vectors over time intervals that include the critical window. The asymptotic type-vector densities decay exponentially in the size of the type-vector outside the critical window but only polynomially within the window (see comments below \eqref{eqn:sde-gamma-defn}). This makes proving a suitable FCLT for the infinite type-density vector, over arbitrary time intervals, challenging. In the critical window, as one transitions from a collection of small components of maximal size $O_{\pr}(\log{n})$ to one macroscopic component (``the giant'') of $\Theta_{\pr}(n)$, the various error terms may explode, and in fact the coefficients of the limiting SDE  diverge as one approaches the critical point. We circumvent this poor behavior in the critical window by establishing separate FCLT in the subcritical regime (i.e. below the critical window) and the supercritical regime (above the critical window) and by carefully analyzing the behavior of the contribution to the cost of the MST over the critical window.}
     
     \abb{Inhomogeneous random graph models  with not just finite but infinite type space cover a  large number of models in probabilistic combinatorics (cf. the survey in \cite{bollobas2007phase}) and have shown up as crucial tools in understanding other network models such as bounded size rules \cites{bhamidi2014augmented,bhamidi2014bounded,bhamidi2015aggregation,spencer2007birth}. As described in \cite{bollobas2007phase} for laws of large numbers and in \cite{andreis2023large} for large deviations, in these general settings, one of the crucial ingredients is approximation using appropriately defined finite type models and leveraging technical estimates for the finite-type setting to understand the general setting. The current work develops such technical tools for finite type processes which we expect will be key in extending the results in this paper, to the more general setting with arbitrary types spaces considered in \cite{bollobas2007phase}. 
       }


\subsection{Organization of the paper}
 Section \ref{sec:defn} contains definitions of the network model and known results from \cite{bollobas2007phase} regarding the sub and supercritical regime of this model. Section \ref{sec:results} describes our main results. Relevance of these results, proof techniques and related literature are discussed in Section \ref{sec:disc}. The remaining sections are devoted to proofs. 

\section{Definitions}
\label{sec:defn}
\subsection{Notation}
\label{sec:not}
We use $\probc$ and $\weakc$ to denote convergence of random variables in probability and in distribution  respectively.
All the unspecified limits are taken as $n \to \infty$.
Given a sequence of events $\{E_n\}_{n\ge 1}$, we say $E_n$  occurs with high probability (whp) if $\pr(\{E_n\}) \to 1$. Equality of random variables in distribution will be denoted as $\stackrel{d}{=}$.
The standard Landau notation $O,\Omega,\Theta, o$ (and corresponding relations $O_{\pr}, \Omega_{\pr}$ etc) are defined in the usual manner. For example,   given a sequence of real valued random variables $\{\xi_n\}_{n\ge 1}$ and a function $f: \bN \to (0,\infty)$, we say $\xi_n = O_{\pr} (f)$ if there is a $C\in (0,\infty)$ such that $\xi_n \le C f(n)$ whp, and we say $\xi_n = \Omega_{\pr}(f)$ if there is a  $C \in (0,\infty)$ such that $\xi_n \ge Cf(n)$ whp. We say that $\xi_n = \Theta_{\pr}(f)$ if $\xi_n = O_{\pr}(f)$  and $\xi_n = \Omega_{\pr}(f)$; and we say  $\xi_n = o_{\pr} (f)$ if $\xi_n/f(n) \probc 0$. For $K\geq 1$ and Euclidean space $\bR^K$, we use boldface e.g. $\mvx = (x_1, x_2, \ldots, x_K)$ for vectors and will use the standard notation $\inpr{\mvx}{\mvy} := \sum_{i=1}^K x_i y_i$ for the inner product. We use $\mvzero$ for the zero vector and $\mvone = (1,1,\ldots, 1)$ for the vector of all ones. Given a Polish space $\cS$, by a kernel $\kappa$ on $\cS$, we will mean a symmetric measurable map from $\cS\times \cS$ to $\bR_+ = [0,\infty)$. For a Polish space $\cS$, $\bD([0,T]:\cS)$ (resp. $\bD([0,\infty):\cS)$) denote the space of $\cS$-valued right continuous functions with left limits (RCLL) from $[0,T]$ (resp. $[0,\infty)$) equipped with the usual Skorohod topology.  For $f: \cS \to \bR$, $\|f\|_{\infty} := \sup_{x\in \cS} |f(x)|$.
For a RCLL function $f:[0,\infty) \to \bR$,
we write $\Delta f(t) = f(t) - f(t-)$, $t >0$. {A sequence of processes is said to be $\bC$-tight in $\bD([0,T]:\cS)$ (resp. $\bD([0,\infty):\cS)$) if it is tight and any weak limit point takes values in $\bC([0,T]:\cS)$ (resp. $\bC([0,\infty):\cS)$) a.s.}

\subsection{Inhomogeneous random graph models}\label{subsec:irg-definitions}
We largely follow  \cite{bollobas2007phase}. There are a number of asymptotically equivalent formulations (cf.  \cite{janson2010asymptotic}), and we fix one such formulation that is particularly useful for the dynamic description it allows. 

Let $\cS$ be a Polish space and a kernel $\kappa$ on $\cS$.  Given a vertex set $\cV$ (typically assumed to be $\cV = [n]:= \set{1,2,\ldots ,n}$ for some $n\geq 2)$ and $\mvx = (x_v \in \cS:v\in \cV)$, we construct a random graph by placing an edge between vertices $u, v \in \cV$, $u\neq v$, with probability $p_{u,v}$ independently across edges where,
    \begin{equation}
        \label{eqn:puv-def}
        p_{u,v} = 1-\exp{\left(-\kappa(x_u,x_v)\right)}.
    \end{equation}
We denote this random graph as $\IRG(\kappa,\cV,\mvx)$. This paper considers the setting where $\cS$ is a finite set and the graph is  {\bf sparse} in the sense described below.

For the remainder of the paper, fix $K \geq 1$ and set $\cS = [K]:= \set{1,2,\ldots, K}$.   
The main object of study  is the following dynamic version of  inhomogeneous random graph models in which the kernels scale as $O(1/n)$.

\begin{defn}[Inhomogeneous random graph (IRG), dynamic version]
\label{def:irg} {Given  sequences of types $\mvx^n = \set{x^n_i \in \cS:i\in [n]}$, and  kernels $\set{\kappa_n}_{n\geq 1 }$ on $\cS$, we define a sequence of  random graph processes $\set{\cG_n(t,\mvx^n,\kappa_n):t\geq 0}_{n\geq 1}$ as follows:} For each $n\in \bN$, let $\set{\cP^n_{ij}: 1\leq i < j \leq n}$ be a collection of independent Poisson processes with $\cP^n_{ij} $ having rate $\kappa_n(x^n_i, x^n_j)/n$. The graph $\cG_n(t,\mvx^n,\kappa_n)$ on vertex set $[n]$ is obtained by placing edges between vertices $i,j \in [n]$, $i\neq j$, at the event times of the Poisson process $\cP^n_{ij}$ up to  time $t$. 
\end{defn}

We note the following.
\begin{enumeratea}
\item The graph $\cG_n(t,\mvx^n,\kappa_n)$ may have multiple edges between distinct vertices but it has no self-loops.

\item 

For each  $t \geq 0$, and any two vertices $i,j \in [n]$, $i\neq j$, at time $t \geq 0$ 
$$
\pr\left( \text{ at least one edge between } i \text{ and } j \text{ in } \cG_n(t,\mvx^n,\kappa_n) \right) = 1-\exp(-t\kappa_n(x_i,x_j)/n).$$ Let $\bar{\cG}_n(t,\mvx^n,\kappa_n)$ be the graph obtained from $\cG_n(t,\mvx^n,\kappa_n)$ by collapsing multiple edges into a single edge i.e., an edge exists between vertices $1 \leq i \neq  j \leq n$  in $ \bar{\cG}_n(t,\mvx^n,\kappa_n)$ if and only if there is at least one edge in $\cG_n(t,\mvx^n,\kappa_n)$. Note that  $\bar{\cG}_n(t,\mvx^n,\kappa_n)$ has the same connectivity structure as $\cG_n(t,\mvx^n,\kappa_n)$. Also, recalling \eqref{eqn:puv-def} we see that, for any  $t \geq 0$, 
$$\bar{\cG}_n(t,\mvx^n,\kappa_n) \stackrel{d}{=}\IRG\left(\frac{t\kappa_n}{n}, \cV = [n], \mvx^n\right).$$
\end{enumeratea}
 
Henceforth we will suppress the $\mvx^n$ and $\kappa_n$ in the notation and simply write $\cG_n(t)$ instead of $\cG_n(t,\mvx^n,\kappa_n)$.
An important mathematical tool in the analysis of the IRG  are multi-type branching processes \cite{athreya2004branching}*{Chapter 5}, which arise as local weak limits around typical vertices \cites{remco-rgcn2, bollobas2007phase}. We recall some basic definitions and properties below.

\begin{defn}[Multi-type Poisson Branching Process]\label{defn:MBP} Let $\mu', \mu$, and $\kappa$ be a probability measure, a measure, and  a kernel, on $\cS$, respectively. Then $\MBP_{\mu'}(\kappa,\mu)$ denotes the multi-type branching process with type space $\cS$, {\bf initial measure} $\mu'$ and {\bf driving measure} $\mu$, defined as follows:
\begin{enumeratei}
    \item We start with one individual with type in $\cS$ sampled according to $\mu'$.
    \item Any individual of type $x$ has $\Pois\left(\kappa(x,y)\mu(y)\right)$ many off-springs of type $y$ for each $y\in [K]$, independent for different $y$ and of other individuals in the population.
\end{enumeratei} 
\end{defn}
When $\mu^\prime = \delta_x$, write $\MBP_x(\kappa,\mu)$ for the multi-type branching process where the first individual is of type $x \in \cS$. When the initial distribution  is clear from the context, we simply write the branching process as $\MBP(\kappa,\mu)$. 

\begin{defn}[Irreducibility]
\label{ass:irred}
    A kernel $\kappa$ on $\cS$ is called irreducible if  $\kappa(x,y) >0$ for all $x,y\in [K]$. For a measure $\mu$ on $\cS$, the pair $(\kappa,\mu)$, is called irreducible if $\kappa$ is irreducible and   $\mu(x) >0 $ for all $x\in [K]$.
\end{defn}

A basic question of interest for such branching processes is their extinction probability (and its complement, the survival probability). These play an important role  in characterizing the  asymptotics of IRG. For describing these results we will need the following definition. Note that  $\bR^K$
can be viewed as the space of functions from $\cS$ to $\bR$ (denoted as $\bB(\cS)$). We will use this identification  throughout without additional comments.

\begin{defn}[Integral Operator]
   With $\kappa$ and $\mu$ as in Definition \ref{defn:MBP},  define $T_{\kappa,\mu}: \bR^{K} \to \bR^{K}$:
   \begin{align*}
        T_{\kappa,\mu} f(x) = \int\limits_{\cS} \kappa(x,y)f(y)d\mu(y) =\sum_{y\in [K]}\kappa(x,y)\mu(y)f(y), \;\; f \in \bB(\cS).
    \end{align*} Let $\norm{T_{\kappa,\mu}} = \sup_{\norm{f}_2 = 1} \norm{T_{\kappa,\mu} f}_2$ be the operator norm, where for $y \in \bR^K$, $\|y\|_2 = \left(\sum_{i=1}^K y_i^2 \mu(i)\right)^{1/2}$.
\end{defn}

The branching process $\MBP(\kappa,\mu)$ is said to be \emph{sub-critical} if $\norm{T_{\kappa,\mu}} < 1$ and \emph{super-critical} if   $\norm{T_{\kappa,\mu}} > 1$ \cite{athreya2004branching}. A \emph{sub-critical} irreducible branching process goes extinct almost surely {(\cite{bollobas2007phase}*{ Theorem 6.2})} . A super-critical irreducible branching process survives with positive probability. This probability can be characterized as follows. Consider for an irreducible pair $(\kappa,\mu)$, the nonlinear map
 $\Phi_{\kappa,\mu}:\bR^{K}  \to \bR^{K}$ defined as \[\Phi_{\kappa,\mu}f(x) =1-e^{-T_{\kappa,\mu}f(x)}, \qquad x\in \cS, \; f \in \bB(\cS).\] 
Let $\rho_{\kappa,\mu} \in \bR^{K}$ be the maximal non-negative fixed point of $\Phi_{\kappa,\mu}$ i.e \[\rho_{\kappa,\mu}(x) = \Phi_{\kappa,\mu}\rho_{\kappa,\mu}(x), \qquad  \mbox{ for all }  x\in \cS,\]
and $\rho_{\kappa,\mu}$ coordinate-wise dominates any other fixed point. Uniqueness of $\rho_{\kappa,\mu}$ follows from \cite{bollobas2007phase}*{ Theorem 6.2}) 
Then, $\rho_{\kappa,\mu}(x)$ is the survival probability of $\MBP_x(\kappa,\mu)$. Further, for any $x \in \cS$, $\rho_{\kappa,\mu}(x)$ is positive if and only if $\norm{T_{\kappa,\mu}} > 1$ {(\cite{bollobas2007phase}*{ Theorem 6.1, 6.2})}.
We define the {\bf dual measure} $\hat{\mu}$ associated with $\MBP(\kappa,\mu)$ as,
\begin{equation}
    \label{eqn:dual-prob}
    \hat{\mu}(x) = (1-\rho_{\kappa,\mu}(x))\mu(x), \qquad x\in \cS. 
\end{equation}
We call $\MBP(\kappa,\hat{\mu})$ to be the {\bf dual} branching process corresponding to $\MBP(\kappa,\mu)$.

\begin{lem}[{\cite{bollobas2007phase}*{Lemma 6.6 and Theorem 6.7}}]
For an irreducible pair $(\kappa,\mu)$,
 the dual branching process $\MBP_x(\kappa,\hat{\mu})$  is sub-critical for any $x\in \cS$,.
\end{lem}

For the next Lemma, see Section 6 and the discussion preceding Lemma 14.13 in \cite{bollobas2007phase}. Let $\bT_{\cS}$ be the space of all rooted finite trees where each vertex has a type in $\cS$.
Denote by $|\MBP(\kappa,\mu)|$ the total number of individuals in $\MBP(\kappa,\mu)$.

\begin{lem}\label{lem:conditional-dual}
     For every $x\in \cS$, conditional on the extinction of $\MBP_x(\kappa,\mu)$, the multi-type branching process $\MBP_x(\kappa,\mu)$ has the same law as $\MBP_x(\kappa,\hat{\mu})$. Namely, for any $\bt \in \bT_{\cS}$, 
$$\pr(\MBP_x(\kappa,\mu) = \bt \mid |\MBP_x(\kappa,\mu)| < \infty) = \pr(\MBP_x(\kappa,\hat{\mu}) = \bt).$$
\end{lem}

Now we return to the IRG, $\set{\cG_n(t,\mvx,\kappa_n):t\geq 0}_{n\geq 1}$, from Definition \ref{def:irg}. 
Let $L_n(t) = L_n(t,\kappa_n,\mvx)$ be the size of the largest component in $\cG_n(t)$. Branching processes of the form discussed above provide  the following characterization of the phase transition for inhomogeneous random graphs. Such a characterization is in fact known in a much more general setting.
\begin{thm}[\cite{bollobas2007phase}]
\label{thm:irg-boll}
Consider the $\IRG$ $\cG_n(t,\mvx^n,\kappa_n)$ in Definition \ref{def:irg}. Assume the empirical measure of types $\mu_n(\cdot) = \frac{1}{n}\sum_{i=1}^n\delta_{x_i^{n}}$ and kernels $\kappa_n$ satisfy $\mu_n \to \mu$ and $\kappa_n \to \kappa$ pointwise on $S$ and $S\times S$, respectively, where $\MBP(\kappa, \mu)$ is irreducible.    Define $t_c=t_c(\kappa,\mu) ={\norm{T_{\kappa,\mu}}}^{-1}$. Then:
\begin{enumeratea}
    \item The parameter $t_c$ is the critical time for the sequence $\{\cG_n(t)\}$ in the sense that for each $t< t_c$ we have $L_n(t)/n\convp 0$ while for $t> t_c$, $L_n(t)/n \convp l(t)$ where $l(t) > 0$ is the survival probability of the multi-type branching process $\MBP_{\mu}(t\kappa,\mu)$.
    \item For any  $t > t_c$, the second largest component in $\cG_n(t)$  is $O_{\pr}(\log n)$.
    \item For any $t< t_c$, $L_n(t) = O_{\pr}(\log{n})$. 
\end{enumeratea}

\end{thm}
The three parts in the above theorem are  Theorems 3.1, 12.6, and 3.12(i) in \cite{bollobas2007phase}, respectively.
The regime $t < t_c $ is called the \emph{sub-critical} regime  while the regime $ t > t_c $ is called \emph{super-critical}, as the largest component grows to an order of $n$, the same order as the size of the network. The time instant $t = t_c$ is called the \emph{critical point}.

\section{Results}
\label{sec:results}
In this section we present our main results.
 In Section \ref{sec:res-micro} we describe functional law of large numbers  and central limit results for finite dimensional projections of the vector of the density of components of various types. Section \ref{sec:res-micro-inf} strengthens the results of the prior section by establishing the weak convergence of the above vector in a suitable infinite dimensional Banach space. This result is then used in Section \ref{sec:res-macro}
to establish functional CLTs for several macroscopic functionals in the supercritical regime.  We conclude in Section \ref{sec:def-graphon-mst} with a result on fluctuations for the weight of the minimal spanning tree on certain dense random graph sequence.

\subsection{Asymptotics for component type vectors.}
\label{sec:res-micro}
We will consider the sequence of IRG, $\cG_n(t,\mvx,\kappa_n)$ introduced in Definition \ref{def:irg} and study its asymptotic behavior as $n\to \infty$. For the law of large number results we will only need pointwise convergence of $(\mu_n, \kappa_n)$ as in Theorem \ref{thm:irg-boll}, however for studying CLT results, we will need the last two parts of the following assumption on the scaling of fluctuations of the limit empirical density of types and of the kernel about their limits. Note  that a real function on $\cS\times \cS$ can be identified with a $K\times K$ real matrix. This identification will be used throughout without further comment.

\begin{ass}[Second order fluctuations of types and kernel]
\label{ass:irg}
    Let $\mu$ be a measure on $\cS$ and $\kappa$ a kernel on $\cS$. Suppose that $(\mu,\kappa)$ is an irreducible pair. Let $\{\kappa_n\}_{n\geq 1}$ be as in Definition \ref{def:irg} and assume that  $\mvx^n$ introduced there defines a collection of $\cS$ valued random variables. Conditional on $\mvx^n$, let $\set{\cP^n_{ij}: 1\leq i < j \leq n}$ be a collection of independent Poisson processes with $\cP^n_{ij} $ having rate $\kappa_n(x^n_i, x^n_j)/n$. Let $\mu_n = \frac{1}{n}\sum_{i=1}^n\delta_{x^n_i}$ be the (random) probability measure defined as in Theorem \ref{thm:irg-boll}.
    
   Suppose that as $n \to \infty$, 
    \begin{enumeratea}
        \item Type distribution $\mu_n \to \mu$, pointwise, in probability, and kernel $\kappa_n \to \kappa$, pointwise.
        \item {\bf Initial type density fluctuations:} 

        There exists a probability measure $\nu$ on $\bR^K$ such that, as $n\to \infty$, the distribution of $\sqrt{n}\left(\mu_n-\mu\right)$ converges  to $\nu$. Furthermore, $\int_{\bR^K} \|\vz\|^2\, \nu(d\vz)<\infty$ and  $\int \vz\, \nu(d\vz)=\mvzero$.  
        \item {\bf Convergence of Kernels:} There exists  $K\times K$-symmetric matrix  $\Lambda$ such that $\sqrt{n}(\kappa_n - \kappa) \to \Lambda$, pointwise. 
        
    \end{enumeratea}
\end{ass}
 The IRG constructed using a random sequence $\mvx^n$ and conditionally Poisson processes $\cP^n_{ij} $ as in the above definition will once more be denoted as $\cG_n(t,\mvx^n,\kappa_n)$ or $\cG_n(t)$ for brevity. This is the object that we will be interested in for the rest of this work.
 
Next, we introduce some basic definitions associated with an IRG.

Let $\bN_0 = \set{0,1,2,\ldots}$ and $\mv0$ be the $K$-dimensional zero vector.
\begin{defn}[$\IRG$ Component type space]
    \label{defn:irg-comp-type}
    Define the type space and norm:
\[ \fT  = \set{\mvl = (l_1,l_2, \ldots l_K): l_i \in \bN_0} = \bN_0^K, \qquad \norm{\mvl} = \sum_{i=1}^K l_i.  \] 
 For any time $t\geq 0$, $\mvl \in \fT\setminus\{\mv0\}$ and a connected component $\cC \subset \cG_n(t)$, we say that the type of the component $\cC$ is $\mvl$, written succinctly as $\fF(\cC) = \mvl$ if $\cC \text{ contains } l_j \text{ many vertices of type } j \text{ for } j \in [K]$.  Fix $N\geq 1$ and define the $N$-dimensional finite slice of $\fT$ by,
 $$\fT_N  = \set{\mvl\in \fT: \norm{\mvl}\leq N}.$$ 
\end{defn}
 For $\mvl, \mvk \in \fT$, we  define a total ordering on $\fT$ as,  $\mvl < \mvk$ if \begin{inparaenumn}
    \item  $ \norm{\mvl} < \norm{\mvk}$ or;
    \item when $ \norm{\mvl} =\norm{\mvk}$ then  $l_\tau < k_\tau$ where $\tau = \inf\set{i: l_i \neq k_i}$. 
\end{inparaenumn}  
We say $\mvl\leq \mvk$ if $\mvl = \mvk$ or $\mvl < \mvk$. 

For a function $\vf: \fT \to \bR$, write $\vf^{\fT_{N}}:\fT_{N} \to \bR $ for the restriction of $\vf$ to $\fT_{N}$. Note that with the identification $(\vf^{\fT_{N}}(\mvl): \mvl \in \fT_{N} ) \in \bR^{|\fT_{N}|}$, such a restriction can be viewed as an element of $\bR^{|\fT_{N}|}$.

Now for a vertex  $v\in [n]$ and $t\geq 0$,  denote the connected component containing vertex $v$ in $\cG_n(t)$ by $\cC_n(v,t)$.  Define a $\bR_+^{\fT}$ valued stochastic process $\{\mvpi_n(t), t \ge 0\}$, where denoting $\mvpi_n(t)(\mvl) = \pi_n(\mvl, t)$, for $\mvl \in \fT$,
\begin{align}
    \label{eqn:pin-irg-def}
    \pi_n(\mv0,t) &= \frac{\text{number of edges in } \cG_n(t)}{n}, \\
 \pi_n(\mvl,t) &=  \frac{1}{\norm{\mvl}} \left[\frac{1}{n}\sum_{i=1}^n \ind \set{\fF(\cC_n(i,t)) = \mvl}\right], \qquad \mvl \neq \mv0 \in \fT, \quad t\geq 0.
\end{align} 
Thus, for fixed  $\mvl \in \fT\setminus\{\mv0\}$, $\pi_n(\mvl, t)$ denotes the density  of connected components of type $\mvl$. 
 
Note that, for $t\ge 0$, $\mvpi_n(t)$ takes values in $\cL$
\begin{align}
    \cL = \set{\mvx = (x_{\mvl} \in \bR_+:\mvl\in \fT):  \sum\limits_{\mvl\in \fT}\norm{\mvl} x_{\mvl} \leq 1},
    \label{eqn:cl-IRG}
\end{align}
Let $(\ve_k: k\in [K])$ denote the standard basis vectors in $\bR^K$. 
When $t=0$, we have \begin{align}\label{eqn:intial-point}
    \pi_n({\mvl},0) = \sum_{k \in [K]} \mu_n(k) \ind\set{\mvl = \ve_k}
       \ \ \ \  \text{and we let} \ \ \ \  \pi({\mvl}, 0) := \sum_{k \in [K]}\mu(k) \ind\set{\mvl = \ve_k}, 
\end{align} 
where the second identity above defines limit densities of connected components of various types at time $0$. 
Specifically, under Assumption \ref{ass:irg}(a), $\pi_n(\mvl, 0) \probc \pi({\mvl}, 0)$ for all $\mvl \in \fT$.

 Theorem \ref{thm:odes-lln} below gives a functional law of large numbers for the finite dimensional projections of $\mvpi_n$ in terms of the  limit  deterministic function $\mvpi(\cdot):= \{\pi({\mvl}, \cdot): \mvl \in  \fT\}$. This limit  function can be uniquely characterized through the solution of an infinite system of differential equations.
 In order to describe this system we now introduce some additional notation.
 
 Let $\cM_K$ denote the space of $K \times K$ symmetric matrices with non-negative entries.
 Note that this space can be identified with the space of all kernels on the space $\cS$. We will occasionally suppress $K$ from the notation. For a kernel $\tilde \kappa \in \cM_K$, define the quadratic form: $\theta_{\tilde \kappa}: \bR^K\times \bR^K \to \bR$ as
\begin{equation}
\label{eqn:thetakapp-def}
    (\mvu, \mvv) \mapsto \theta_{\tilde\kappa}(\mvu,\mvv) = \sum_{i,j=1}^Ku_iv_j\tilde\kappa(i,j), \qquad \mvu,\mvv \in \bR^K.
\end{equation} 

For each $\mvl \in \fT$, define  $F_{\mvl}:\cL\times \cM \times \bR_+^K \to \bR$ as: For $(\mvx, \tilde\kappa, \tilde\mu) \in \cL\times \cM \times \bR_+^K$,
\begin{equation}
\begin{aligned}\label{eqn:F_l-defn}
F_{\mvl}(\mvx,\tilde{\kappa},\tilde\mu) &= \frac{1}{2} \theta_{\tilde\kappa}(\tilde\mu,\tilde\mu), \qquad \mvl = \mvzero,  \\
    F_{\mvl}(\mvx,\tilde\kappa,\tilde\mu) &= \frac{1}{2} \sum\limits_{\mvk_1 + \mvk_2 = \mvl}x_{\mvk_1} x_{\mvk_2} \theta_{\tilde\kappa}(\mvk_1,\mvk_2) - x_{\mvl}\theta_{\tilde{\kappa}}(\mvl,\tilde\mu), \qquad \mvl  \in \fT\setminus \{\mvzero\}.
\end{aligned}
\end{equation}
The next proposition introduces an infinite system of ordinary differential equations that will characterize the law of large numbers limit $\mvpi(\cdot)$ described above.

\begin{prop}\label{prop:ode-solutions}
Let $(\kappa, \mu)$ be an irreducible pair.
   The  system of differential equations
    \begin{align}
    \label{eqn:ode-irg}
    \frac{d x({\mvl},t)}{dt} = F_{\mvl}(\mvx(t),\kappa,\mu), \qquad x(\mvl, 0) = \pi(\mvl, 0), \qquad \mvl \in \fT,
\end{align}
with $\pi({\mvl}, 0)$ as in \eqref{eqn:intial-point}, has a unique solution on $[0, \infty)$.
In particular, $x(\mv0,t) = t\theta_{\kappa}(\mu,\mu)/2$ for all $t\geq 0$.
\end{prop}
The proof of the proposition is immediate on observing that the functions $\{F_{\mvl}, \mvl \in \fT\}$ is a lower triangular system (namely, for  $\mvx \in \cL$, $F_{\mvl}(\mvx)$ depends only on  $x_{\mvk}$, $\mvk \le \mvl$) and that, for $\mvx \in \cL$, $F_{\mvl}(\mvx)$ depends linearly on $x_{\mvl}$. We omit the details.

\begin{rem}
\label{rem:rem1}
Such systems of differential equations, and their generalizations to a continuous state space(``inhomogeneous coagulation processes''),  have been previously used to study  large deviations for sparse $\IRG$ models in \cite{andreis2023large}*{Section 2.4}, \cite{kovchegov2023multidimensional}. See Section \ref{sec:disc} for further discussion.
\end{rem}

Let $\mvpi(t) = (\pi({\mvl}, t):\mvl\in \fT)$ denote the unique solution to the differential equations in \eqref{eqn:ode-irg} when $(\kappa, \mu)$ is as in Assumption \ref{ass:irg}(a). Note that, for each $\mvl\in \fT$, $t\mapsto \pi({\mvl}, t)$
is a continuous function on $\bR_+$. Also for each $t\ge 0$, $\mvl \mapsto \pi_n(\mvl, t)$ is a sequence of functions on $\fT$.
Recall  that for fixed $N \in \bN$ and $t\ge 0$, $\mvpi_n^{\fT_N}(\cdot, t)$ denotes the restriction of the latter collection of functions to index set $\fT_N$ (the notation $\mvpi^{\fT_N}(\cdot,t)$ is interpreted in a similar way). The next result describes the connection between the component densities $\mvpi_n$ and the unique solution $\mvpi$ obtained from Proposition \ref{prop:ode-solutions}.
The proof can be found in Section \ref{sec:pfllnclt}.

\begin{thm}[Laws of Large numbers]\label{thm:odes-lln}
Suppose that Assumption \ref{ass:irg}(a) holds and let $\cG_n(t)$ be as introduced below Assumption \ref{ass:irg}.
    For each fixed $N \geq 1$, we have $\mvpi_n^{\fT_N}(\cdot)\convp \mvpi^{\fT_N}(\cdot)$ in $\bD([0,\infty): \bR^{|\fT_N|})$.
\end{thm} 

For fixed time $t$,  the next result describes a well known probabilistic characterization of the limiting functions $\pi({\mvl}, t)$ in terms of the distribution of the multi-type branching process $\MBP(t\kappa,\mu)$. Recall that $\bT_{\cS}$ denotes the space of all rooted finite trees where each vertex has a type in $\cS$.
For  $\mvl \in \fT\setminus\set{\mv0}$,  let $\bT_{\cS}({\mvl}) \subseteq \bT_{\cS}$ be the set of trees containing $l_i$ many vertices of type $i$ for each $i \in [K]$.

\begin{prop}[{\cite{bollobas2007phase}, \cite{remco-rgcn2}*{Chapter 3.5}.}]
\label{prop:ode-mbp-irg}
Let $\mvpi(t) = (\pi({\mvl}, t):\mvl\in \fT)$ be the unique solution to the differential equations in \eqref{eqn:ode-irg}.
For any $\mvl \in \fT\setminus \{\mv0\}$ and $t\geq 0$, we have 
\begin{align}\label{eqn:mbp-ode-relation}
    \pi({\mvl}, t) = \pi({\mvl}, t; \mu, \kappa) &:= \frac{\pr\left(\MBP_{\mu}(t\kappa,\mu) \in \bT_{\cS}(\mvl)\right)}{\norm{\mvl}}.
\end{align} Furthermore, 
$$ \pr\left(\MBP_{\mu}(t\kappa,\mu) \in \bT_{\cS}(\mvl)\right) =\sum_{x \in [K]} \hat{\mu}_t(x) \pr(\MBP_{x}(t\kappa,\hat{\mu}_t) \in \bT_{\cS}(\mvl)) = \pr(\MBP_{\hat{\mu}_t}(t\kappa,\hat{\mu}_t) \in \bT_{\cS}(\mvl)),$$
where $\hat\mu_t$ is the dual measure associated with $\MBP(t\kappa,\mu)$.
\end{prop} 
We remark that the dual measure $\hat{\mu}_t$ is not necessarily a probability measure and in \eqref{eqn:mbp-ode-relation}, the expression $\pr(\MBP_{\hat{\mu}_t}(t\kappa,\hat{\mu}_t) \in \cdot)$ is simply defined to be  $\sum_{x \in [K]} \hat{\mu}_t(x) \pr(\MBP_{x}(t\kappa,\hat{\mu}_t) \in \cdot)$, thus the second equality in the the final statement of the proposition is a tautology. The limit objects $\pi({\mvl}, t)$ are functions of the driving parameters of the model namely type density $\mu$ and kernel $\kappa$. However this dependence will be suppressed from the notation.  
The main result of this Section is on the fluctuations of $\mvpi_n^{\fT_N}(\cdot)$ from $\mvpi^{\fT_N}(\cdot)$, where $\mvpi_n$ corresponds to the IRG model introduced below Assumption \ref{ass:irg}. Define, 
\begin{align}
    \label{eqn:xl-def}
    X_n({\mvl}, t) &= \sqrt{n}(\pi_n({\mvl}, t) - \pi({\mvl}, t)), \text{ for } \mvl \in \fT; \qquad \vX_n(t) = (X_n({\mvl}, t), \mvl \in \fT), \qquad t\geq 0. 
\end{align}
We first prove that for each fixed $N\geq 1$, the finite-dimensional restrictions  $\vX_n^{\fT_N}(t)$ converges in distribution to the solution of a linear stochastic differential equation, see \eqref{eqn:fdd-sde}. We need some notation in order to give a description of this SDE.  

Define the set of vectors $e_{\mvk} = (\ind\set{\mvl = \mvk}: \mvl \in \fT_N)$ for $\mvk \in \fT_N$.
Recall the measure $\nu$ from Assumption \ref{ass:irg}. Consider a probability space on which is given $\bR^K$ valued random variable $\Psi$ distributed as $\nu$ and a $|\fT_N|$-dimensional standard Brownian motion $\mvB^{\fT_N}$ independent of $\Psi$.
Then, writing $\Psi = (\Psi_k: k \in [K])$, under Assumption \ref{ass:irg},
\begin{align}\label{eqn:inital-point-x}
   X_n({\mvl}, 0) = \sqrt{n}(\pi_n({\mvl}, 0)-\pi({\mvl}, 0)) &\convd  X({\mvl}, 0) := \sum_{k=1}^K \Psi_k \ind\set{\mvl = \ve_k}.
\end{align}

Recall the matrix $\Lambda$ in Assumption \ref{ass:irg}(b) that described the second order fluctuations of the kernel $\kappa_n$ around its limit $\kappa$.  Also recall that  for $\mvl \in \fT_N$, $F_{\mvl}(\mvx,\tilde\kappa,\tilde\mu)$ depends on $\mvx$, only through $x_{\mvk}$ for $\mvk \in \fT_N$. Thus for such $\mvl$, abusing notation, we will write $F_{\mvl}(\mvx,\tilde\kappa,\tilde\mu)$ as $F_{\mvl}(\mvx^{\fT_N},\tilde\kappa,\tilde\mu)$. 
The following function will give the first part of the drift term of the SDE that governs the fluctuation behavior of $\vX_n^{\fT_N}$.
 \begin{align}\label{eqn:fdd-fclt-drift}
    a^{\fT_N}(t) = \left[\frac{1}{2}\theta_{\Lambda}(\mu,\mu) + \theta_{\kappa}(\Psi,\mu)\right] e_{\mv0} +  \sum_{\mvl \in \fT_N\setminus\set{\mv0}} \left[ F_{\mvl}(\mvpi^{\fT_N}(t),\Lambda,\mu) - \pi(\mvl,t) \theta_{\kappa}(\mvl,\Psi)\right]e_{\mvl}.
\end{align} 
 
We now describe the second term in the drift coefficient. For $t \ge 0$, define the $|\fT_N| \times |\fT_N|$ dimensional matrix $\Gamma^{\fT_N}(t)$ as\begin{align}\label{eqn:fdd-fclt-linear}
    \Gamma^{\fT_N}(t) = \sum_{\mvl \in \fT_N}\left[\sum_{\mvk_1+\mvk_2 = \mvl} \pi(\mvk_1,t) \theta_{\kappa}(\mvk_1,\mvk_2) e_{\mvl}e_{\mvk_2}^T\right] - \sum_{\mvl\in \fT_N}\theta_{\kappa}(\mvl,\mu)e_{\mvl}e_{\mvl}^T.
\end{align} Note that the matrix $\Gamma^{\fT_N}(t)$ is lower-triangular, namely
$e_{\mvk}^T \Gamma^{\fT_N}(t) e_{\mvl} = 0$ if $\mvk < \mvl$. 
The terms $(a^{\fT_N}(t), \Gamma^{\fT_N}(t))$ together will determine the drift coefficient (see \eqref{eqn:fdd-sde}).
 
Next for $\mvl,\mvk \ \in \fT_{N}\setminus \{\mv0\}$, define \begin{align}\label{eqn:Delta-N-defns}
    \Delta^{\fT_N}_{\mvl,\mvk} = e_{\mv0} + e_{\mvl+\mvk}\ind\set{\mvl+\mvk \in \fT_N} -e_{\mvl}-e_{\mvk},\qquad \Delta_{\mvl} = e_{\mv0} -e_{\mvl}.
\end{align} 
The diffusion coefficient in the SDE is given by the  $|\fT_N| \times |\fT_N|$ dimensional matrix $G^{\fT_N}(t)$, which is the symmetric square root of the following matrix 
\begin{align}\label{eqn:fdd-fclt-bm}
    \Phi^{\fT_N}(t) &= \frac{1}{2}\left[\theta_{\kappa}(\mu,\mu) - 2 \sum_{\mvl \in \fT_N} \pi(\mvl,t) \theta_{\kappa}(\mvl,\mu) + \sum_{\mvl,\mvk \in \fT_N} \pi(\mvl,t)\pi(\mvk,t)\theta_{\kappa}(\mvl,\mvk)\right] e_{\mv0}e_{\mv0}^T \\
    &\hspace{5mm}+ \sum_{\mvl \in \fT_N\setminus\set{\mv0}}\pi(\mvl,t)\left[\theta_{\kappa}(\mvl,\mu) - \sum_{\mvk\in \fT_N}\pi(\mvk,t)\theta_{\kappa}(\mvl,\mvk)\right] \Delta_{\mvl}\Delta_{\mvl}^T \nonumber \\
    &\hspace{7mm}+\frac{1}{2} \sum_{\mvl, \mvk \in \fT_N\setminus\set{\mv0}}  \pi(\mvl,t) \pi(\mvk,t) \theta_{\kappa}(\mvl,\mvk) \Delta^{\fT_N}_{\mvl,\mvk}(\Delta^{\fT_N}_{\mvl,\mvk})^T.\nonumber
\end{align} 
Observe that, by Theorem \ref{thm:odes-lln} and Fatou's Lemma, we have $\sum_{\mvl\in \fT_N}l_i \pi(\mvl,t) \leq \mu(i)$ for all $t\geq 0$. This observation, along with the bilinearity of the map $\theta_{\kappa}(\cdot,\cdot)$, guarantees that the terms in the first two sets of braces in the above display are nonnegative and so the  matrix $\Phi^{\fT_N}(t)$ is non-negative definite for all $t \geq 0$.
Consider the SDE
\begin{align}\label{eqn:fdd-sde}
     d\vV^{\fT_N}(t) = \left[a^{\fT_N}(t) +\Gamma^{\fT_N}(t) \vV^{\fT_N}(t)\right]dt + G^{\fT_N}(t)d\mvB^{\fT_N}, \; \vV^{\fT_N}(0) = X^{\fT_N}(0),
\end{align} 
where $X^{\fT_N}(0)$ is as introduced in \eqref{eqn:inital-point-x}.
 We remark here that the drift $a^{\fT_N}(t)$ in the equation is random since it depends on $\Psi$. However, $\Psi$  is independent of the Brownian motion $\mvB^{\fT_N}$ and so existence and uniqueness of a strong solution to the SDE follows from standard results.

 { For the rest of this section and the next two sections (Sections \ref{sec:res-micro-inf} and \ref{sec:res-macro}), {\bf Assumption \ref{ass:irg} will be taken to hold}. This assumption will not be noted explicitly in the statement of the various results.}

\begin{thm}[Finite Dimensional FCLT]\label{thm:fdd-fclt}

Fix $N\in \bN$. Then, $\vX^{\fT_N}_n(\cdot) \convd \vX^{\fT_N}(\cdot)$ in $\bD([0,\infty):\bR^{|\fT_N|})$ where $\vX^{\fT_N}$ is the unique solution to the SDE \eqref{eqn:fdd-sde}.

\end{thm}
We will next give formulas for the conditional means and variances of the limit process.
Recall that  the quantities $\pi({\mvl}, t)$ are functions of the driving parameters of the model, namely the type density $\mu \in \bR^K$ and kernel $\kappa \in \bR^{K\times K}$. Below, we will denote  the  gradients with respect to these parameters as $\nabla_\mu$ and $ \nabla_\kappa$ respectively. For vectors $v_1, v_2 \in \bR^K$ and $K\times K$ matrices $A_1, A_2$, we denote
$$v_1\cdot v_2 := \sum_{i \in [K]}v_1(i)v_2(i), \; A_1 \cdot A_2 := \sum_{i,j \in [K]}A_1(i,j)A_2(i,j).$$

\begin{prop}[Expectations, Variance and Covariance]\label{prop:mean-variance}
    Let $\vX^{\fT_N}(t)$ be the solution to the SDE in \eqref{eqn:fdd-sde}.  Then, 
    \begin{enumeratei}
        \item $m(\mvl,t) := \E(X^{\fT_N}(\mvl,t)|\Psi) = \nabla_{\mu}\pi(\mvl,t) \cdot \Psi + \nabla_{\kappa}\pi(\mvl,t) \cdot \Lambda$ for $\mvl \in \fT_{N}$.
        \item $\Sigma_{\mvl,\mvk}(t) := \E\left(\left(X^{\fT_N}(\mvl,t) - m(\mvl,t)\right)\left(X^{\fT_N}(\mvk,t) - m(\mvk,t)\right)| \Psi\right)$ is given by \begin{align}\label{eqn:variance-covariance-matrix-irg}
            \Sigma_{\mv0,\mv0}(t) = \frac{1}{2}\theta_{\kappa}(\mu,\mu), \qquad \Sigma_{\mv0,\mvl}(t) = \pi(\mvl,t)\left(\norm{\mvl} -1- t\theta_{\kappa}(\mvl,\mu)\right), \qquad \mvl \neq \mv0\\
             \Sigma_{\mvl,\mvk}(t) = \delta_{\mvl,\mvk} \pi(\mvl,t) + \pi(\mvl,t) \pi(\mvk,t)\left[t\theta_{\kappa}(\mvl,\mvk) - \left(\sum\limits_{i=1}^K\frac{l_ik_i}{\mu(i)}\right) \right], \qquad \mvl,\mvk \neq \mv0.
        \end{align}
        \ab{In particular, 
        $$\E\left(\left(X^{\fT_N}(\mvl,t) - m(\mvl,t)\right)\left(X^{\fT_N}(\mvk,t) - m(\mvk,t)\right)| \Psi\right) = \E\left(\left(X^{\fT_N}(\mvl,t) - m(\mvl,t)\right)\left(X^{\fT_N}(\mvk,t) - m(\mvk,t)\right)\right).$$}
        \item For fixed $t\ge 0$, the random variable $ \vX^{\fT_N}(t)$ conditional on $\Psi$ has a multivariate Normal distribution with mean vector given by (i) and covariances given by (ii).
    \end{enumeratei}
\end{prop}

\subsection{Infinite dimensional FCLTs for microscopic component densities}
\label{sec:res-micro-inf}
In order to understand the fluctuation properties of the giant component in the supercritical regime, it turns out that a FCLT for $X_n^{\fT_{N}}$ for each fixed $N$ is not enough, and   in fact one needs distributional convergence of $X_n^{\fT_{M\log n}}(\cdot) = \left( \sqrt{n}\left(\pi(\mvl,\cdot) - \pi(\mvl,\cdot)\right): \mvl \in \fT_{M \log n}\right)$, viewed as a sequence of processes with values in a suitable Banach space, for \emph{every} $M > 0$. 

Roughly speaking, the reason why it suffices to consider only type vectors of sizes $O(log n)$ is as follows.    Theorem \ref{thm:irg-boll} says that, for any fixed $t > t_c$, the size of the second largest component in $\cG_n(t)$ is of order $M\log n$ whp for an appropriate $M < \infty$. Consequently the size of the largest component, 
\[L_n(t) = n\left[1 - \sum_{\mvl \in \fT_{M\log n }} \norm{\mvl} \pi_n(\mvl,t)\right], \qquad \text{whp.}\]
This suggests that in order to capture the fluctuations of $L_n(t)$ it suffices to understand 
those of $\pi_n(\mvl)$ for $\mvl \in \fT_{M\log n }$ for a suitable choice of $M$.

Following \cite{pittel1990tree}, we now introduce a suitable Banach space that will be used to describe the  convergence of $X_n^{\fT_{M\log n}}$.  Fix $M,\delta \geq 0$.
Then, for each $t\ge 0$, $X_n^{\fT_{M\log n}}(t)$ will be viewed as an infinite random vector (which has zero values for all but finitely many coordinates) that take values in the separable Banach space $\ell_{1,\delta}$ defined as 
\begin{align}\label{eqn:l_1,delta-space}
    \ell_{1,\delta} = \set{\vz = (z_{\mvl}):\mvl \in \fT, \norm{z}_{1,\delta} < \infty } \text{ where } \norm{z}_{1,\delta} = |z_{\mv0}| + \sum_{\mvl \in \fT\setminus\set{\mv0}}\norm{\mvl}^\delta |z_{\mvl}|.
\end{align} Due to the phase transition at $t_c$, we will need to study the asymptotics of $X_n^{\fT_{M\log n}}$ in the \emph{sub-critical} regime and \emph{super-critical} regime separately. In both regimes, we show, in Theorem \ref{thm:fclt-irg}, that $X_n^{\fT_{M\log n}}$ converges in distribution to an $\ell_{1,\delta}$ valued process, which can be characterized as  the unique solution of an infinite dimensional linear SDE which is the analogue of the finite dimensional SDE in \eqref{eqn:fdd-sde}. 
We now introduce the drift and diffusion coefficients of this SDE.
Abusing notation, we will  once more denote by $e_{\mvk}$  the infinite vector  $(\ind\set{\mvl = \mvk}: \mvl \in \fT)$, for $\mvk \in \fT$.
 
The first part of the drift term, similar to the one in \eqref{eqn:fdd-fclt-drift}, is defined as \begin{align}\label{eqn:fclt-drift}
    a(t) = \left[\frac{1}{2}\theta_{\Lambda}(\mu,\mu) + \theta_{\kappa}(\Psi,\mu)\right] e_{\mv0} + \sum_{\mvl \in \fT\setminus\set{\mv0}} \left[ F_{\mvl}(\mvpi(t),\Lambda,\mu) - \pi(\mvl,t)\theta_{\kappa}(\mvl,\Psi)\right]e_{\mvl}.
\end{align} 
It can be verified that for each $t\neq t_c$, $a(t)$ is an element of $\ell_{1,\delta}$ for every $\delta \ge 0$. This is a direct consequence of exponential decay of the deterministic functions $\pi(\mvl,t)$ with $\mvl $ for $t\neq t_c$ as noted in Lemma \ref{lem:all-moments-finite}.
To define the second term in the drift, define the vector space 
\begin{align}
    \tilde{\ell}_1 = \bigcap_{\delta \geq 0}\ell_{1,\delta}. \label{eqn:l1-subspace}
\end{align}
 For every $t\neq t_c$ and $\delta \geq 0$, we define the linear operator $\Gamma(t):  \tilde{\ell}_1 \to \tilde{\ell}_1$ as \begin{align}\label{eqn:sde-gamma-defn}
    \left(\Gamma(t)(\mvx)\right)_{\mvl} = \left[\sum_{\mvk_1+\mvk_2 = \mvl} \pi(\mvk_1,t) x_{\mvk_2} \theta_{\kappa}(\mvk_1,\mvk_2) \right] - x_{\mvl}\theta_{\kappa}(\mvl,\mu), \text{ for } \mvl \in \fT.
\end{align}

Note that $\left(\Gamma(t)(\mvx)\right)_{\mv0} = 0$ for all $\mvx \in \tilde\ell_1$. Also, observe that the operator $\Gamma(t)$ is lower-triangular, i.e. $(\Gamma(t)(e_{\mvl}))_{\mvk}=0$ for $\mvk<\mvl$. The proof of the fact that the operator $\Gamma(t)$ maps elements in $\tilde{\ell}_1$ to itself is a consequence of Lemma \ref{lem:all-moments-finite}, and will be proved in Lemma \ref{lem:operator-properties}. 

The key idea in the proof of Lemma \ref{lem:all-moments-finite} is the observation in Lemma \ref{lem:mgf-mbp} that  $\pi(k,t)$ decays exponentially in $k$ for $t \neq t_c$. One does not expect such an exponential decay to hold for $t=t_c$. 
Indeed, in the case of \erdos (i.e $K = 1$ and $\kappa(1,1) = 1$), the  function $\pi(k,t_c)$ decays like $k^{-5/2}$ with $k$ (see \cite{pittel1990tree}). Similarly, for the Bohman-Frieze model, the  functions $\pi(k,t)$ for $t$ near the critical time $t_c$ exhibit $k^{-5/2}$ decay (see \cite{kang-spencer}). One expects analogous behavior for a general  $\IRG$. This is the reason that $\Gamma(t)$ is only defined for $t \neq t_c$.

Consider the Hilbert space
\begin{align}\label{eqn:l_2,delta-space}
    \ell_{2} = \set{\vz = (z_{\mvl}):\mvl \in \fT, \norm{z}_{2} < \infty } \text{ where } 
     \norm{z}_{2}^2 = \langle z, z\rangle \mbox{ and }
    \langle z, z'\rangle =   \sum_{\mvl \in \fT}  z_{\mvl} z'_{\mvl}.
\end{align}
and define the following nonnegative symmetric trace class operator (see Lemma \ref{lem:operator-properties}) on  $\ell_{2}$ as
\begin{align}\label{eqn:sde-phi-defn}
    \Phi(t) &=  \frac{1}{2}\left[\theta_{\kappa}(\mu,\mu) - 2 \sum_{\mvl \in \fT\setminus\set{\mv0}} \pi(\mvl,t) \theta_{\kappa}(\mvl,\mu) + \sum_{\mvl,\mvk \in \fT\setminus\set{\mv0}} \pi(\mvl,t)\pi(\mvk,t)\theta_{\kappa}(\mvl,\mvk)\right] e_{\mv0}e_{\mv0}^T\nonumber\\
    &\hspace{5mm}+ \sum_{\mvl \in \fT\setminus\set{\mv0}}\pi(\mvl,t)\left[\theta_{\kappa}(\mvl,\mu) - \sum_{\mvk\in \fT\setminus\set{\mv0}}\pi(\mvk,t)\theta_{\kappa}(\mvl,\mvk)\right] \Delta_{\mvl}\Delta_{\mvl}^T \nonumber\\
    &\hspace{5mm}+\frac{1}{2} \left[\sum_{\mvl,\mvk \in \fT\setminus\set{\mv0}}  \pi(\mvl,t) \pi(\mvk,t) \theta_{\kappa}(\mvl,\mvk) \Delta_{\mvl,\mvk}\Delta_{\mvl,\mvk}^T\right] 
\end{align} 
where $\Delta_{\mvl,\mvk} = e_{\mv0}+e_{\mvl+\mvk}-e_{\mvl}-e_{\mvk}$ and $\Delta_{\mvl} = e_{\mv0}-e_{\mvl}$ and the above infinite sums converge in the strong operator topology in $\ell_2$. Specifically, denoting the first infinite sum above by $\Phi_1(t)$, for $x \in \ell_2$, $\Phi_1(t)x$ is defined to be
\begin{equation}\label{eq:ssse} \sum_{\mvl \in \fT\setminus\set{\mv0}}\pi(\mvl,t)\left[\theta_{\kappa}(\mvl,\mu) - \sum_{\mvk\in \fT\setminus\set{\mv0}}\pi(\mvk,t)\theta_{\kappa}(\mvl,\mvk)\right] \langle x, \Delta_{\mvl}\rangle\Delta_{\mvl},
\end{equation}
where the above sum converges in $\ell_2$. The second infinite sum in \eqref{eqn:sde-phi-defn} is defined similarly. See Lemma \ref{lem:operator-properties}.

Let $G(t)$ be the symmetric square root of $\Phi(t)$. It can be verified that if  $\cI$ is a compact interval in $\bR_+$ with $t_c \not\in \cI$, then $\sup_{t \in \cI} \|G(t)\|_{\mbox{\tiny{HS}}} <\infty$, where $\|\cdot\|_{\mbox{\tiny{HS}}}$ denotes the Hilbert-Schmidt norm (see Lemma \ref{lem:operator-properties}).

Next, let $\set{B_{\mvl}(t): \mvl \in \fT}$ be \emph{i.i.d.} copies of a standard Brownian motion, and define $\mvB(\cdot) = (B_{\mvl}(\cdot):\mvl \in \fT)$. Then, it follows that, the stochastic integral $\int_{\cI}G(t)d\mvB(t)$ is well defined as a $\ell_2$-valued random variable, and 
\begin{align*}
     \left(\int_{\cI}G(t)d\mvB(t)\right)_{\mvl} = \sum_{\mvk \in \fT} \int_{\cI} \la e_{\mvl}, G(t) e_{\mvk}\ra dB_{\mvk}(t), \qquad 
\mvl \in \fT.
\end{align*} 
where the sum on the right converges a.s. and in $\cL^2$.
We will in fact show in Lemma \ref{lem:operator-properties} that $\int_{\cI}G(t)d\mvB(t)$ takes values in $\ell_{1,\delta}$ a.s. for every $\delta>0$.

We assume that on the  probability space where  $\mvB(\cdot)$ is defined, we are also  given a $\bR^K$ valued random variable $\Psi$ distributed as $\nu$ that is independent of $\mvB(\cdot)$.

The SDE that describes the limit of $X_n^{\fT_{M\log n}}$ is given as,  
\begin{align}\label{eqn:sde-vector}
    d\vV(t) = \left[ a(t) + \Gamma(t)\vV(t)\right]dt + G(t)d\mvB(t).
\end{align} 
The SDE is interpreted in the usual integral sense where the integral $\int_0^t (a(s) + \Gamma(s)\vV(s)) ds$ is a Bochner integral in the Banach space $\ell_{1,\delta}$ while the integral $\int_0^t G(s) d\mvB(s)$ is the stochastic integral discussed above.
The precise result is as follows. Recall the critical point $t_c$ introduced in Theorem \ref{thm:irg-boll}, associated with the irreducible pair $(\mu, \kappa)$ of Assumption \ref{ass:irg}.

\begin{thm}[Infinite dimensional FCLT for type densities of microscopic component sizes]
    \label{thm:fclt-irg}

Fix $\delta, M\geq 0$.
\begin{enumeratea}
    \item \textbf{Sub-Critical Regime:} For $T < t_c$, we have $ \vX_n^{\fT_{M\log n}}(\cdot)\convd \vX(\cdot)$ in $\bD\left([0,T]:  \ell_{1,\delta}\right)$ where $\vX(t) \in \tilde{\ell}_1$ for all $t \leq T$, and $\vX(\cdot)$ is the unique solution to the SDE \eqref{eqn:sde-vector} with $\vX(0)$ as defined in \eqref{eqn:inital-point-x}. 
    \item \textbf{Super-Critical Regime:} Fix $t_c < T_1 < T_2$.
    Assume, without loss of generality, that on the probability space on which  $\mvB$ and $\Psi$ are given, we have a $\tilde{\ell}_1$ valued random variable $\vX(T_1)$, which is independent of $\{\mvB(t+T_1) - \mvB(T_1), 0 \le t \le T_2-T_1\}$ and is such that, 
    conditionally on $\Psi$, $\{X(\mvl, T_1), \mvl \in \fT\}$ is a Gaussian family with mean $m(T_1) = \left(m(\mvl,T_1):\mvl \in \fT\right)$ and covariances $\Sigma(T_1) = (\Sigma_{\mvl,\mvk}(T_1))_{\mvl,\mvk \in \fT}$, defined in Proposition \ref{prop:mean-variance}. Then $ \vX_n^{\fT_{M\log n}}(\cdot)\convd \vX(\cdot)$ in $\bD([T_1,T_2]: \ell_{1,\delta})$ where $\vX(t) \in \tilde{\ell}_1$ for all $T_1 \leq t \leq T_2$ and is the unique solution to the SDE \eqref{eqn:sde-vector} on $[T_1, T_2]$ with initial value $\vX(T_1)$, namely, for all $t \in [T_1, T_2]$,
     $$\vX(t) = \vX(T_1) + \int_{T_1}^t \left[ a(s) + \Gamma(s)\vX(s)\right]ds + 
     \int_{T_1}^t G(s)d\mvB(s).$$
\end{enumeratea}
\end{thm}

\subsection{Functional central limit theorems for macroscopic functionals in the supercritical regime}
\label{sec:res-macro}

For $t\ge 0$, let $N_n(t)$ be the number of connected components in $\cG_n(t)$.
Fix $t> t_c$ and let $\cC_{\sss(1),n}(t) = (\cV(\cC_{\sss(1),n}(t)), \cE(\cC_{\sss(1),n}(t)) )$ be the giant component with its corresponding vertex and edge set. Let
\[L_n(t) = |\cV(\cC_{\sss(1),n}(t)|, \qquad S_n(t) =|\cE(\cC_{\sss(1),n}(t))| - (L_n(t) -1), \]
 be the {\em size} (number of vertices) of the giant and the {\em surplus}, namely the excess number of edges in the giant over a tree of the same size,  respectively. 
 Recall the deterministic function $\mvpi(t)= (\pi({\mvl}, t):\mvl\in \fT)$ from the previous section, introduced below Remark \ref{rem:rem1}, and define,  
 \begin{equation}\label{eq:837}
 \eta(t) = \sum_{\mvl  \in \fT\setminus\{ \mv0\}} \pi(\mvl,t), \qquad l(t) = 1-\sum_{\mvl\in \fT\setminus\set{\mv0}} \norm{\mvl} \pi(\mvl,t), \qquad  s(t) = \sum_{\mvl \in \fT} \pi(\mvl,t) -1.\end{equation}
\begin{thm}
\label{thm:fclt-giant-surplus}
Suppose Assumption \ref{ass:irg} holds.
Fix  $t_c < T_1 < T_2$. 
Then as $n\to\infty$, 
\begin{equation}\label{eq:843}\left(\sqrt{n} \left[\frac{1}{n}\begin{pmatrix}
        N_n(t)\\
        L_n(t)\\
        S_n(t)
    \end{pmatrix} - \begin{pmatrix}
        \eta(t)\\
        l(t)\\
        s(t)
    \end{pmatrix}\right] \right)_{T_1\leq t \leq T_2} \convd \begin{pmatrix}
        \sum_{\mvl  \in \fT\setminus \{\mv0\}} X(\mvl,t) \\
        -\sum_{\mvl\in \fT\setminus\set{\mv0}} \norm{\mvl}X(\mvl,t)\\
        \sum_{\mvl \in \fT} X(\mvl,t)
    \end{pmatrix}_{T_1\leq t \leq T_2},\end{equation}
    in $D([T_1, T_2]: \bR^3)$ where, $\vX(t) = (X(\mvl,t) : \mvl \in \fT)$ is the limit process in Theorem \ref{thm:fclt-irg}(b). 
    
    Furthermore, for any $T<t_c$,
\begin{equation}\label{eq:843bn}\left(\sqrt{n} \left[\frac{1}{n}
       \left( N_n(t) -\eta(t)\right)
    \right] \right)_{0\leq t \leq T} \convd 
        \left(\sum_{\mvl  \in \fT\setminus \{\mv0\}} X(\mvl,t)\right)_{0\leq t \leq T},\end{equation}
    in $D([0, T]: \bR)$ where, $\vX(t) = (X(\mvl,t) : \mvl \in \fT)$ is the limit process in Theorem \ref{thm:fclt-irg}(a).
    
\end{thm}
\begin{rem}
    There are various other formulations of inhomogeneous random graph models that are asymptotically equivalent to the one considered in the current work (cf.  \cite{janson2010asymptotic}), e.g., the model obtained by replacing the connection probability $(1- \exp(-t K(x_i, x_j)/n))$ by $\min(tK(x_i, x_j)/n,1)$. It is easily verified that the the above result also establishes central limit theorems  at any fixed time $t> t_c$ for such model variants. We believe that, with some additional work, one can show that the same functional central limit theorems over fixed time intervals $[T_1, T_2]$  hold for such model variants, however we do not pursue this direction here.

\end{rem}

\begin{rem}
   In the setting of the \erdos random graph model Pittel\cite{pittel1990tree} shows  that the mean and variance for the limit random variable in the central limit theorem for the giant, at fixed time $t> t_c$, has a  nice interpretation in terms of survival probabilities and the mean of the associated dual branching process. 
   See also Section \ref{sec:er} below for other similar characterizations for the \erdos random graph model. After this paper was posted on arxiv, a paper by Clancy \cite{clancy2025centrallimittheoremgiant} has given nice formulae for the asymptotic mean and variance and a different proof for a fixed time CLT for the giant component in the IRG. 
\end{rem}

 One can also obtain functional central limit theorems for other observables of the giant. One example is as follows. For $t_c < T_1 < T_2$, let $L_{n}(i,t)$ be the number of vertices of type $i \in [K]$ in the largest component and let $\ell_n(i,t) =L_n(i,t)/n $. Also, let $\ell(i,t) = \mu(i) -\sum_{\mvl \in \fT} l_i \pi(\mvl,t)$, $i \in [K]$. Consider the $D([T_1, T_2]: \bR^K)$ valued process, 
\[\fL_n(t):= \sqrt{n}\left(\ell_n(i,t) - \ell(i,t): i\in [K]\right), \qquad T_1\leq t \leq T_2. \]

\begin{thm}[FCLT for density of types in the giant]
\label{thm:359}
Suppose Assumption \ref{ass:irg} holds.
As $n\to\infty$, 
$$ \left(\fL_n(t)\right)_{T_1\leq t\leq T_2} \convd \left(\Psi(i) - \sum_{\mvl\in \fT} l_i X(\mvl,t): i\in [K]\right)_{T_1\leq t\leq T_2},
$$
in $D([T_1, T_2]: \bR^K)$ where, $\vX(t) = (X(\mvl,t) : \mvl \in \fT)$ is the limit process in Theorem \ref{thm:fclt-irg}(b) and $\Psi$ is as introduced above Theorem \ref{thm:fclt-irg}.
\end{thm}
We in fact have joint weak convergence of $\fL_n(\cdot)$ and the processes on the left side of \eqref{eq:843}.

\subsubsection{Special case: The \erdos random graph}
\label{sec:er}
In this section we note some consequences of the above result in the setting of the \erdos random graph where many of the limit constants can be explicitly computed. Most of the results in this section  are known,  see Remark \ref{rem:know}. 

The dynamic version of \erdos random graph is a special case of $\IRG$ with $K=1$, $\mu_n(1) =\mu(1) = 1$ and $\kappa_n(1,1) = \kappa(1,1) = 1$. In this case, the critical time $t_c =1$. Let $\rho_t \in (0,1)$ be the positive solution of $\rho_t = 1-\exp(-t\rho_t)$. Then $\rho_t$ is the survival probability of the Galton-Watson Branching process with progeny distribution $\mathrm{Poisson}(t)$. Define, 

 \begin{align}\label{eqn:ER-surplus-giant-defns}
    \eta(t) &= (1-\rho_t)\left(1-\frac{t(1-\rho_t)}{2}\right), \qquad s(t) = (t-1)\rho_t - \frac{t\rho_t^2}{2}, \qquad l(t) = \rho_t.
\end{align} Let $\tilde{\Sigma}(t)$ be the $3 \times 3$ symmetric matrix with 
\begin{align}\label{eqn:er-surplus-giant-variance}
    \tilde{\Sigma}_{1,1}(t) &= (1-\rho_t)\left[\rho_t + \frac{(1-\rho_t)t}{2}\right], \   \tilde{\Sigma}_{1,2}(t) = \frac{-\rho_t(1-\rho_t)}{1-t(1-\rho_t)}, \ \tilde{\Sigma}_{1,3}(t) = -(t-1) \rho_t (1-\rho_t) \nonumber\\
    \tilde{\Sigma}_{2,2}(t) &= \frac{\rho_t(1-\rho_t)}{\left[1-t(1-\rho_t)\right]^2}, \  \tilde{\Sigma}_{2,3}(t) = \frac{\rho_t(1-\rho_t)(t-1)}{1-t(1-\rho_t)},\ \tilde{\Sigma}_{3,3}(t) = \rho_t(1-\rho_t) + t\rho_t\left(\frac{3\rho_t}{2} -1\right).
\end{align}

Let $N_n(t), L_n(t), S_n(t)$ be as in the previous section, but now corresponding to the setting of the \erdos graph considered here.
As a special case of Theorem \ref{thm:fclt-giant-surplus} we obtain the following joint FCLT for these three functionals in the super-critical regime of \erdos random graph and a FCLT for the first functional in the subcritical regime.
\begin{thm}[FCLT for \erdos random graph]\label{thm:joint-fclt-erdos} 
The statements in Theorem \ref{thm:fclt-giant-surplus} hold for the \erdos graph where $t_c=1$ and the coefficients $a, \Gamma$ and $G$ in the definition of the limit process $\vX$ are defined by setting $\Psi=0$, $\Lambda=0$,  $K=1$,  $\mu(1)=1$, and $\kappa(1,1)=1$.
Furthermore, for  $t>1$, denoting by $V(t)$ the $\bR^3$ valued random variable in the parenthesis on the right side of \eqref{eq:843},
we have that
 $V(t)$ is a Gaussian process  with $V(t) \sim N(0,\tilde{\Sigma}(t))$ for all $t>1$.
\end{thm}

\begin{rem}
\label{rem:know}
The FCLT  for the size of the giant component (i.e. $L_n(\cdot)$) in \erdos random graphs was recently established in \cite{enriquez2024processfluctuationsgiantcomponent}, and the result concerning the number of components in the sub-critical regime was shown in \cite{corujo2024numberconnectedcomponentssubcritical}. Both of these results follow as  consequence of Theorem \ref{thm:joint-fclt-erdos}.  More recently, for the \erdos random graph, using the ``simultaneous breadth first exploration algorithm" constructed in \cite{limic2019eternal}, a different proof for the FCLT for the size of the giant component in the \erdos random graph was established in \cite{corujo2024novel} allowing one to even establish central limit theorems for the giant component in the so-called barely supercritical regime. Personal communication from the authors indicate that in work in progress by Corujo, Enriquez, Faraud, Lemaire and  Limic, for the \erdos random graph, the authors are able to recover the above joint distributional result in this setting using different (neighborhood exploration and associated random walk) techniques.
\end{rem}

\begin{rem}
    The fixed time result { for the joint distributional convergence for the number of components and size of largest component in }the \erdos random graph goes back to the foundational work of Pittel \cite{pittel1990tree}.
\end{rem}

\subsection{Graphon modulated dense graphs}
\label{sec:def-graphon-mst}
We now describe how the above results in the sparse regime also allow one to derive asymptotics for various functionals in the dense regime, as a consequence of the close connection between the world of graphons (dense graph limits) and inhomogeneous random graphs\cite{bollobas2010percolation}. Consider a kernel $\kappa_n$ on $[K]$ with entries $\kappa_n(i,j) \in (0,1)$ for $i,j \in [K]$ and let $\mu$ be a probability measure on $[K]$. to simplify presentation, we allow the kernel to depend on $n$ and keep $\mu$ to be fixed; the results below can  be extended to the more general case. 

The parameters $\kappa_n$ and $\mu$ are used to construct a finite dimensional graphon as follows: 
\begin{enumeratea}
    \item Partition the unit interval $[0,1]$ into the sub-intervals $\cI_j = (\sum_{i=1}^{j-1} \mu(i), ~  \sum_{i=1}^{j} \mu(i)] $ for $1\leq i\leq K$. Include $0$ with $\cI_1$.
    \item Define the symmetric measureable function $\thickbar{\kappa}_n:[0,1]\times [0,1] \to [0,1]$ via $\thickbar{\kappa}_n(u,v) = \kappa_n(i,j)$ for $u\in \cI_i, v\in \cI_j$. The function $\thkappa_n$ is called the graphon associated with $\kappa_n$ and $\mu$.
\end{enumeratea}
Let $\cG_n^{\dense}(\thickbar{\kappa}_n,\mu)$ be the random graph (sometimes referred to as the $\thickbar{\kappa}_n$-random graph \cite{lovasz-book}) on vertex set $[n]$ generated as follows: first generate $\vU_n:= \set{U_i:1\leq i\leq n}$ i.i.d. $U[0,1]$ random variables. Next connect $i,j \in [n]$, $i\neq j$, with probability $\thickbar{\kappa}_n(U_i, U_j)$, independent across edges.

\subsubsection{{\bf Percolation on dense graphs}}
\label{sec:podg}
With  $\kappa_n$ as above, suppose that $\kappa_n \to \kappa$ pointwise and that $(\mu, \kappa)$ is an irreducible pair.  
  Fix any $t\ge 0$ and write $\cG_n^{\dense}[\thkappa_n,\mu; t/n]$ for the percolated graph which is defined as the graph obtained by retaining each edge in $\cG_n^{\dense}(\thkappa_n,\mu)$, independently, with probability $t/n$. 
 The graph process $\{\cG_n^{\dense}[\thkappa_n,\mu; t/n], t \ge 0\}$ is closely related to the IRG process introduced in Definition \ref{def:irg}. Specifically, let $\{x_i, i \in \bN\}$ be an i.i.d. sequence distributed as $\mu$ and let $\mvx^n = \set{x_i \in \cS:i\in [n]}$. Then, for a fixed $t\ge 0$, the  only difference between the graphs 
 $\cG_n(t,\mvx^n,\kappa_n)$ and $\cG_n^{\dense}[\thkappa_n,\mu; t/n]$,
 is that, independently across all $i,j \in [n]$, $i\neq j$, in the former an edge between vertices $i$ and $j$  occurs with probability $\sum_{x,y \in [K]}(1- e^{-\kappa_n(x,y)t/n})\mu(x)\mu(y)$ while in the latter it occurs with probability $ t/n \ \sum_{x,y \in [K]}\kappa_n(x,y)\mu(x)\mu(y)$.
 For the IRG model given by  
 $\{\cG_n(t,\mvx^n,\kappa_n)\}$ note that parts (a) and (b) of Assumption \ref{ass:irg} are automatically satisfied, where $\nu$ is a Normal distribution with mean $\mvzero$ and $K\times K$ covariance matrix $\Sigma$
with $\Sigma_{i,j} = \delta_{i,j} \mu(i) - \mu(i)\mu(j)$, $i, j \in [K]$.
Assume in addition that part (c) of Assumption \ref{ass:irg} also holds. Then it follows that
the results of Sections \ref{sec:res-micro}, \ref{sec:res-micro-inf}, and \ref{sec:res-macro}
(specifically, Theorems \ref{thm:odes-lln}, \ref{thm:fdd-fclt}, \ref{thm:fclt-irg}, \ref{thm:fclt-giant-surplus}, \ref{thm:359}) hold with $\{\cG_n(t,\mvx^n,\kappa_n), t \ge 0\}$ 
and with $\Psi$ a Normal random variable with mean $0$ and covariance matrix $\Sigma$.
From these results and using the observations in  \cite{janson2010asymptotic}*{Example 3.1}, for fixed $t > t_c$, one can easily obtain as a corollary, a joint central limit theorem for the size of giant, its surplus, vector of number of type of vertices in it, and number of components in $\cG_n^{\dense}[\thkappa_n,\mu; t/n]$.

\begin{rem}
The general case of percolation on a dense graph sequence $\set{\cG_n:n\geq 1}$ converging to a Graphon \emph{in the cut norm}, and the corresponding formulation of the phase transition for the emergence of a giant component, has been studied in \cite{bollobas2010percolation}. It would be interesting to understand, what additional assumptions one needs in this general setting, both on the limit graphon, as well as on the rates of convergence, for central limit theorems for the giant component to  hold in this setting. The ``$\bL^3$ -Graphon'' conditions introduced in \cite{baslingker2023scaling} to understand the \emph{critical regime} of such models  could be a good starting point for this research direction. 
\end{rem}

\subsubsection{{\bf Fluctuations results for the weight of the MST on dense random graphs}}
\label{sec:res-mst-weight}
We now introduce the final object of interest. 

\begin{defn}[Random weight MST on dense random graphs]\label{def:mst-dense}
Let $\mu, \kappa_n, \thkappa_n$ be as introduced above.
    For each edge $\ve$ in
    $\cG_n^{\dense}(\thickbar{\kappa}_n,\mu)$, let $w_{\ve}\sim Exp(1)$, independent across edges. We refer to $w_{\ve}$ as the weight of the edge $\ve$. Let $\cW(\cG_n^{\dense}(\thickbar{\kappa}_n,\mu))$ be the {\bf weight} of the corresponding minimal spanning tree on $\cG_n^{\dense}(\thickbar{\kappa}_n,\mu)$, namely the sum of the weights of edges in a spanning tree for $\cG_n^{\dense}(\thickbar{\kappa}_n,\mu)$, for which this sum is minimal.   We use the convention that $\cW(\cG_n^{\dense}(\thickbar{\kappa}_n,\mu)) = \infty$ if $\cG_n^{\dense}(\thickbar{\kappa}_n,\mu)$ is not connected. 
\end{defn}

\begin{rem}
    One may consider other distributions than Exponential distributions for weights, e.g. $U[0,1]$. We believe that for the next result, the exact distribution of the weights should not be important and one only needs a suitable control over the density of the weight distribution at zero (see e.g.  Frieze \cite{frieze1985value} and Steele \cite{steele1987frieze} for a similar observation for the random weight MST problem on the complete graph).  However for simplicity  we choose Exponential weights since in this case the weight of the MST can be directly analyzed using our results on the functional CLT for the number of components for sparse IRG presented in Theorem \ref{thm:fclt-giant-surplus}.
\end{rem}

For $(\mu, \kappa_n, \kappa)$ as in Section \ref{sec:podg}, suppose that Assumption \ref{ass:irg} (c) is satisfied.
 For any $\mvl,\mvk\in \fT\setminus\set{\mv0}$, $s,t\geq 0$, and $N = \max(\norm{\mvl},\norm{\mvk})$, \ab{let $\vU^{\fT_N}$ be the mean $0$ Gaussian process given by the SDE
\begin{align}\label{eqn:fdd-sdeU}
     d\vU^{\fT_N}(t) =  \Gamma^{\fT_N}(t) \vU^{\fT_N}(t) dt + G^{\fT_N}(t)d\mvB^{\fT_N}(t), \; \vU^{\fT_N}(0) = 0.
\end{align} 
Let
 \begin{align}\label{eqn:two-point-covariance}
    \Sigma(\mvl,\mvk,s,t) = \E(\vU^{\fT_N}(\mvl,t) \vU^{\fT_N}(\mvl,s)),
\end{align}}
namely the covariations between two different time points.
Note that $\Sigma(\mvl,\mvk,t,t)$ is given by (ii) in Proposition \ref{prop:mean-variance}.
Now, define \begin{align}\label{eqn:mst-mean-variance}
 \sigma_\infty = \sum_{\mvk,\mvl \in \fT\setminus\set{\mv0}}\int_0^\infty\int_0^\infty \Sigma(\mvl,\mvk,s,t) ds dt.
\end{align} 
Well-definedness and finiteness of the above constant is proven in Section \ref{sec:proofs-dense-mst}. 

\begin{thm}[Weight of the MST on dense graphs]\label{thm:clt-mst-weight}
Let $(\mu, \kappa_n, \kappa)$ be as in Section \ref{sec:podg}. 
\abb{Consider $\cG^{\dense}_n(\thkappa_n,\mu)$ defined as in Section \ref{sec:def-graphon-mst} using i.i.d. $U[0,1]$ random variables $\vU_n:= \set{U_i:1\leq i\leq n}$.}
 Suppose   that   Assumption \ref{ass:irg} (c) is satisfied. Then there exist a sequence  $\set{\cK_n(\thkappa):n\geq 1}$, \ab{where $\cK_n(\thkappa)$ is a measurable function of \abb{$\vU_n$}}, such that
$$\sqrt{n}\left(\cW(\cG^{\dense}_n(\thkappa_n,\mu)) -\cK_n(\thkappa)\right) \convd \fN_\infty:= \cN(0,\sigma_{\infty}),$$
 where $\sigma_{\infty}$ is as defined in \eqref{eqn:mst-mean-variance}, and further, $\cK_n(\thkappa) \to \cK(\thkappa) $, \ab{in probability}, as $n\to\infty$, where \begin{align}\label{eqn:mst-limit-bp}
     \cK(\thkappa) = \sum_{k=1}^\infty \int_0^\infty  \frac{\pr\left(\abs{\MBP_\mu(t\kappa,\mu)} = k\right)}{k}dt.
 \end{align} 
\end{thm}
Recall the $\pi(\mvl,t)$ given as the solution of the ODE \eqref{eqn:ode-irg} in Proposition \ref{prop:ode-solutions}, with  $\pi({\mvl}, 0)$ as in \eqref{eqn:intial-point}. From Proposition \ref{prop:ode-mbp-irg} we see that the limit constant has the alternate formula $\cK(\thkappa)= \int_{0}^\infty \sum_{\mvl \in \fT}\pi(\mvl,t)dt$.  

\begin{rem}\label{rem:8}
  For the case where $\kappa_n \equiv \mvone$ for all $n$, namely $\cG^{\dense}_n(\thkappa_n,\mu)$ is a  complete graph, Frieze\cite{frieze1985value} shows that the limit constant $\cK(\mvone) = \zeta(3) = \sum_{k=1}^\infty 1/k^3$ while Janson \cite{janson1995minimal} proved that a central limit theorem as in the above theorem holds and in this case, $\cK_n(\thkappa)$ can be taken to be $\cK(\mvone) = \zeta(3)$ and 
  the limit random variable $\fN_\infty = N(0, \sigma^2)$, where,  
  \[\sigma^2 = \frac{\pi^4}{45} - 2 \sum_{i=0}^{\infty} \sum_{j=1}^{\infty} \sum_{k=1}^{\infty} \frac{(i+k-1)! k^k (i+j)^{i-2} j}{i! k! (i+j+k)^{i+k+2}}.\] 
  Theorem \ref{thm:clt-mst-weight}, extends these results to a more general setting of Graphon modulated dense graphs. 
   The law of large numbers for the weight of the MST are available on more general types of dense graph sequences than those considered in Theorem \ref{thm:clt-mst-weight}, see \cite{hladky2023random}.  These results in particular show that $\cK_n(\thkappa)$ converges as $n \to \infty$ and the limit constant $\cK(\thkappa)$ is finite (see \cite{hladky2023random}*{Fact 6}). 
\end{rem}

\begin{rem}
      In Theorem \ref{thm:clt-mst-weight}, as will be seen in the proof,  the centering $\cK_n(\thkappa)$ is taken to be equal to  \abb{$\E(\cW_n^{\fanz}\mid \vU_n)$} where $\cW_n^{\fanz}$ is the weight of the MST in a related model on the complete graph  with heterogeneous (i.e. type dependent) edge weights. We believe that the Theorem should be true with $\cK_n(\thkappa)$ replaced by the limit $\cK(\thkappa)$, however this would require showing that $\sqrt{n}|\cK_n(\thkappa)-\cK(\thkappa)| = o(1)$ which will require additional work.
\end{rem}

\section{Discussion}
\label{sec:disc}

The goal of this section is to provide general context to place the results in this paper. 
\subsection{Related work and other regimes}
The main goal of this paper is to study {\bf second order fluctuations} of various functionals of interest for inhomogeneous random graphs, both microscopic functionals for general $t> 0$ and macroscopic functionals for $t> t_c$.  The critical regime $t= t_c$ and more generally the  critical scaling window of the form $t = t_c + \lambda n^{-1/3}$ for fixed $\lambda \in \bR$ was addressed in \cite{baslingker2023scaling} where all maximal components are of the same order in the sense that the ordered list of normalized component sizes $\vC_n(t_c+\lambda n^{-1/3}) = n^{-2/3}(|\cC_{\sss(i),n}(t_c+\lambda n^{-1/3})|: i\geq 1) \convd \mvxi_{\IRG}(\lambda)$, 
where $|\cC_{\sss(i),n}(t)|$ is the size of the $i$-th largest component and $\mvxi_{\IRG}(\lambda)$ is an 
infinite dimensional random vector closely related to Aldous's standard multiplicative coalescent \cite{aldous1997brownian}. In particular this result says that the fluctuations of the maximal component in this regime are of order $n^{2/3} \gg \sqrt{n}$. This 
anomalous order of fluctuations is also manifested in the  blow up of the variances in the CLT of giant component as $t\downarrow t_c$ (see e.g. the formula for $\tilde \Sigma_{2,2}(t)$ in \eqref{eqn:er-surplus-giant-variance}). 

Closer to the current work is \cite{andreis2023large} which studies large deviation properties of both densities of microscopic components as well as large components at fixed time $t$ for (general type space) $\IRG$. Building on finite state space settings and deriving detailed connections to multitype branching processes, they are able to derive rate functions for both the microscopic and macroscopic functionals for fixed time $t$. Two directions would be worth investigating: \begin{inparaenuma}
    \item Understanding the connection to the dynamics used in the current paper;  in particular whether our dynamic formulation  can  be used to derive such, and perhaps stronger path space, large deviation results. The authors of \cite{andreis2023large} observe the connection between objects in their rate function and equations derived from certain `inhomogeneous coagulation models', however their proofs of the LDP hinge on fixed time calculations using careful analysis of approximating multitype branching processes.
    \item Understanding the connections between Theorem \ref{thm:fclt-irg} and the domain of ``mesoscopic portion''  of the mass conjectured in \cite{andreis2023large}*{Remark 1.5}.
\end{inparaenuma}

 Our approach to the CLT for the weight of the MST in Theorem \ref{thm:clt-mst-weight}  is inspired by Janson \cite{janson1995minimal} who studied the fluctuations of the  weight of the MST on the complete graph with i.i.d. edge weights. The key observation of \cite{janson1995minimal} was that for such a CLT one needs a suitable FCLT for the number of components $N_n(t)/n = \sum_{k=1}^{\infty} \pi_n(k,t)$ for $t\geq 0$ in the \erdos~random graph process. Here,  $\pi_n(k,t)$  denotes the density of components of size $k$ in an \erdos random graph with connection probability $t/n$. 
The paper \cite{janson1995minimal} proceeds by first establishing
 a finite dimensional FCLT of the form in Theorem \ref{thm:fdd-fclt}, namely for the truncated sums $\sum_{k=1}^K \pi_n(k,t) $ for ``large $K$''. This  coupled with explicit computations, to obtain uniform control over tail terms  $\sum_{K+1}^\infty \pi_n(k,t)$, which are  tractable for the \erdos random graph,  allowed him to analyze the  functional $N_n(\cdot)$ and consequently obtain a CLT for the weight of the MST on the complete graph. 

Our work considers a more general setting, however even in the setting of \cite{janson1995minimal}, this paper posed the   question of an infinite dimensional FCLT of the form in Theorem \ref{thm:fclt-irg}.   The current paper answers this question in the affirmative, in the more general setting of inhomogeneous random graphs and shows how one can use such FCLTs  to establish an FCLT for functionals analogous to $N_n(\cdot)$ (Theorem \ref{thm:fclt-giant-surplus}) and understand fluctuations of the weight of the MST for broad families of  dense graph sequences.

\subsection{Proof techniques in the supercritical regime}
\label{sec:proofidea}
At a high level,  two general techniques used in the literature for understanding second order fluctuations of the giant component in the supercritical regime are as follows 
\begin{inparaenuma}
    \item {\bf Exploration and walk based techniques:} In this approach, one explores components of a given random graph starting from a fixed vertex and typically  associates a ``random walk'' type process with the exploration procedure whose hitting times of specific subsets of $\bR$ describe sizes of component explorations; then using e.g. the martingale FCLT for such exploration walks,  one derives fluctuations around laws of large numbers for these hitting times which result in CLTs for the size of the largest component. Examples of this approach include \cites{bollobas2012asymptotic,Puhalskii} for the \erdos random graph and \cite{riordan2012phase} for the configuration model; \cite{clancy2025centrallimittheoremgiant} for fixed time CLT for the giant component in the IRG; \cites{enriquez2024processfluctuationsgiantcomponent,corujo2024novel} for FCLT type results for the \erdos random graph and \cite{aldous1997brownian} for understanding the entire critical regime;
    \item {\bf Analyzing microscopic components:} This is the approach taken in the current paper.
    In this approach, one firsts studies fluctuations of the microscopic components and then uses the associated central limit theorems to get results on the largest component.  This approach was initiated by Pittel for the \erdos random graph \cite{pittel1990tree} and was particularly influential for our work.
    
    In a different direction, fluctuation analysis for the microscopic components of the configuration model using Stein's method can be found in \cite{barbour2019central}. A closely related paper is \cite{seierstad2013normality} which also uses dynamics to provide conditions for asymptotic normality of the giant component and applies these techniques to random graph processes driven by degree based evolutions. 
\end{inparaenuma}

\subsection{Work in progress and open questions}
As noted previously, we are currently investigating extensions  of the functional central limit theorems in the current work  to the setting of more general type spaces. 
More generally, the question of understanding second order fluctuations of densities of microscopic components for general coagulation systems was already raised in the survey of Aldous in \cite{aldous1999deterministic}*{Section 5.5}. One of our future goals is to explore the techniques of the  current paper to understand second order fluctuations for various other types of coagulation models for which laws of large numbers are already known e.g. \cite{norris1999smoluchowski}.
Other related questions include: \begin{inparaenuma}
    \item Path space large deviations results extending the fixed time results of   \cite{andreis2023large} in the setting of dynamic graph valued processes; 
    \item Concentration inequalities and rates of convergence for the various macroscopic functionals; 
    \item Extensions of functional central limit results in Theorem \ref{thm:fclt-giant-surplus} to include other functionals such as the two-core and tree mantle -- fixed time instant CLT for these objects for the supercritical \erdos random graph have been studied in \cite{pittel2005counting}.  
\end{inparaenuma}

\subsection{Overview of the Proofs}\label{sec:proofs}

We begin in Section \ref{sec:fdd-fclt} with the proof of the finite-dimensional functional central limit theorem and characterizing the distribution of the limiting object namely all the results in Section \ref{sec:res-micro}. Next, Section \ref{sec:proofs-inf-dim} considers the infinite dimensional formulation, proving the Theorems in \ref{sec:res-micro-inf}. Section \ref{sec:proofs-macro-supercrit} then uses these results to prove fluctuation results of macroscopic functionals in the supercritical regime, namely the results in Section \ref{sec:res-macro}. Finally, Section \ref{sec:proofs-dense-mst} proves Theorem \ref{thm:clt-mst-weight} on the fluctuations of the weight of the MST on dense graphs.

\section{Finite Dimensional Functional Central limit theorems}
\label{sec:fdd-fclt}

In this section, we prove the finite dimensional LLN given in Theorem \ref{thm:odes-lln} and the  FCLT in Theorem \ref{thm:fdd-fclt}, and establish the mean and variance formulas presented in Proposition \ref{prop:mean-variance}. Throughout the section, fix $N \in \bN$.
For the rest of this section we will assume without loss of generality that 
\begin{equation} \label{eq:nonrandom}
\mbox{ the collection $\mvx^n$ is nonrandom and $\sqrt{n} (\mu_n-\mu)$ converges to a (nonrandom) vector $\Psi$ in $\bR^K$.}
\end{equation}
The fact that we can take $\mvx^n$ to be nonrandom is a consequence of the following elementary lemma.
\begin{lemma}
\label{lem:weakcgce}
$S_i$, $i=1,2,3,4$ be Polish spaces. Let, for each $n \in \bN$, $\{Y_n(u), u \in S_1\}$ be a collection of $S_2$ valued random variables given on some probability space $(\Omega, \cF, \bP)$, such that $(u,\omega) \mapsto Y_n(u,\omega)$ is a measurable map from $S_1\times \Omega$ to $S_2$.
Let $S^n_1$ be a measurable subset of $S_1$ and $\Phi_n: S^n_1 \times S_2 \to S_4$ be a measurable map.
Let $G: S_1 \times S_3 \to S_4$ be another measurable map. Let $Z$ be a $S_3$ valued random variable.
Suppose that for any sequence $\{u_n \in S^n_1, n \in \bN\}$, with $u^n \to u \in S_1$
\begin{equation}\label{eq:sn1}
\Phi_n(u_n, Y_n(u_n)) \weakc G(u, Z) \mbox{ as } n\to \infty.\end{equation}
Also suppose that for any $\tilde u_n \to \tilde u$ in $S_1$, $G(\tilde u_n, Z) \weakc G(\tilde u,Z)$, as $n \to \infty$.
Let, for each $n$
$U_n$ be a  $S_1^n$ valued random variable, defined on $(\Omega, \cF, \bP)$ and independent of $\{Y_n(x), x \in S_1\}$ such that $U_n \weakc U$, where $U$ is a $S_1$ valued random variable defined on the same space where $Z$ is defined and is independent of $Z$. Then
\begin{equation}\label{eq:sn2}
\Phi_n(U_n, Y_n(U_n))  \weakc G(U, Z) \mbox{ as } n \to \infty.\end{equation}
\end{lemma}

The proof of the lemma can be found in the Appendix.
To see how the lemma is applied to reduce the study to the setting of nonrandom $\mvx^n$, we consider for example the setting of Theorem \ref{thm:fdd-fclt}. Similar observations apply to other results proved in this section and those in Sections \ref{sec:proofs-inf-dim} - \ref{sec:proofs-dense-mst}. For a vector $\mvy^n = \{y^n_i \in \cS, i \in [n]\} \in \cS^n$, let $\mu_n(\mvy^n) : = \frac{1}{n} \sum_{i=1}^n \delta_{y^n_i}$. Let $S_1= \{v \in \bR^K: \sum_{i\in [K]} v_i =0\}$ and for $n \in \bN$, let
$$S^n_1 = \{v \in S_1: v_j= \mvv^n_j(\mvy^n)   = \sqrt{n}(\mu_n(\mvy^n)(j) - \mu(j)), \, j \in [K], \mbox{ for some } \mvy^n \in \cS^n\}.$$
Consider now for a fixed $\mvx^n \in \cS^n$ the corresponding $\pi_n^{\fT_N}
= \pi_n^{\fT_N}(\mvx^n)$ which can be represented as in \eqref{eqn:irg-poisson-eqn}. The second term on the right side of this equation can be viewed as a random field with values in $S_2 = \bD([0,\infty): \bR^{|\fT_N|})$, indexed by $\mvx^n \in \cS^n$, or more precisely, by $u_n = \mvv^n(\mvx^n) \in S^n_1$. We denote this random field as $\{\tilde Y_n(u_n), u_n \in S^n_1\}$. Also, from the ODE in Proposition \ref{prop:ode-solutions} we have that
$$\pi^{\fT_N}(t) = \pi^{\fT_N}(0) + \bar Y (t)$$
where $\bar Y$ is an element of  $\bD([0,\infty): \bR^{|\fT_N|})$ defined as
$\bar Y_{\mvl}(t) = \int_0^t F^{\fT_N}_{\mvl}(\pi^{\fT_N}(s),\kappa,\mu) ds$ for $t \ge 0$ and $\mvl \in \fT_N$.
Let $\{ Y_n(u_n), u_n \in S^n_1\}$ be a $S_2$ valued random field defined as
$Y_n(u_n) = \sqrt{n}(\tilde Y_n(u_n) - \bar Y)$.
Also, for suitable maps $\eta_n: S^n_1 \to \bR^{|\fT_N|}$, we have that,
$\sqrt{n}(\pi_n^{\fT_N}(0) - \pi^{\fT_N}(0)) = \eta^n(\mvv^n(\mvx^n)) = \eta^n(u_n)$.
Then, for any given $\mvx^n \in \cS^n$, the corresponding $\vX^{\fT_N}_n$
in Theorem \ref{thm:fdd-fclt} can be written as
$\vX^{\fT_N}_n = \Phi_n(u_n, Y_n(u_n))$ where $u_n = \mvv^n(\mvx^n)$ and, with
$S_4 = S_2$,
$\Phi_n:S^n_1 \times S_2 \to S_4$ is defined as 
$\Phi_n(u_n, Y_n(u_n)) = \eta^n(u_n) + Y_n(u_n)$.
Also from classical SDE theory we have that, associated with a solution of 
\eqref{eqn:fdd-sde}, with $S_3 = \bC([0,\infty): \bR^{|\fT_N|})$, there is a measurable map
$G: S_1\times S_3 \to S_4$ such that with $\mvB^{\fT_N}$ as in Section \ref{sec:res-micro}
and $u \in S_1$, setting $Z = \mvB^{\fT_N}$, $G(u,Z)$ is the solution of \eqref{eqn:fdd-sde} with $\vV^{\fT_N}_{\mvl}(0) = 
\sum_{k \in [K]}u_k \ind\set{\mvl = \ve_k}$, $\mvl \in \fT_N$.
Furthermore one has that $G(\tilde u_n, Z) \to G(\tilde u,Z)$ whenever
$\tilde u_n \to \tilde u$ in $S_1$. 

Then Lemma \ref{lem:weakcgce} says that once Theorem  \ref{thm:fdd-fclt} has been established for all nonrandom $\mvx^n$ satisfying \eqref{eq:nonrandom} (namely \eqref{eq:sn1} holds whenever $u_n \to u$), then the theorem also holds for all $\mvx^n$ that are random as in Assumption \ref{ass:irg} (namely \eqref{eq:sn2} holds whenever  $U_n \weakc U$).

Recall the finite dimensional truncation of the densities of types of components $\mvpi_n^{\fT_N}(t) = (\pi_n(\mvl,t) : \mvl \in \fT_N)$ defined in Section \ref{sec:results}. Our first step is to obtain a semi-martingale representation of  $\mvpi^{\fT_N}_n(t)$, by understanding the evolution  of $\mvpi_n^{\fT_N}(\cdot)$ .

\subsection{Dynamics of finite dimensional component densities}
\label{subsec:three-events}

Recall the definition of the quadratic form $\theta_{\tilde{\kappa}}$ in \eqref{eqn:thetakapp-def} as well as $\Delta_{\mvl,\mvk}^{\fT_N}$, and $\Delta_{\mvl}$ defined in \eqref{eqn:Delta-N-defns}. For $\mvl \in \fT$, define the functions $\varphi_{\mvl}: \cL \times \cM \times \bR_+^K \to \bR$, and $ \varphi_{\mvl,\mvk}: \cL \times \cM \to \bR$ for $\mvl,\mvk \in \fT \setminus \set{\mv0}$:

\begin{equation}\label{eqn:rate-functions}
\begin{aligned}
\varphi_{\mv0}(\mvx,\tilde\kappa,\tilde\mu) &= \frac{1}{2}\left[\theta_{\tilde\kappa}(\tilde\mu,\tilde\mu) - 2\sum_{\mvl\in \fT} x_{\mvl}\theta_{\tilde\kappa}(\mvl,\tilde\mu) + \sum_{\mvl,\mvk \in \fT} x_{\mvl}x_{\mvk} \theta_{\tilde\kappa}(\mvl,\mvk)\right]\\
\varphi_{\mvl}(\mvx,\tilde\kappa,\tilde\mu) &= x_{\mvl}\left[\theta_{\tilde\kappa}(\mvl,\tilde\mu) - \sum_{\mvk\in \fT}x_{\mvk}\theta_{\tilde\kappa}(\mvl,\mvk)\right],\qquad \mvl \neq \mv0\\
\varphi_{\mvl,\mvk}(\mvx,\tilde\kappa) &= 2^{-\ind \set{\mvl = \mvk}}x_{\mvl} x_{\mvk}\theta_{\tilde\kappa}(\mvl,\mvk), \qquad \mvl,\mvk \neq \mv0.
\end{aligned}
\end{equation}
We note that from the definition of $\cL$ (see \eqref{eqn:cl-IRG}) the infinite sums in the above display are well defined.

Following are the four ways, the addition of an edge can change $\mvpi^{\fT_N}_n(t)$:
\begin{enumeratei}
\item An edge is created between two components of type $\mvl$ and $\mvk$ with $\mv0 < \mvl < \mvk \in \fT_N$. This event happens at a rate of \begin{align*}
    n\varphi_{\mvl,\mvk}(\mvpi^{\fT_N}_n(t),\kappa_n) = [n\pi_n(\mvl,t)][n\pi_n(\mvk,t)]\frac{\theta_{\kappa_n}(\mvl,\mvk)}{n}.
\end{align*} This event changes $\mvpi_n^{\fT_N}(t)$ by $\Delta^{\fT_n}_{\mvl,\mvk}/n.$

\item Two \emph{different} components of type $\mvl  \in \fT_N\setminus\{ \mv0\}$ are merged by the addition of an edge. This event occurs at rate, 
$$n\varphi^*_{\mvl,\mvl}(\mvpi^{\fT_N}_n(t),\kappa_n) = \binom{n\pi_n(\mvl,t)}{2}\frac{\theta_{\kappa_n}(\mvl,\mvl)}{n}.$$ The first term in the RHS of the above equation is the number of ways of choosing two different components of type $\mvl \in \fT_N$. Recalling the definition of $\varphi_{\mvl,\mvl}$ in \eqref{eqn:rate-functions}, we have \begin{align}
    \varphi_{\mvl,\mvl}^*(\mvpi^{\fT_N}_n(t),\kappa_n) &= \varphi_{\mvl,\mvl}(\mvpi^{\fT_N}_n(t),\kappa_n) + \epsilon_n(\mvl,t) \label{eqn:phi*-l-defn}\\
     \epsilon_n(\mvl, t) &= -\frac{1}{2n}\pi_n(\mvl,t)\theta_{\kappa_n}(\mvl,\mvl)\label{eqn:epsilon_l-defn}.
\end{align}Note that for all $t\geq 0$, we have \begin{align}\label{eqn:epsilon_l-bound}
    |\epsilon_n(\mvl, t)| \leq \frac{1}{n}\norm{\kappa_n}_{\infty}\norm{\mvl}.
\end{align} This event changes $ \mvpi^{\fT_N}_n(t)$ by $\Delta^{\fT_N}_{\mvl,\mvl}/n$.
\item A component of type $\mvl  \in \fT_N\setminus\{ \mv0\}$ merges with a component of type \emph{not} in $\fT_N$. This event occurs at a rate of \begin{align*}
    n\varphi_{\mvl}(\mvpi^{\fT_N}_{n}(t),\kappa_n,\mu_n) &= [n\pi_n(\mvl,t)]\left[\sum_{i=1}^K l_i \left(\sum_{j=1}^K \left(\mu_n(j)- \sum\limits_{\mvk \in \fT_N} k_j \pi_n(\mvk,t)\right) \kappa_n(i,j)\right)\right]\\
    &= [n\pi_n(\mvl,t)]\left[\theta_{\kappa_n}(\mvl,\mu_n) - \sum_{\mvk\in \fT_N}\pi_n(\mvk,t)\theta_{\kappa_n}(\mvl,\mvk)\right].
\end{align*} This event changes $\mvpi^{\fT_N}_{n}(t)$ by $\Delta_{\mvl}/n$.

\item An edge is added either (a) between two vertices in \emph{the same} component of type in $\fT_N$ or (b) between two vertices in components of type not in $\fT_N$. We first compute the rate of these two events individually.
\begin{enumeratea}
\item For any $\mvl  \in \fT_N\setminus\{ \mv0\}$, following is the rate at which edges are added between vertices in a given component of type $\mvl$:\begin{align*}
    R_{n}(\mvl,t) &=\frac{1}{n}\left[ \sum_{1\leq i\leq j\leq K} 2^{-\ind\set{i=j}} l_i(l_j-\delta_{i,j}) \kappa_n(i,j)\right] = \frac{1}{2n} \left[\theta_{\kappa_n}(\mvl,\mvl) - \sum_{i=1}^K l_i\kappa_n(i,i)\right].
\end{align*} 
Therefore, the total rate at which edges are added in the same component in $\fT_N$ is, 
$$
R^{(a)}_n(t) = \sum_{\mvl \in \fT_N} n \pi_n(\mvl,t) R_n(\mvl,t) = \frac{1}{2} \left[\sum_{\mvl \in \fT_N} \pi_n(\mvl,t) \left(\theta_{\kappa_n}(\mvl,\mvl) - \sum_{i=1}^K l_i\kappa_n(i,i)\right)\right].$$
\item Recall $\varphi_{\mv0}$ defined in \eqref{eqn:rate-functions}. The rate at which edges are created between vertices in components of type \emph{not} in $\fT_N$ is  given by $R_n^{(b)}(t)$ defined as
\begin{align*}
       &\frac{1}{n}\left[\sum_{1\leq i\leq j\leq K} 2^{-\ind\set{i=j}}\left(n \mu_n(i) -\sum_{\mvl \in \fT_N} n\pi_n(\mvl,t) l_i\right)\left(n \mu_n(j)-\sum_{\mvl \in \fT_N} n\pi_n(\mvl,t) l_j -\delta_{i,j}\right)\kappa_n(i,j)\right]\\
    &= n \varphi_{\mv0}(\mvpi^{\fT_N}_n(t),\kappa,\mu) - \frac{1}{2} \sum_{i=1}^K \mu_n(i) \kappa_n(i,i) + \frac{1}{2}\sum_{\mvl \in \fT_N} \pi_n(\mvl,t) \sum_{i=1}^K l_i\kappa_n(i,i).
\end{align*}
\end{enumeratea}
Therefore, the total rate of occurrence of either event (a) or (b) is \begin{align}
    n\varphi_{\mv0}^*(\mvpi^N(t),\kappa_n,\mu_n) &=  R_n^{(a)}(t) + R_b^{(b)}(t)  =  n\varphi_{\mv0}(\mvpi^N(t),\kappa_n,\mu_n) + n \epsilon_n(\mv0, t) \qquad \text{ where}\label{eqn:phi^*-defn}\\
    \epsilon_n(\mv0,t) &= \frac{1}{2n} \left[\sum_{\mvl \in \fT_N} \pi_n(\mvl,t) \theta_{\kappa_n}(\mvl,\mvl) - \sum_{i=1}^K \mu_n(i)\kappa_n(i,i)\right] \label{eqn:epsilon-0-defn}
\end{align} Note that for all $t\geq 0$, we have \begin{align}\label{eqn:epsilon-0-bound}
    |\epsilon_n(\mv0, t)| \leq \frac{ \norm{\kappa_n}_\infty (N^{2} + 1) }{2n} \leq \frac{ \norm{\kappa_n}_\infty N^{2} }{n}.
\end{align} 

\end{enumeratei}
In the remainder of the section, to reduce notational overhead,  we suppress dependence on $\kappa_n,\mu_n$ and, e.g., write $\varphi_{\mvl,\mvk}(\pi^{\fT_N}_{n}(t))$ for $\varphi_{\mvl,\mvk}(\pi^{\fT_N}_n(t),\kappa_n)$ etc.

\subsection{Proof of Theorem \ref{thm:fdd-fclt} and Theorem \ref{thm:odes-lln}:} 
\label{sec:pfllnclt} 
In this subsection, we prove Theorems \ref{thm:odes-lln} and \ref{thm:fdd-fclt} using classical results from \cite{kurtz-ethier} (see  Theorem \ref{thm:fdd-fclt-kurtz}).

Let $\set{Y_{\mvl,\mvk}: \mvl \leq \mvk \in \fT\setminus \set{\mv0}}$ and $\set{Y_{\mvl}: \mvl \in \fT}$ be a collection of independent Poisson processes of rate $1$. From the description of jumps of 
 $\mvpi^{\fT_N}_{n}(t)$  given in the previous section, one can give a distributionally equivalent  representation for $\mvpi^{\fT_N}_{n}$ in terms of these Poisson processes as follows(see Chapter 6 of \cite{kurtz-ethier}).
\begin{align}\label{eqn:irg-poisson-eqn}
    \mvpi^{\fT_N}_n(t) &= \mvpi^{\fT_N}_n(0) + \frac{1}{n}\bigg[\sum_{\mvl \in \fT_N \setminus \set{\mv0}} \Delta_{\mvl} Y_{\mvl}\left(n\int_{0}^t \varphi_{\mvl}(\mvpi^{\fT_N}_n(s))ds\right) + e_{\mv0} Y_{\mv0}\left(n \int_0^t \varphi_{\mv0}^*(\mvpi_n^{\fT_N}(s))ds\right)\\
    &\hspace{-2mm}+\sum_{\mvl,\mvk \in \fT_N \setminus\set{\mv0}, \mvl<\mvk}\Delta^{\fT_N}_{\mvl,\mvk} Y_{\mvl,\mvk}\left(n \int_{0}^t \varphi_{\mvl,\mvk}(\mvpi^{\fT_N}_n(s))ds\right) +  \sum_{\mvl \in \fT_N\setminus \set{\mv0}}\Delta^{\fT_N}_{\mvl,\mvl} Y_{\mvl,\mvl}\left(n \int_{0}^t \varphi^*_{\mvl,\mvl}(\mvpi^{\fT_N}_n(s))ds\right)\bigg].\nonumber
\end{align} 

\noindent We now rewrite the above process using the notation of Theorem \ref{thm:fdd-fclt-kurtz}. We can view  the vector $(\pi_n^{\fT_n}(\cdot),\kappa_n,\mu_n) = (\mvxi_n^{\fT_N}(\cdot), \mvxi_n^{(0)}(\cdot),\mvxi_n^{(1)}(\cdot))$ as continuous time jump Markov process with on $\bR_+^{|\fT_N|+K^2 + K}$ with \begin{align*}
    \xi_n^{\fT_N}(t) = \pi_n^{\fT_N}(t),\qquad\xi_n^{(0)}(t) = \xi_n^{(0)}(0) = \kappa_n,\qquad \xi_n^{(1)}(t) = \xi_n^{(1)}(0) = \mu_n
\end{align*} where the rate functions for $\mvz = (\mvx^{\fT_N},\tilde\kappa,\tilde\mu)$ are given by $\beta_{\mvl}(\mvz) = \varphi_{\mvl}(\mvx^{\fT_N},\tilde\kappa,\tilde\mu)$, and $\beta_{\mvl,\mvk}(\mvz) = \varphi_{\mvl,\mvk}(\mvx^{\fT_N},\tilde\kappa,\tilde\mu)$ for $\mvl,\mvk \in \fT_{N}$ with associated jump vectors given by $\tilde\Delta_{\mvl} = (\Delta_{\mvl},\mv0,\mv0)$ and $\tilde\Delta^{\fT_N}_{\mvl,\mvk} = (\Delta^{\fT_N}_{\mvl,\mvk},\mv0,\mv0)$. From \eqref{eqn:rate-functions}, we see that the associated rate functions are differentiable with continuous derivatives. The error functions are  $\epsilon_{n,\mvl}(s) = \ind\set{\mvl = \mv0}\epsilon_n(\mv0,s)$ and $\epsilon_{n,\mvl,\mvk}(s) = \ind\set{\mvk = \mvl} \epsilon_n(\mvl,\mvl,s)$, where $\epsilon_n(\mv0,s)$
and $\epsilon_n(\mvl,\mvl,s)$ are
as defined in \eqref{eqn:epsilon-0-defn} and \eqref{eqn:epsilon_l-defn}. 

From \eqref{eqn:epsilon_l-bound} and \eqref{eqn:epsilon-0-bound}, we see that the error functions are bounded by $\gamma/n$ where $\gamma = \sup_{n}\norm{\kappa_n} N^2$. This shows that the first two assumptions of Theorem \ref{thm:fdd-fclt-kurtz} are satisfied.

Consider now Theorem \ref{thm:odes-lln} and suppose that Assumption \ref{ass:irg}(a) holds. Then, with
$\mvxi(0) = (\mvpi^{\fT_N}(0),\kappa,\mu)$, $\mvxi_n(0) \to \mvxi(0)$ in probability. Also, from Proposition \ref{prop:ode-solutions}, part (iv) of the assumption in Theorem \ref{thm:fdd-fclt-kurtz} is satisfied. Theorem \ref{thm:odes-lln} is now immediate from the first assertion in this theorem.

Now consider Theorem \ref{thm:fdd-fclt} and suppose that Assumption \ref{ass:irg} holds. Then, from
\eqref{eqn:inital-point-x}, we have $\mvxi(0) = (\mvpi^{\fT_N}(0),\kappa,\mu)$ and $\tilde{\vX}(0) = \sqrt{n}\left(\mvxi_n(0)- \mvxi(0)\right)\convd (\vX^{\fT_N}(0),\Lambda,\Psi)$. This verifies the additional condition (below (iv)) in  Theorem \ref{thm:fdd-fclt-kurtz}.

Therefore, as an immediate consequence of Theorem \ref{thm:fdd-fclt-kurtz}, we now have $\vX_n^{\fT_N}(\cdot) \convd \vX^{\fT_N}(\cdot)$ which is a solution to a \emph{linear} sde. To describe the coefficients of the  SDE, we make the following observation on the gradients of the function $F_{\mvl}$ for $\mvl \in \fT_N$ defined in \eqref{eqn:F_l-defn}: 
    \begin{align} \label{eqn:gradient-relations}
    \nabla_{\mvx}F_{\mvl}(\mvx,\kappa,\mu) = \bigg[\sum\limits_{\mvk_1 + \mvk_2 = \mvl}  x_{\mvk_1}  \theta_{\kappa}(\mvk_1,\mvk_2) e_{\mvk_2}\bigg] &- \theta_{\kappa}(\mvl,\mu)e_{\mvl},\hspace{2mm} \nabla_{\kappa}F_{\mvl}(\mvx,\kappa,\mu)\at[\big]{\kappa= \Lambda} \cdot \Lambda = F_{\mvl}(\mvx,\Lambda,\mu) \\
  \nabla_{\mu}F_{\mv0}(\mvx,\kappa,\mu)\at[\big]{\mu = \Psi} \cdot \Psi = \theta_{\kappa}(\mvl,\Psi), \hspace{0.5cm} &\nabla_{\mu}F_{\mvl}(\mvx,\kappa,\mu)\at[\big]{\mu = \Psi} \cdot \Psi = -x_{\mvl}\theta_{\kappa}(\mvl,\Psi), \hspace{2mm} \mvl \neq \mv0.
 \end{align}
Using the above, and from the SDE in \eqref{eqn:thm-kurtz-sde}, we see that the $\vX_n^{\fT_N}(\cdot)$ is the solution to the SDE in \eqref{eqn:fdd-sde}. The result follows. 
\qed

\subsection{Proof of Proposition \ref{prop:mean-variance}}
   The goal of this section is to  compute the mean and variance of the solution to the \eqref{eqn:fdd-sde}. Conditional on $\Psi$, the SDE \eqref{eqn:fdd-sde} is a linear SDE  and thus (cf.  \cite{karatzas2014brownian}*{Section 5.6}), the mean vector $m^{\fT_N}(t) = (m(\mvl,t):\mvl \in \fT_N)$ satisfies the differential equation $\frac{d}{dt}m^{\fT_N}(t) = \Gamma^{\fT_N}(t)m^{\fT_N}(t) + a^{\fT_N}(t)$. In particular, for $\mvl \in \fT_N$, we have, 
   \begin{equation}\label{eqn:mean-ode}
   \begin{aligned}
   m'(\mv0,t) &=  F_{\mv0}(\mvpi^{\fT_N}(t),\Lambda,\mu) + \theta_{\kappa}(\mvl,\Psi)\\
        m'(\mvl,t) &= \sum_{\mvk_1 + \mvk_2 = \mvl} \pi(\mvk_1,t)m(\mvk_2,t)\theta_{\kappa}(\mvk_1,\mvk_2) - \theta_{\kappa}(\mvl,\mu)m(\mvl,t) + F_{\mvl}(\mvpi^{\fT_N}(t),\Lambda,\mu) - \pi(\mvl,t)\theta_{\kappa}(\mvl,\Psi)\\
        &= (\Gamma^{\fT_N}(t) m(\cdot, t))_{\mvl} + F_{\mvl}(\mvpi^{\fT_N}(t),\Lambda,\mu) - \pi(\mvl,t)\theta_{\kappa}(\mvl,\Psi).
        \end{aligned}
    \end{equation}
    Using the ODEs in \eqref{eqn:ode-irg} for $\{\pi(\mvl,t)\}$, one can verify that $m(\mvl,t) = \nabla_{\kappa}\pi(\mvl,t)^T \cdot \Lambda + \nabla_{\mu}\pi(\mvl,t)^T \cdot \Psi$ solves the system of ODE \eqref{eqn:mean-ode}.

 Next,  recall $\Sigma_{\mvl,\mvk}(t) = \E\left((X(\mvl,t) - m(\mvl,t))(X(\mvk,t) - m(\mvk,t))| \Psi\right)$. Applying \cite{karatzas2014brownian}*{Equation 6.31}, $\Sigma^{\fT_N}(t)= (\Sigma_{\mvl,\mvk}(t))_{\mvl,\mvk \in \fT_N}$ satisfies the differential equation (coordinate-wise):

 \begin{align*}
        \frac{d}{dt}\Sigma^{\fT_N}(t) = \Gamma^{\fT_N}(t) \Sigma^{\fT_N}(t) + \Sigma^{\fT_N}(t)\Gamma^{\fT_N}(t)^T + \Phi^{\fT_N}(t).
    \end{align*}

\noindent One can verify that $\Sigma_{\mvl,\mvk}(t)$ given in \eqref{eqn:variance-covariance-matrix-irg} solves the above ODE.
This completes the proof of Proposition \ref{prop:mean-variance}. \qed

\section{Proofs: convergence in \texorpdfstring{$\ell_{1,\delta}$}{l1d}}
\label{sec:proofs-inf-dim}
We now develop the technical tools to extend the finite dimensional convergence of component densities established in Theorem \ref{thm:fdd-fclt} to a convergence result in infinite dimensions, namely Theorem \ref{thm:fclt-irg}. A starting point is 
establishing tightness of the process  $\set{X_n^{\fT_{M\log n}}(t)}_{n\geq 1}$ where $X_n^{\fT_{M\log n}}(t) = \left(\sqrt{n}(\pi_n(\mvl,t) -\pi(\mvl,t)):\mvl \in \fT_{M\log n}\right)$ for which we need suitable moment bounds, the crux of Section \ref{sec:moment-bounds}. As in Section \ref{sec:fdd-fclt}, we will assume without loss of generality that \eqref{eq:nonrandom} holds. Assumption \ref{ass:irg} will be in force throughout the section.

\subsection{Moment Bounds}\label{sec:moment-bounds}
 In this section, we prove bounds on the first two moments, in Theorem \ref{thm:mean-moment-bound-main} and Theorem \ref{thm:variance-moment-bound-main} respectively, which are the main tools in establishing tightness in Section \ref{sec:tightness}. For the rest of the paper, we use $C,C',C'',\dots, C_1,C_2,\dots$ to denote generic finite constants that might change from one line to the next. 

\begin{thm}[First Moment Bound]
\label{thm:mean-moment-bound-main}For any $M, \delta \geq 0$, and a compact interval $\cI \subseteq \bR_+$ \emph{not} containing $t_c$, we have
\begin{align*}
            \sup_{n \geq 1}\sup_{t\in \cI} \bigg[\sum_{\mvl \in \fT_{M\log n}} \norm{\mvl}^\delta \E(\pi_n(\mvl,t))\bigg] < \infty.
        \end{align*} 
\end{thm}

The proof of the theorem requires several lemmas which we prove first. The proof is broadly divided into two steps: (i) approximating $\E(\pi_n(\mvl,t))$ by the limit density $\pi(\mvl,t)$ uniformly in $\mvl$ and $t$;  (ii) establishing decay of $\pi(\mvl,t)$ with $\norm{\mvl}$, uniformly in an interval not containing $t_c$. We now focus on the first step.

The following lemma is needed to approximate laws of $\IRG$ by the distribution of the multi-type branching processes. The proof is a modification of the argument found in Lemma 9.6 in \cite{bollobas2007phase}. 

\begin{lemma}
\label{lem:TV-approximation}
    Let $U_n$ be a vertex uniformly picked at random from $[n]$, independent of other random processes defining the dynamic IRG model. Then, there is a $C\ge 0$, such that for any $\mvl  \in \fT\setminus\{\mv0\}$, $t \geq 0$, and $n \in \bN$, we have \begin{align*}
        \bigg|\pr\left(\fF(\cC_n(U_n,t)) = \mvl)\right) - \pr\left(\MBP_{\mu_n}(t\kappa_n,\mu_n)\in \bT_{\cS}(\mvl\right)\bigg| \leq \frac{Ct\norm{\mvl}^2}{n}.
    \end{align*} 
\end{lemma}

\begin{proof}
Consider the exploration of a component of a chosen vertex, where we begin by revealing all its adjacent edges one by one, and adding any discovered neighbors to a list of unexplored vertices. Select one vertex from this list, reveal its neighborhood, add to the list, and continue this process until either (i) the number of vertices of type $i$ exceed $l_i$ for some $i \in [K]$ or (ii) if there are no more vertices to explore.

In the exploration process, if at some stage there are $u_j$ many vertices of type $j$ which are not discovered in the exploration.  Then,  $n\mu_n(j) - l_j \leq u_j \leq n\mu_n(j)$, and the number of the type $j$ vertices discovered in the next step, provided that that the chosen vertex at that stage is of type $i$, follows $\text{Binomial}\left(u_j, 1-\exp(-\frac{t\kappa_n(i,j)}{n})\right)$ distribution. Also, we have \begin{align*}
         \norm{\text{Binomial}\left(u_j, 1-\exp\left(-\frac{t\kappa_n(i,j)}{n}\right)\right) - \text{Poisson}\left(t\kappa_n(i,j)\mu_n(j)\right)}_{TV}&\\
         &\hspace{-10cm}\leq  \norm{\text{Binomial}\left(u_j, 1-\exp\left(-\frac{t\kappa_n(i,j)}{n}\right)\right) - \text{Poisson}\left(t\kappa_n(i,j)\frac{u_j}{n}\right)}_{TV} \\
         &\hspace{-9.5cm}+ \norm{\text{Poisson}\left(t\kappa_n(i,j)\frac{u_j}{n}\right) - \text{Poisson}\left(t\kappa_n(i,j)\mu_n(j)\right)}_{TV} \\
         &\hspace{-10cm}\leq \frac{C_1 t\kappa_n(i,j)}{n} + C_1't\kappa_n(i,j) \left|\mu_n(j)- \frac{u_j}{n}\right| \leq \frac{C_2t\norm{\kappa_n}_\infty\norm{\mvl}}{n}
     \end{align*} for some constants $C_1,C_1',C_2 \geq 0$, not depending on $i,j, n$ or $t$. The last line follows from the upper bounds on the total variational distance for Poisson and Binomial distributions (see \cite{TV-Poisson}). The rest of the argument follows from the same reasoning as in Lemma $9.6$ in \cite{bollobas2007phase}. Since the exploration process ends after at most $\norm{\mvl}$ steps, we have $$
     |\pr\left(\fF(\cC_n(U_n,t)) = \mvl\right) - \pr\left(\MBP_{\mu_n}(t\kappa_n,\mu_n) \in \bT_{\cS}(\mvl)\right)| \leq \norm{\mvl} \cdot \frac{C_2t\norm{\mvl}\norm{\kappa_n}_\infty}{n},$$ for some constant $C_2 \geq 0$. This finishes the proof. 
\end{proof}

Following is a simple bound on the difference between distribution of two multi-type branching processes.
\begin{lemma}
\label{lem:mbp-tv-approx}
   Fix kernels $\tilde\kappa_1,\tilde\kappa_2$ and measures $\tilde\mu_1,\tilde\mu_2$ on $\cS$. For any $\mvl \in \fT\setminus\set{\mv0}$, we have \begin{align*}
      \sup_x \left|\pr\left(\MBP_x(\tilde\kappa_1,\tilde\mu_1) \in \bT_{\cS}(\mvl)\right)-\pr(\MBP_x(\tilde\kappa_2,\tilde\mu_2) \in \bT_{\cS}(\mvl))\right| \leq C \norm{\mvl}\left[\norm{\tilde\kappa_1-\tilde\kappa_2}_\infty + \norm{\tilde\mu_1-\tilde\mu_2}_\infty\right],
   \end{align*}
   where $C = C' \left[ \tilde\mu_2(\cS) + \norm{\tilde\kappa_1}_\infty \right]K$ for some constant $C'$.
\end{lemma}
\begin{proof} Observe that the event $\set{\MBP_x(\tilde\kappa_1,\tilde\mu_1) \in \bT_{\cS}(\mvl)}$ is determined by $l_i$ many collection of independent $\set{\text{Poisson}\left(\tilde\kappa_1(i,j)\tilde\mu_1(j)\right):j\in [K]}$ random variables for $i\in [K]$. Hence, \begin{align*}
\bigg|\pr\left(\MBP_x(\tilde\kappa_1,\tilde\mu_1) \in \bT_{\cS}(\mvl)\right)-\pr\left(\MBP_x(\tilde\kappa_2,\tilde\mu_2) \in \bT_{\cS}(\mvl)\right)\bigg| &\\
        &\hspace{-9cm}\leq    \sum_{1\leq i,j\leq K} \norm{\Pois\left(\tilde\kappa_1(i,j)\tilde\mu_1(j)\right) -\Pois\left(\tilde\kappa_2(i,j)\tilde\mu_2(j)\right) }_{TV} l_i\\ 
        &\hspace{-9cm}\leq C' \sum_{1\leq i,j\leq K}|\tilde\kappa_1(i,j)\tilde\mu_1(j) - \tilde\kappa_2(i,j)\tilde\mu_2(j)| l_i,
    \end{align*}  for some constant $C'$ given by the bound on total variational distance between Poisson random variables.
    Triangle inequality and algebraic manipulations now yield the result.
\end{proof}

We next state the following result, proved in \cite{bhamidi2015aggregation}*{Corollary $6.16$ and Lemma $6.17$}, which is an important tool to bound the moments of the total population for a sub-critical multi-type branching process. 
Let $\ind: \cS \to \bR$ be defined as $\ind(x) = 1$ for all $x \in \cS$.
\begin{lemma}
\label{lem:mgf-mbp}
Fix a kernel $\kappa'$ and a measure $\mu'$ on $\cS$.     For $x\in \cS$, let $\MBP_x(\kappa',\mu')$ be as introduced below Definition \ref{defn:MBP}. assume $\norm{T_{\kappa',\mu'}} < 1$. For any $\delta \geq 0 $ such that $(1+\delta)\norm{T_{\kappa',\mu'}} < 1$, define $\phi_{\delta}(x) = \sum\limits_{i=0}^\infty (1+\delta)^iT_{\kappa',\mu'}^i\ind(x)$. For any $\epsilon < \log (1+\delta)/\norm{\phi_{\delta}}_{\infty}$, we have \begin{align*}
        \sup_{x \in \cS}\E\left(\exp\left(
\epsilon |\MBP_{x}(\kappa',\mu')|\right)\right) \leq \exp\left(\epsilon \left(1+ \norm{\kappa'}_{\infty}\mu'(\cS) \frac{1+\delta}{1-(1+\delta)\norm{T_{\kappa',\mu'}}}\right)\right).
    \end{align*}
\end{lemma}
The following  consequence of the above lemma, which is proved below the proof of Lemma \ref{lem:dual-norm}, will be used in the proof of tightness.

\begin{lem}\label{lem:all-moments-finite}
    For any compact interval $\cI \subseteq \bR_+$  not containing $t_c$, and $\delta \geq 0$, we have \begin{align*}
        \sup_{t\in \cI} \left[\sum_{\mvl \in \fT} \norm{\mvl}^{\delta} \pi(\mvl,t)\right] \leq \sup_{x\in [K]}\sup_{t\in \cI} \left[\sum_{\mvl \in \fT\setminus \set{\mv0}} \norm{\mvl}^{\delta-1} \pr\left(\MBP_x(t\kappa,\mu) \in \bT_{\cS}(\mvl)\right)\right]  < \infty.
    \end{align*}
\end{lem}

The proof of the above lemma  relies on the following continuity property of the norms of the integral operators associated with the dual MBP (see Section \ref{subsec:irg-definitions}), when the sequence of kernels converge to a kernel, and the measure $\mu$ is fixed. 

\begin{lemma}[Continuity of dual norms]
\label{lem:dual-norm}
Consider a sequence of kernels $\kappa'_n$  converging to $\kappa'$ and let $\mu_0$ be a finite measure  on $\cS$. Consider the associated dual operators $\hat{T}_{\kappa'_n,\mu_0} = T_{\kappa'_n,\hat{\mu}_0}$, and similarly define $\hat{T}_{\kappa',\mu_0}$. We then have \begin{align*}
  \lim_{n\to \infty}  \norm{\hat{T}_{\kappa'_n,\mu_0}} = \norm{\hat{T}_{\kappa',\mu_0}}.
\end{align*}
\end{lemma}

\begin{proof}
    By Theorem $6.4$ in \cite{bollobas2007phase}, we have $\norm{\rho_{\kappa'_n,\mu_0} -\rho_{\kappa',\mu_0}}_\infty  \to 0$. Also, recalling the definition of $\hat \mu_0$ from \eqref{eqn:dual-prob}, we have \begin{align*}
        \left|  \norm{\hat{T}_{\kappa'_n,\mu_0}} -\norm{\hat{T}_{\kappa',\mu_0}} \right| \leq\left[ \norm{\kappa'_n-\kappa'}_{\infty} + \norm{\kappa'}_{\infty} \norm{\rho_{\kappa'_n,\mu_0} -\rho_{\kappa',\mu_0}}_\infty\right] \mu_0(\cS)^2.
    \end{align*}The result is now immediate.
\end{proof}

We now give a proof of  Lemma \ref{lem:all-moments-finite} using  Lemma \ref{lem:dual-norm} and Lemma \ref{lem:mgf-mbp}. 

\begin{proof}[Proof of Lemma \ref{lem:all-moments-finite}:]
The first inequality in the assertion is immediate from the first statement in Proposition \ref{prop:ode-mbp-irg}. So let us now prove the finiteness assertion. Since the dual multi-type branching process  goes extinct almost surely, we have $\norm{\hat{T}_{t\kappa,\mu}} < 1$ for all $t\neq t_c$. Also, by Lemma \ref{lem:dual-norm}, we have $t\to \norm{\hat{T}_{t\kappa,\mu}}$ is a continuous function on $[0,\infty)$. Since $\cI$ is compact, we have $\sup_{t\in \cI} \norm{\hat{T}_{t\kappa,\mu}} < 1$.
Hence, there exists $\delta > 0$ such that $(1+\delta)\norm{\hat{T}_{t\kappa,\mu}} < 1$ for all $t\in \cI$.  By Lemma \ref{lem:mgf-mbp}, there exists $\epsilon > 0$, such that for all $x\in \cS$,  $\sup\limits_{t\in \cI}\E\left(\exp\left(\epsilon |\MBP_{x}(t\kappa,\hat{\mu}_t)|\right)\right) < \infty$, where $\hat \mu_t$ is as in Proposition \ref{prop:ode-mbp-irg}. Thus, for any $\delta \geq 0$ and $x\in \cS$, by Markov's inequality,
\begin{multline*}
    \sup_{t\in \cI}\left[\sum_{\mvl \in \fT} \norm{\mvl}^{\delta-1} \pr(\MBP_x(t\kappa,\mu) \in \bT_{\cS}(\mvl))\right]\le  \sup_{t\in \cI}\left[\sum_{k=1}^\infty k^{\delta-1} \pr(|\MBP_{x}(t\kappa,\hat{\mu}_t)| = k) \right] \\
    \leq  \left[\sup_{t\in \cI}\E(\exp\left(\epsilon|\MBP_{x}(t\kappa,\hat{\mu}_t)|\right) \right] \left[\sum_{k=1}^\infty k^{\delta-1} \exp(-\epsilon k)\right] < \infty,
\end{multline*} 
where the first inequality uses Lemma \ref{lem:conditional-dual}.
This finishes the proof.
\end{proof}

\begin{rem}\label{rem:exp-decay}
    The key observation in the above proof is the following. For every $t\neq t_c$, there exists $\delta(t) > 0$ and $C(t)>0$ such that $\pi(\mvl,t) \leq C(t)\exp\left(-\delta(t)\norm{\mvl}\right)$. Furthermore, for any compact interval $\cI$ not containing $t_c$, we have $\sup_{\cI}C(t) < \infty$ and \abb{$\inf_{\cI}\delta(t) > 0$}.
\end{rem}

We now provide a proof of Theorem \ref{thm:mean-moment-bound-main} using the above Lemmas.
\begin{proof}[Proof of Theorem \ref{thm:mean-moment-bound-main}]
Fix $M,\delta \geq 0$, and a compact interval $\cI$ not containing $t_c$. Let $T = \sup \cI$. For $n\geq 1$, let $U_n$ be a uniformly picked vertex from $[n]$. Observe that $$\sum\limits_{\mvl \in \fT_{M\log n}} \norm{\mvl}^\delta \E\left(\pi_n(\mvl,t)\right) = \sum_{i=0}^{3} R_n^{(i)}(t) $$ where \begin{align}\label{eqn:r_i's-defn}
R_n^{(0)}(t)&= \E(\pi_n(\mv0,t))\\
    R_n^{(1)}(t) &= \sum_{\mvl \in \fT_{M\log n}\setminus\{\mv0\}} \norm{\mvl}^{\delta-1} \bigg[\pr(\fF(\cC_n(U_n,t)) = \mvl) -\pr(\MBP_{\mu_n}(t\kappa_n,\mu_n) \in \bT_{\cS}(\mvl))\bigg]\\
    R_n^{(2)}(t) &= \sum_{\mvl \in \fT_{M\log n}\setminus\{\mv0\}}\norm{\mvl}^{\delta-1} \bigg[\pr(\MBP_{\mu_n}(t\kappa_n,\mu_n) \in \bT_{\cS}(\mvl)) -\pr(\MBP_{\mu_n}(t\kappa,\mu)\in \bT_{\cS}(\mvl))\bigg]\\
    R_n^{(3)}(t) &= \sum_{\mvl \in \fT_{M\log n}\setminus\{\mv0\}} \norm{\mvl}^{\delta-1}  \pr(\MBP_{\mu_n}(t\kappa,\mu)\in \bT_{\cS}(\mvl)).
\end{align}
To prove the result, it is sufficient to show that $
\sup_n \sup_{t\in \cI} \left|R_n^{(i)}(t)\right| < \infty$ for $0\leq i\leq 3.$ For the first term, note that $\{n\pi_n(\mv0,t)\}_{t\ge 0}$ is stochastically bounded by a Poisson process with rate $n\theta_{\kappa_n}(\mu_n,\mu_n)$. Therefore, $$\sup_{t\in \cI}R_n^{(0)}(t) \leq  T \norm{\kappa_n}_\infty.$$ For $R_n^{(1)}(t)$, using Lemma \ref{lem:TV-approximation}, we have 
\begin{align}\label{eqn:r_1,n-bound}
\sup_n \sup_{t\in \cI}\left|R_n^{(1)}(t)\right| \leq \sup_n \sum_{\mvl \in \fT_{M\log n}\setminus \{\mv0\}} \norm{\mvl}^{\delta-1} \frac{C_1T\norm{\mvl}^2}{n} \leq \sup_n \frac{C_1' T(\log n)^{\delta+K+1}}{n} < \infty    
\end{align} for some positive constants $C_1$. 
 For the next term $R_n^2(t)$, by Lemma \ref{lem:mbp-tv-approx}, we have 
\begin{align}\label{eqn:r_2,n-bound}
         \sup_n \sup_{t\in \cI} R_n^{(2)}(t) &\leq C_2 T\sup_n \left(\norm{\mu_n-\mu}_\infty + \norm{\kappa_n-\kappa}_\infty\right)  \left[\sum_{\mvl \in \fT_{M\log n}} \norm{\mvl}^{\delta}\right] \nonumber\\
         &\leq sup_n \frac{C_2'T(\log n)^{\delta+K}}{\sqrt{n}} <\infty
      \end{align} for some constants $C_2, C_2' \geq 0$. 

For the last term, by Lemma \ref{lem:all-moments-finite}, we have \begin{align}\label{eqn:r_3,n-bound}
   \sup_n \sup_{t\in \cI} R_n^{(3)}(t) \leq \sup_{x \in [K]}\sup_{t\in \cI} \left[\sum_{\mvl \in \fT_{M\log n}\setminus\{\mv0\}} \norm{\mvl}^{\delta-1} \pr(\MBP_{x}(t\kappa,\mu)\in \bT_{\cS}(\mvl))\right] < \infty.
\end{align} Thus, by \eqref{eqn:r_1,n-bound}, \eqref{eqn:r_2,n-bound},\eqref{eqn:r_3,n-bound}, we have \begin{align*}
    \sup_{n}\sup_{t\in \cI} \left[\sum_{\mvl \in \fT_{M \log n}} \norm{\mvl}^\delta \E\left(\pi_n(\mvl,t)\right)\right] < \infty.
\end{align*} This finishes the proof.
\end{proof}

We now focus on the second moment bound.  
\begin{thm}[Second Moment Bound]
\label{thm:variance-moment-bound-main}
Let $\cI$ be a compact interval of $\bR_+$ not containing $t_c$, then for any $M,\delta \geq 0$, we have
    \begin{align*}
            \sup_{n}\sup_{t\in \cI} \left[\E\left(X_n(\mv0,t)\right)^2 + \sum_{\mvl \in \fT_{M \log n}\setminus\set{\mv0}} \norm{\mvl}^\delta \E\left(X_n(\mvl,t)\right)^2 \right] < \infty
        \end{align*} 
\end{thm} 

We need four lemmas before commencing on the proof of the theorem.

\begin{lem}\label{lem:individial-terms-variance}
Fix $v,u \in [n]$, $v \neq u$, and $\mvl,\mvk \in \fT\setminus\set{\mv0}$. Then 
\begin{align*}
    &\pr\left(\fF(\cC_n(v,t))= \mvl,\fF(\cC_n(u,t)) = \mvk, u\not\in \cC_n(v,t) \right) -\pr\left(\fF(\cC_n(v,t))= \mvl\right)\pr\left(\fF(\cC_n(u,t))= \mvk\right) \\
    & \sa{\leq \frac{\norm{\kappa_n}_\infty \norm{\mvl}\norm{\mvk}t}{n} \pr(\fF(\cC_n(v,t))= \mvl,\fF(\cC_n(u,t)) = \mvk, u \not\in \cC_n(v,t)).}
\end{align*}    
\end{lem}
\begin{proof}
    For simplicity of exposition, we only  prove the result when  $K = 2$. The proof  for a general $K $ can be completed in a similar manner. Without loss of generality, assume both vertices $v$ and $u$ are of type $1$. Let $n_i = n \mu_n(i)$ for $i=1,2$, and $\mvn = (n_1,n_2)$. We then have\begin{align}\label{eqn:81239}
        &\pr(\fF(\cC_n(v,t))= \mvl,\fF(\cC_n(u,t)) = \mvk, u \not\in \cC_n(v,t))\nonumber\\
        & = \binom{n_1 -2}{l_1-1} \binom{n_2}{l_2} \binom{n_1-l_1-1}{k_1-1}\binom{n_2-l_2}{k_2}\nonumber \\
        & \quad \exp\left(-\frac{t}{n}\left[\theta_{\kappa_n}(\mvl,\mvn-\mvl) + \theta_{\kappa_n}(\mvk,\mvn-\mvl-\mvk)\right]\right) p_n(\mvl,t) p_n(\mvk,t)
    \end{align} where $p_{n}(\mvl,t)$ is the probability that a graph with $l_i$ many vertices of type $i$ and edge connection probability between a vertex of type $x$ and type $y$ given by $1-\exp(-\frac{t\kappa_n(x,y)}{n})$ is connected. Similarly, we have \begin{align}
        &\pr(\fF(\cC_n(v,t))= \mvl) \pr(\fF(\cC_n(u,t)=\mvk)\notag \\
        &= \binom{n_1-1}{l_1-1}\binom{n_2}{l_2} \binom{n_1-1}{k_1-1} \binom{n_2}{k_2} \exp\left(-\frac{t}{n}\left[\theta_{\kappa_n}(\mvl,\mvn-\mvl) + \theta_{\kappa_n}(\mvk,\mvn-\mvk)\right]\right) p_n(\mvl,t) p_n(\mvk,t). \label{eqn:62342}
    \end{align} 
    Combining \eqref{eqn:81239} and \eqref{eqn:62342}, we have \begin{align*}
        &\pr(\fF(\cC_n(v,t))= \mvl,\fF(\cC_n(u,t)) = \mvk, u \not\in \cC_n(v,t))- \pr(\fF(\cC_n(v,t))= \mvl) \pr(\fF(\cC_n(u,t)) = \mvk)\\
        &\leq  \pr\left(\fF(\cC_n(v,t)) = \mvl \right) \binom{n_1-l_1 - 1}{k_1-1} \binom{n_2-l_2}{k_2}  p_n(\mvk,t) \\
        &\quad \quad\exp\left(-\frac{t\theta_{\kappa_n}(\mvk,\mvn-\mvl-\mvk)}{n}\right) \left[1-\exp\left(-\frac{t\theta_{\kappa_n}(\mvl,\mvk)}{n}\right) \right]\\
        &\leq \pr\left(\fF(\cC_n(v,t)) = \mvl \right) \binom{n_1-l_1 - 1}{k_1-1} \binom{n_2-l_2}  {k_2}  p_n(\mvk,t) \exp\left(-\frac{t\theta_{\kappa_n}(\mvk,\mvn-\mvl-\mvk)}{n}\right) \frac{t\theta_{\kappa_n}(\mvl,\mvk)}{n}\\
        &\sa{ = \frac{t\theta_{\kappa_n}(\mvl,\mvk)}{n} \pr(\fF(\cC_n(v,t))= \mvl,\fF(\cC_n(u,t)) = \mvk, u \not\in \cC_n(v,t))}\\
        &\sa{\leq \frac{\norm{\kappa_n}_\infty \norm{\mvl}\norm{\mvk}t}{n} \pr(\fF(\cC_n(v,t))= \mvl,\fF(\cC_n(u,t)) = \mvk, u \not\in \cC_n(v,t)).} 
    \end{align*} The second inequality follows from the observation that $|1-e^{-z}| \leq z$ for all $z\in \bR_+$.
    By symmetry, the same inequality holds with $(v, \mvl)$ replaced with $(u, \mvk)$. This finishes the proof.
\end{proof}

\sa{\begin{lem}\label{lem:variance-sum-bound}
     For any $\cA \subseteq \fT\setminus{0}$, we have $$
     \ab{n}\var\left(\sum_{\mvl \in \cA} \pi_n(\mvl,t)\right) \leq (1+\norm{\kappa_n}_\infty t) \pr\left(\fF(\cC_n(U_n,t)) \in \cA \right)$$
 \end{lem}
 \begin{proof}
     Let $U_n$ and $V_n$ be two iid Uniform random variables on $[n]$ that are independent of the process $\set{\cG_n(t):t\geq 0}$. Note that \begin{align}\label{eqn:square-20}
    \E\left(\sum_{\mvl\in \cA}\pi_n(
\mvl,t)\right)^2 &= \E\left(\frac{\ind\set{\fF(\cC_n(U_n,t) \in \cA,\fF(\cC_n(V_n,t) \in \cA}}{\abs{\cC_n(U_n,t)} \abs{\cC_n(V_n,t)}}\right)\nonumber\\
&= \E\left(\frac{\ind\set{\fF(\cC_n(U_n,t) \in \cA, V_n \in \cC_n(U_n,t)}}{\abs{\cC_n(U_n,t)}^2}\right)\nonumber \\
&\hspace{1cm}+ \E\left(\frac{\ind\set{\fF(\cC_n(U_n,t) \in \cA,\fF(\cC_n(V_n,t) \in \cA, V_n \notin \cC_n(U_n,t)}}{\abs{\cC_n(U_n,t)}\abs{\cC_n(V_n,t)} }\right).
\end{align} 
Observe that the first term in the above is \begin{align}\label{eqn:first-term-21}
\E\left(\frac{\ind\set{\fF(\cC_n(U_n,t) \in \cA, V_n \in \cC_n(U_n,t)}}{\abs{\cC_n(U_n,t)}^2}\right) &= \E\left(\frac{\ind\set{\fF(\cC_n(U_n,t) \in \cA}}{\abs{\cC_n(U_n,t)}^2} \pr\left( V_n \in \cC_n(U_n,t)| \cC_n(U_n,t)\right)\right)\nonumber\\
&=\E\left(\frac{\ind\set{\fF(\cC_n(U_n,t) \in \cA}}{\abs{\cC_n(U_n,t)}^2} \ \frac{\abs{\cC_n(U_n,t)}}{n}\right)\nonumber\\
&= \E\left(\sum_{\mvl\in \cA}\pi_n(\mvl,t)\right) \leq \pr\left(\fF(\cC_n(U_n,t)) \in \cA\right).
\end{align}
Also, expanding the second term in \ref{eqn:square-20} and using Lemma \ref{lem:individial-terms-variance}, we get
\begin{align}\label{eqn:second-term-22}
&\E\left(\frac{\ind\set{\fF(\cC_n(U_n,t)) \in \cA,\fF(\cC_n(V_n,t)) \in \cA, V_n \notin \cC_n(U_n,t)}}{\abs{\cC_n(U_n,t)} \abs{\cC_n(V_n,t)} }\right) - \left(\E\left(\sum_{\mvl\in\cA} \pi_n(\mvl,t)\right)\right)^2 \nonumber\\
&\leq \frac{1}{n^2} \Big[ \sum_{\substack{i\neq j \in [n]\\{\mvl,\mvk\in \cA}}}\nonumber\\
&\quad\quad\frac{\pr\left(\fF(\cC_n(i,t)) = \mvl,\fF(\cC_n(j,t)) = \mvk, i \notin \cC_n(j,t)\right) - \pr\left(\fF(\cC_n(i,t)) = \mvl\right)\pr\left(\fF(\cC_n(j,t)) = \mvk\right)}{\norm{\mvl}\norm{\mvk}} \Big]\nonumber \\
&\leq \frac{\norm{\kappa_n}_\infty t}{n^3}\left[ \sum_{i\neq j \in [n]}\sum_{\mvl,\mvk\in \cA} \pr\left(\fF(\cC_n(i,t)) = \mvl,\fF(\cC_n(j,t)) = \mvk, i \notin \cC_n(j,t)\right)  \right]\nonumber\\
&=\frac{\norm{\kappa_n}_\infty t}{n} \pr\left(\fF(\cC_n(U_n,t)) \in \cA,\fF(\cC_n(V_n,t)) \in \cA, V_n \notin \cC_n(U_n,t)\right).
\end{align} 
Combining \eqref{eqn:first-term-21} and \eqref{eqn:second-term-22}, we have $$n \var\left(\sum_{\mvl \in \cA} \pi_n(\mvl,t)\right) \leq (1+\norm{\kappa_n}_\infty t) \pr\left(\fF(\cC_n(U_n,t)) \in \cA\right).
$$ This completes the proof
 \end{proof}
We note that as a special case, letting $\cA = \set{\mvl}$ for $\mvl \in \fT\setminus\set{\mv0}$, we get the following.
\begin{cor}
\label{cor:var-inequality}
    For any $\mvl \in \fT\setminus\set{\mv0}$ and $t\geq 0$, we have\begin{align*}
        n \var(\pi_n(\mvl,t)) \leq \norm{\mvl} \E(\pi_n(\mvl,t)) \left( 1 +  \norm{\kappa_n}_\infty t \right).
    \end{align*}
\end{cor}}

The following lemma provides  more refined estimates on the perturbations of distributions of multi-type branching processes than those provided by Lemma \ref{lem:mbp-tv-approx}. 
\begin{lemma}
\label{lem:abs-cont-mbp}
    Let $\tilde\kappa_1,\tilde\kappa_2$ be  kernels on $\cS$ and $\tilde\mu_1,\tilde\mu_2$ be  probability measures on $\cS$ such that all entries of $\tilde\kappa_i$ and $\tilde\mu_i$ are bounded below by $\alpha \in (0,1)$ for $i=1,2$. For any $\mvl \in \fT\setminus\set{\mv0}$, we have    \begin{align*}
\abs{\pi(\mvl,t;\tilde\kappa_1,\tilde\mu_1) - \pi(\mvl,t;\tilde\kappa_2,\tilde\mu_2)}\leq C\norm{\mvl}(1+t) \int_0^1 \pi(\mvl,t;\tilde\kappa^{(u)},\tilde\mu^{(u)})  du
    \end{align*} where $\tilde\kappa^{(u)} = u\tilde\kappa_1 + (1-u)\tilde\kappa_2$, $\tilde\mu^{(u)} = u\tilde\mu_1 + (1-u)\tilde\mu_2$, and $C = 2\left(\norm{\tilde\kappa_1-\tilde\kappa_2}_\infty + \norm{\tilde\mu_1-\tilde\mu_2}_\infty\right)/\alpha$.
\end{lemma}
\begin{proof}
    Fix $\vt \in \bT_{\cS}(\mvl)$ and $t\geq 0$. For a vertex $v \in \vt$, let $x_v$ be its type and $d_v(j)$ to be the number of off-springs of $v$ that are of type $j \in [K]$. Then, for any kernel $\tilde\kappa$ and a measure $\tilde\mu$, we have\begin{align}\label{eqn:123142156775}
\pr\left(\MBP_x\left(t\tilde\kappa,\tilde\mu\right)=\vt\right) &= \prod\limits_{v \in \vt} \prod\limits_{j=1}^K \pr\left(\Pois\left(t\tilde\kappa\left(x_v,j)\tilde\mu(j)\right)=d_{v}(j)\right)\right)\nonumber\\
        &=\prod\limits_{v \in \vt} \prod\limits_{j=1}^K \frac{\left(t\tilde\kappa(x_v,j)\tilde\mu(j)\right)^{d_{v}(j)}}{d_v(j)!} \exp\left(-t\tilde\kappa(x_v,j)\tilde\mu(j)\right).
    \end{align}
The above expression when viewed as a function of $\tilde\kappa$ and $\tilde\mu$ is continuously differentiable. In particular, for  $\tilde\Lambda \in \bR^{K\times K}$ and $\tilde\Psi \in \bR^K$, we have 
\begin{align}\label{eqn:gradient-trees-kappa}
&\abs{\nabla_{\tilde\kappa}\pr\left(\MBP_{\tilde\mu}\left(t\tilde\kappa,\tilde\mu\right)=\vt\right)\cdot \tilde\Lambda}\nonumber\\
&\leq \frac{\norm{\tilde\Lambda}_\infty}{(\inf_{x,y \in \cS}\tilde\kappa(x,y))\wedge 1} \left[ \sum_{v\in \vt}\sum_{j=1}^K d_v(j) +  t \sum_{v\in \vt}\sum_{j=1}^K \tilde\mu(j)\right] \pr\left(\MBP_{\tilde\mu}\left(t\tilde\kappa,\tilde\mu\right)=\vt\right)\nonumber\\
&\leq \frac{2\norm{\tilde\Lambda}_\infty }{(\inf_{x,y \in \cS}\tilde\kappa(x,y))\wedge 1} \norm{\mvl} (1+t) \pr\left(\MBP_{\tilde\mu}\left(t\tilde\kappa,\tilde\mu\right)=\vt\right).
\end{align} 
In the above display, the second line follows from the observation that $\sum_{v\in \vt} \sum_{j=1}^K d_v(j)$ is the total number of vertices in the tree $\vt$ which is bounded by $\norm{\mvl}$. Similarly, one also has \begin{align}\label{eqn:gradient-trees-mu}
\abs{\nabla_{\tilde\mu}\pr\left(\MBP_{\tilde\mu}\left(t\tilde\kappa,\tilde\mu\right)=\vt\right)\cdot \tilde\Psi} &\leq \frac{2\norm{\tilde\Psi}_\infty }{\inf_{x \in \cS}\tilde\mu(x)} \norm{\mvl} (1+ \norm{\tilde\kappa}_\infty t) \pr\left(\MBP_{\tilde\mu}\left(t\tilde\kappa,\tilde\mu\right)=\vt\right).
\end{align}
Finally, by \eqref{eqn:gradient-trees-kappa} and \eqref{eqn:gradient-trees-mu}, along with Proposition \ref{prop:ode-mbp-irg}, we have \begin{align}\label{eqn:908908}
\abs{\pi(\mvl,t;\tilde\kappa_1,\tilde\mu_1) - \pi(\mvl,t;\tilde\kappa_2,\tilde\mu_2)} &= \abs{ \int_0^1 \frac{d}{d u}\pi(\mvl,t;\tilde\kappa^{(u)},\tilde\mu^{(u)}) d u } \\
&\leq \frac{1}{\norm{\mvl}} \sum_{\vt \in \bT_{\cS}(\mvl)} \int_0^1 \abs{\frac{d}{d u}\pr\left(\MBP_{\tilde\mu^{(u)}}(t\tilde{\kappa}^{(u)},\tilde\mu^{(u)}) \in \bT_{\cS}(\mvl)\right)} d u \nonumber\\
&\leq  C \norm{\mvl} (1+t) \int_0^1 \pi(\mvl,t;\tilde\kappa^{(u)},\tilde\mu^{(u)}) d u \nonumber
\end{align} where $C$ is the constant mentioned in the statement of the lemma.\end{proof}
Recall from Assumption \ref{ass:irg} and from the assumption in \eqref{eq:nonrandom} that $\sqrt{n}(\kappa_n-\kappa)$ and $\sqrt{n}(\mu_n-\mu)$ converge to nonrandom constants $\Lambda$ and $\nu$ respectively.
Thus
there exists $L\geq 0$ such that 
\begin{equation}\label{eq:bdonerr}
\sqrt{n}\left(\norm{\kappa_n-\kappa}_{\infty} + \norm{\mu_n-\mu}_\infty\right) \leq L.
\end{equation}

\begin{lemma}
\label{lem: MBP-approx}
   For any $T>0$ and $0\le t\leq T$, there exists $n_0 \in \bN$ and a constant $\tilde{C} = \tilde{C}(L,T)<\infty$ such that for all $\mvl \in \fT\setminus\set{\mv0}$ we have  \begin{align*}
\sqrt{n}\ \left|\pi(\mvl,t;\kappa_n,\mu_n) - \pi(\mvl,t;\kappa,\mu) \right| \leq \tilde{C}\norm{\mvl}\left(\frac{\norm{\mvl}}{\sqrt{n}} +  \pi(\mvl,t;\kappa,\mu)\right),\qquad \text{ for all } n\geq n_0
    \end{align*} 
\end{lemma}
\begin{proof}
Since the entries of $\kappa$ and $\mu$ are strictly positive, there exists $n_0\in \bN$ and $\alpha_* >0 $, $\inf_{x,y} \kappa_n(x,y) \geq \alpha_*$ and $\inf_{x} \mu_n(x) \geq \alpha_*$ for all $n\geq n_0$. Therefore, by Lemma \ref{lem:abs-cont-mbp}, we have \begin{align*}
\left|\pi(\mvl,t;\kappa_n,\mu_n) - \pi(\mvl,t;\kappa,\mu) \right| \leq  C_1 \norm{\mvl}\int_0^{1} \pi(\mvl,t;\kappa_n^{(u)},\mu_n^{(u)}) d u,\qquad \mvl \in \fT\setminus\set{\mv0}, n\geq n_0,
\end{align*} 
where $\kappa_n^{(u)} = u\kappa + (1-u)\kappa_n$ and $\mu_n^{(u)} = u\mu + (1-u)\mu_n$
and $C_1$ is a constant that depends only on $L,T$ and $\alpha_*$. Now using Proposition \ref{prop:ode-mbp-irg} and Lemma \ref{lem:mbp-tv-approx}, the result is immediate. 
\end{proof}
We now complete the proof of Theorem \ref{thm:variance-moment-bound-main} using Lemma \ref{lem: MBP-approx}.
\begin{proof}[Proof of Theorem \ref{thm:variance-moment-bound-main}]
Fix $M,\delta \geq 0$. Let $T = \sup \cI$. Note that the process $\set{n\pi_n(\mv0,t):t\geq 0}$ is in distribution equal to a Poisson process with rate $$\xi_n = \frac{n}{2} \left[\theta_{\kappa_n}(\mu_n,\mu_n) - \frac{1}{n}\sum_{i=1}^K \kappa_n(i,i)\mu_n(i)\right].$$ Furthermore, for $\mvl \in \fT$, we have $\E(X_n(\mvl,t))^2 = n \var(\pi_n(\mvl,t)) + n \left[\E(\pi_n(\mvl,t)) - \pi(\mvl,t)\right]^2.$ Therefore, $$\sup_n\sup_{t\in \cI}\E\left(X_n(\mv0,t) \right)^2 \leq \sup_n\left[ \frac{T\xi_n }{n} + \left(1+2K\norm{\kappa
}\right)^2 T^2L^2  + \frac{KT^2\norm{\kappa_n}_\infty}{n}\right] < \infty.$$

By Corollary \ref{cor:var-inequality} and Theorem \ref{thm:mean-moment-bound-main}, we have \begin{align*}
        \sup_{n \geq 1}\sup_{t\in \cI} \ n \left[\sum_{\mvl \in \fT_{M\log n}\setminus\set{\mv0}}\norm{\mvl}^\delta \var(\pi_n(\mvl,t)) \right] < \infty.
    \end{align*} Therefore, it suffices to show that \begin{align}\label{eqn:variance-sufficient-condition}
\sup_{n \geq 1}\sup_{t\in \cI} \left[ n\left(   \sum_{\mvl \in \fT_{M\log n}\setminus\set{\mv0}}\norm{\mvl}^\delta \left[\E(\pi_n(\mvl,t)) - \pi(\mvl,t)\right]^2 \right)\right] < \infty.
\end{align} 
 By Proposition \ref{prop:ode-mbp-irg}, we have, with $U_n$ as in Lemma \ref{lem:TV-approximation},
\begin{align}\label{eqn:9029}
\left[\E(\pi_n(\mvl,t)) - \pi(\mvl, t)\right]^2 &=\frac{1}{\norm{\mvl}^2} \bigg[\pr(\fF(\cC_n(U_n,t))=\mvl) - \pr(\MBP_{\mu}(t\kappa,\mu) \in \bT_{\cS}(\mvl))\bigg]^2\nonumber\\
     &\leq 2 \sum_{i=1}^2 q_n^{(i)}(\mvl,t),\end{align} where \begin{align*} 
           q_n^{(1)}(\mvl,t) &= \frac{1}{\norm{\mvl}^2}\left[\pr\left(\fF(\cC_n(U_n,t)) =\mvl\right) - \pr(\MBP_{\mu_n}(t\kappa_n,\mu_n)\in \bT_{\cS}(\mvl))\right]^2 \\
           q_n^{(2)}(\mvl,t) &= \left[\pi(\mvl,t;\kappa_n,\mu_n) - \pi(\mvl,t;\kappa,\mu)\right]^2.
       \end{align*} Therefore, to prove \eqref{eqn:variance-sufficient-condition}, from \eqref{eqn:9029}, it enough to show \begin{align}\label{eqn:Q's-sufficient-condition}
     \sup_{n}\sup_{t\in \cI}Q_n^{(i)}(t) < \infty,\qquad i=1,2
 \end{align} where \begin{align*}
     Q_n^{(i)}(t) = n \cdot \bigg[\sum_{\mvl\in \fT_{M\log n}\setminus\set{\mv0}} \norm{\mvl}^{\delta} q_n^{(i)}(\mvl,t)\bigg], \qquad  i=1,2.
 \end{align*} To this end, by Lemma \ref{lem:TV-approximation}, for some constant $C_1 \geq 0$,  we have $q_n^{(1)}(t) \leq C_1 t^2 \norm{\mvl}^4/n^2$, for  all $\mvl \in \fT$, $n \in \bN$ and $t\ge 0$.  Therefore, we have
 \begin{align}\label{eqn:Q_n^1-bound}
     \sup_n \sup_{t\in \cI} \ Q_n^{(1)}(t)\leq C_1 T^2 \sup_n \left[ \frac{\sum_{\mvl \in \fT_{M\log n}} \norm{\mvl}^{\delta+2}}{n}\right] \leq C_1'T^2 \sup_n \left[\frac{(\log n)^{\delta+2+K}}{n}\right] < \infty.
  \end{align} 
  For $q_n^{(2)}$, by  Lemma \ref{lem: MBP-approx}, we have\begin{align*}
       q_n^{(2)}(\mvl, t) &\leq C_2 \left(\frac{\norm{\mvl}^4}{n} + \norm{\mvl}^2 \pi(\mvl,t;\kappa,\mu)^2\right)
    \end{align*} for all $n \geq n_0$ where $n_0$ is as given by Lemma \ref{lem: MBP-approx}. Therefore, we have \begin{align*}
        Q_n^{(2)}(t) &\leq C_2 \left(\frac{1}{n}\sum_{\mvl\in \fT_{M\log n}\setminus\set{\mv0}} \norm{\mvl}^{\delta+4} + \sum_{\mvl\in \fT_{M\log n}\setminus\set{\mv0}} \norm{\mvl}^{\delta+2}\pi(\mvl,t;\kappa,\mu)\right)\\
        &\leq C_2' \left(\frac{(M\log n)^{\delta+K+3}}{n} + \sum_{\mvl\in \fT_{M\log n}\setminus\set{\mv0}} \norm{\mvl}^{\delta+2}\pi(\mvl,t;\kappa,\mu)\right),\qquad n\geq n_0.
    \end{align*}  Hence, by Lemma \ref{lem:all-moments-finite}, we have $\sup_n \sup_{t\in \cI} Q_n^{(2)}(t) < \infty$. This finishes the proof.
\end{proof}

\subsection{Semi-Martingale Representation }
Throughout this section, we fix $M \geq 0$ . 
To reduce notation, we suppress the superscript $M\log n $  and simply write $\mvpi_n(t)$ and $\vX_n(t)$ for $\mvpi_n^{\fT_{M\log n}}(t)$ and $\vX_n^{\fT_{M\log n}}(t)$ respectively. Similarly,  write $\Delta_{\mvl,\mvk}^{\fT_{M\log n}} $ as $\Delta_{\mvl,\mvk}^n$ for all $\mvl,\mvk \in \fT_{M\log n}$, and so on. For fixed $n\geq 1$, by letting $N = M\log n$ in \eqref{eqn:irg-poisson-eqn}, we have \begin{align}\label{eqn:irg-poisson-eqn-biggggggggg}
    &\mvpi_n(t) = \mvpi_n(0)+ \frac{1}{n}\bigg[\sum_{\mvl \in \fT_{M\log n} \setminus \set{\mv0}} \Delta_{\mvl} Y_{\mvl}\left(n\int_{0}^t \varphi_{\mvl}(\mvpi_n(s))ds\right) + e_{\mv0} Y_{\mv0}\left(n \int_0^t \varphi_{\mv0}^*(\mvpi_n(s))ds\right)\\
    &+\sum_{\mvl, \mvk \in \fT_{M\log n} \setminus\set{\mv0}, \mvl<\mvk}\Delta^{n}_{\mvl,\mvk} Y_{\mvl,\mvk}\left(n \int_{0}^t \varphi_{\mvl,\mvk}(\mvpi_n(s))ds\right) +  \sum_{\mvl \in \fT_{M\log n}\setminus \set{\mv0}}\Delta_{\mvl,\mvl} Y_{\mvl,\mvl}\left(n \int_{0}^t \varphi^*_{\mvl,\mvl}(\mvpi_n(s))ds\right)\bigg].\nonumber
\end{align} By adding and subtracting the compensators of the Poisson processes, we have the following semi-martingale decomposition for $\mvpi_n$, 
 \begin{align}\label{eqn:biggggggg-semi-mg}
    \mvpi_n(t) =\mvpi_n(0) + \mvA_{n}(t) + \mvM_{n}(t), 
\end{align} where \begin{align}\label{eqn:fdd-pi-mg-term-infinite}
\mvM_{n}(t) &= \frac{1}{n}\bigg[\sum_{\mvl \in \fT_{M\log n}\setminus\set{\mv0}} \Delta_{\mvl} N_{\mvl}\left(n \int_{0}^t \varphi_{\mvl}(\mvpi_n(s))ds\right) + e_{\mv0} N_{\mv0}\left(n\int_0^t \varphi_{\mv0}^*(\mvpi_n(s))ds\right) \\
&\hspace{-0.7cm}+\sum_{\mvl,\mvk \in\fT_{M\log n}\setminus \set{\mv0}, \mvl<\mvk}\Delta^{n}_{\mvl,\mvk} N_{\mvl,\mvk}\left(n\int_{0}^t \varphi_{\mvl,\mvk}(\mvpi_n(s))ds\right) +\sum_{\mvl \in\fT_{M\log n}\setminus\set{\mv0}} \Delta^{n}_{\mvl,\mvl} N_{\mvl,\mvl}\left(n\int_{0}^t \varphi^*_{\mvl,\mvl}(\mvpi_n(s))ds\right) \bigg]\nonumber
\end{align} where $N_{\mvl,\mvk}(t) = Y_{\mvl,\mvk}(t) - t$ and $N_{\mvl}(t) =Y_{\mvl}(t) -t$ for $t\geq 0$. Therefore, $\mvM_n(t)$ is a square-integrable martingale with respect to a suitable filtration $\set{\cF_n(t):t\geq 0}$ satisfying  $\cF_n(t) \supset \sigma\{\mvpi_n(s): s \le t\}$, $t\ge 0$. Here $\mvA_n(t)$ is given as 
\begin{align}\label{eqn:bdd-process-defn}
    A_n(\mvl,t) = \int_{0}^t F_{\mvl}(\mvpi_n(s),\kappa_n,\mu_n) ds + \int_{0}^t r_{n}(\mvl,s)ds, \qquad \mvl\in \fT_{M\log n}.
\end{align} where $F_{\mvl}$ are as defined in \eqref{eqn:F_l-defn} and, with $\epsilon_n$ as in \eqref{eqn:epsilon_l-defn} and \eqref{eqn:epsilon-0-defn},
\begin{align}\label{eqn:r_n-defn}
r_n(\mv0,t) = \sum_{\mvl \in \fT_{M\log n}} \epsilon_n(\mvl,t),\;\;
    r_n(\mvl,t) = \ind\set{\mvl/2 \in \fT_{M\log n}} \epsilon_n(\mvl/2,t) - 2\epsilon_n(\mvl,t),\; \mvl  \in \fT_{M\log n}\setminus\{\mv0\}.
\end{align}Note that from \eqref{eqn:epsilon_l-bound} and \eqref{eqn:epsilon-0-bound}, for all $t\geq 0$, we have \begin{align}\label{eqn:r_n-bound}
    |r_n(\mvl,t)| \leq (K+2)M^2\norm{\kappa_n}_\infty \frac{(\log n)^2}{n}, \qquad \mvl \in \fT_{M\log n}.
\end{align}
Using the ODEs in \eqref{eqn:ode-irg} for the limit process $\mvpi$, we obtain a semi-martingale representation for the process $\vX_n = \vX_n^{\fT_{M\log n}}$ as \begin{align}\label{eqn:semi-mg-rep-X}
    \vX_n(t) = \vX_n(0) + \mvA^c_n(t) + \mvM_n^c(t),
\end{align} where \begin{align}\label{eqn:centered-defns}
    A_n^c(\mvl,t) = \sqrt{n}\cdot\left(A_n(t) - \int_0^t F_{\mvl}(\mvpi(s),\kappa,\mu)ds\right), \qquad M_n^c(\mvl,t ) = \sqrt{n}\cdot M_n(\mvl,t), \qquad \mvl\in \fT_{M\log n}.
\end{align}

\subsection{Tightness}\label{sec:tightness}
Once more we fix $M\ge 0$ throughout this section.
In this section, we establish tightness of the process of fluctuations $\vX_n^{\fT_{M\log n}}(\cdot) = \vX_n(\cdot)$. This is proved by verifying the Aldous-Kurtz tightness criterion (see \cite{kurtz1981approximation} and  Theorem \ref{thm:AppSemiMartTight} in Appendix \ref{sec:appendix}).

We first focus on establishing tightness in the sub-critical regime. The tightness in the super-critical regime follows from minor modifications to the arguments in the sub-critical regime, which we briefly discuss in the ensuing subsection. 
\subsubsection{Tightness in Sub-critical Regime:} We first prove the tightness of the martingale process $\mvM_n^c$ in Proposition \ref{lem:mg-tightness-sub}, and then the tightness of process $\vX_n$. This will also yield the tightness of $\mvA_n^c$.

The following two inequalities are used multiple times throughout this section. So we state them here in a lemma.
\begin{lem}\label{lem:norm-inequalities}
For any $\nu > 0$,  $\mvx,\mvy \in \bR_+^{\fT}$, and $\fT_0 \subset \fT$, the following hold
\begin{enumeratei}
    \item \begin{align}\label{eqn:product-inequality}
 &\sum_{\mvl \in \fT_0} \norm{\mvl}^{\nu} \left[\sum_{\mvk_1 + \mvk_2 =\mvl} \norm{\mvk_1} \norm{\mvk_2} x_{\mvk_1} y_{\mvk_2}\right] \\
  &\leq 2^{\nu} \left[\left(\sum_{\mvl\in \fT_0}\norm{\mvl}^{\nu+1}x_{\mvl}\right)\cdot \left(\sum_{\mvl\in \fT_0}\norm{\mvl}y_{\mvl}\right) + \left(\sum_{\mvl\in \fT_0}\norm{\mvl}^{\nu+1}y_{\mvl}\right)\cdot \left(\sum_{\mvl\in \fT_0}\norm{\mvl}x_{\mvl}\right)\right].\nonumber\end{align}
    \item There exists a constant $C_{\nu}$ depending on $\nu$.\begin{align}\label{eqn:CS-inequality-norm}
    \left(\sum_{\mvl\in \fT_0} \norm{\mvl}^{\nu}x_{\mvl}\right)^2 \leq C_\nu \left[\sum_{\mvl \in \fT_0} \norm{\mvl}^{2\nu +K+1}x_{\mvl}^2\right]
\end{align} 
\end{enumeratei}
\end{lem}  
Both inequalities are easy to verify. For the first we use the observation that $\norm{\mvk_1+\mvk_2}^{\nu} \le 2^{\nu}(\norm{\mvk_1}^{\nu} + \norm{\mvk_2}^{\nu})$ for
$\mvk_1, \mvk_2 \in \fT$ and the second
 is a direct consequence of the Cauchy-Schwarz inequality.
\begin{prop}\label{lem:mg-tightness-sub}
    For $0< T < t_c$ and $\delta \geq 0$, the collection $\set{ \mvM_n^c}_{n\geq 1}$ is a $\bC$-tight sequence of $\ \bb{D}([0,T]:\ell_{1,\delta})$ random variables.
\end{prop}

\begin{proof}
   Fix $\delta,\epsilon \geq 0$ and $T < t_c$. We start by verifying condition $T_1$ in Theorem \ref{thm:AppSemiMartTight}. 
\vspace{0.2cm}
    
\noindent \textbf{Verifying Condition $\text{T}_1:$} For any $\eta > 0$, consider the set \begin{align}\label{eqn:compact-set}
        \cK_{\eta,\delta} = \set{x\in \ell_{1,\delta}: |x_{\mv0}| \leq \eta \  \text{ and } \ |x_{\mvl}| \leq  \eta \norm{\mvl}^{-(K+\delta+2)} \text{ for } \mvl \in \fT\setminus\set{\mv0}}.
    \end{align}
    Since $\abs{\set{\mvl \in \fT:\norm{\mvl} = m}} \leq (m+1)^{K-1}$ for all $m\in \bN$, the set $\cK_{\eta,\delta}$ is a compact subset of $\ell_{1,\delta}$. Let $\delta' = K+\delta+2$. For any $t < T$, by union bound, we have
\begin{align*}
    \pr\left(\mvM_n^{c}(t)\not\in \cK_{\eta,\delta}\right) &\leq \pr\left(\abs{M_n^{c}(\mv0,t)}\geq \eta\right) + \sum\limits_{\mvl \in \fT_{M\log n}\setminus\set{\mv0}} \pr\left(\abs{M_n^{c}(\mvl,t)}\geq \eta \norm{\mvl}^{-\delta'}\right).
\end{align*}
Observe that from \eqref{eqn:fdd-pi-mg-term-infinite} we have
\begin{align}\label{eqn:21319-qv-m_0}
\la M^{c}_n(\mv0) \ra_t &= \int_0^t \bigg[\sum_{\mvl, \mvk \in \fT_{M\log n}\setminus\set{\mv0}, \mvl<\mvk}\varphi_{\mvl,\mvk}(\mvpi_n(s)) + \sum_{\mvl \in \fT_{M\log n}\setminus\set{\mv0}}\varphi_{\mvl}(\mvpi_n(s)) \nonumber \\
&\hspace{2cm}+ \sum_{\mvl\in \fT_{M\log n}\setminus\set{\mv0}}\varphi_{\mvl,\mvl}^*(\mvpi_n(s)) + \varphi_{\mv0}^{*}(\mvpi_n(s))\bigg]ds \nonumber \\
&=\frac{\theta_{\kappa_n}(\mu_n,\mu_n) t}{2} + \int_0^t r_n(\mv0,s)ds.
\end{align} Furthermore, by \eqref{eqn:r_n-bound}, we have $\int_0^t |r_n(\mv0,s)|ds \leq 4\norm{\kappa_{n}}_\infty M^2 t (\log n)^2/n$. Therefore, by Markov's inequality, for suitably large choice of $\eta$, we have $\pr\left(\abs{M_n^{c}(\mv0,t)}\geq \eta\right) \leq \epsilon$. Similarly, we also have
    \begin{align}\label{eqn:T_1-condition-bound-mg}
       \sum\limits_{\mvl \in \fT_{M\log n}\setminus\set{\mv0}} \pr\left(\abs{M_n^{c}(\mvl,t)}\geq \eta \norm{\mvl}^{-\delta'}\right) &\leq \frac{1}{\eta^2} \cdot \left[ \sum_{\mvl \in \fT_{M\log n}\setminus\set{\mv0}} \norm{\mvl}^{2\delta'} \E\langle M_n^{c}(\mvl,t) \rangle\right],
    \end{align}
\noindent
where for $\mvl\in \fT_{M\log n}\setminus\set{\mv0}$\begin{align}\label{eqn:qv-centered-mg-l}
    \la M_n^c(\mvl) \ra_t &= \int_0^t \bigg[\sum_{\mvk_1,\mvk_2 \in \fT_{M\log n}\setminus\set{\mv0}, \mvk_1<\mvk_2} \ind\set{\mvk_1+\mvk_2 = \mvl} \varphi_{\mvk_1,\mvk_2}(\mvpi_n(s))  + \sum_{\mvk,\mvl\in \fT_{M\log n}\setminus\set{\mv0}, \mvk\neq \mvl} \varphi_{\mvl,\mvk}(\mvpi_n(s))   \nonumber\\
    &\hspace{1cm}+4 \varphi_{\mvl,\mvl}^*(\mvpi_n(s))  + \ind\set{\mvl/2 \in \fT_{M\log n}} \varphi_{\mvl/2,\mvl/2}^*(\mvpi_n(s)) + \varphi_{\mvl}(\mvpi_n(s))\bigg]ds \nonumber\\
    &=\int_{0}^t  \frac{1}{2}\left[\sum_{\mvk_1+\mvk_2 = \mvl}\pi_n(\mvk_1,s)\pi_n(\mvk_2,s)\theta_{\kappa_n}(\mvk_1,\mvk_2)\right]ds + \int_{0}^t \pi_n(\mvl,s)\theta_{\kappa_n}(\mvl,\mu_n) ds \nonumber \\
    & \hspace{2cm}+ \int_0^t \pi_n(\mvl,s)^2 \theta_{\kappa_n}(\mvl,\mvl)ds + \int_0^t \tilde{r}_n(\mvl,s)ds, \qquad \mvl \in \fT_{M\log n}\setminus\set{\mv0}\end{align} where $\tilde{r}_{n}(\mvl,s) = 4\epsilon_n(\mvl,s) + \ind\set{\mvl \text{ is even}}\epsilon_n(\mvl/2,s)$, with $\epsilon_n$'s as defined in \eqref{eqn:epsilon_l-defn}.  From \eqref{eqn:epsilon_l-bound}, we see that $\tilde{r}_{n}(\mvl,s) \leq \frac{5}{n} \norm{\kappa_n}_\infty  \norm{\mvl}.$ Hence, for all $t\leq T$ and $\mvl\in \fT_{M\log n}\setminus\set{\mv0}$, we have \begin{align}\label{eqn:QV-mg-bound}
       \la M_n^{c}(\mvl) \ra_t &\leq \norm{\kappa_n}_\infty \left[\int_0^t \Big[\sum_{\mvk_1+\mvk_2 = \mvl}\norm{\mvk_1}\norm{\mvk_2}\pi_n(\mvk_1,s)\pi_n(\mvk_2,s)\Big]ds + 2\int_0^t \norm{\mvl}\pi_n(\mvl,s)ds + \frac{5\norm{\mvl}T}{n} \right].
    \end{align}

\noindent Now, using  Lemma \ref{lem:norm-inequalities}(i), we have \begin{align}\label{eqn:mg-sub-first-term-final-bound}
\sum_{\mvl \in \fT_{M\log n}\setminus\set{\mv0}} \norm{\mvl}^{2\delta'} \left[\sum_{\mvk_1+\mvk_2 = \mvl} \norm{\mvk_1}\norm{\mvk_2}\pi_n(\mvk_1,t)\pi_n(\mvk_2,t)\right] \leq 2^{2\delta'+1}\left[\sum_{\mvl \in \fT_{M\log n}\setminus\set{\mv0}} \norm{\mvl}^{2\delta'+1}\pi_n(\mvl,t)\right].
\end{align} The last inequality follows from the observation that $\sum_{\mvl\in \fT_{M\log n}}\norm{\mvl}\pi_n(\mvl,t)$ is the proportion of vertices in components of size at most $M\log n$, and is therefore bounded by $1$. Using the bound in \eqref{eqn:mg-sub-first-term-final-bound} and \eqref{eqn:QV-mg-bound}, we get\begin{align}\label{eqn:QV-bound-mg-final}
    \sum_{\mvl\in \fT_{M\log n}\setminus\set{\mv0}} \norm{\mvl}^{2\delta'} \langle M_n^{c}(\mvl)\rangle_t &\leq 2^{2\delta'+2}\norm{\kappa_n}_\infty \int_{0}^t  \sum_{\mvl \in \fT_{M\log n}\setminus\set{\mv0}} \norm{\mvl}^{2\delta'+1}\pi_n(\mvl,s)ds \nonumber\\
    &\qquad + \frac{5T\norm{\kappa_n}_\infty }{n}\sum_{\mvl \in \fT_{M\log n}\setminus\set{\mv0}} \norm{\mvl}^{2\delta'+1}\nonumber\\
    &\leq C_1 \left[ \int_{0}^t  \sum_{\mvl \in \fT_{M\log n}\setminus\set{\mv0}} \norm{\mvl}^{2\delta'+1}\pi_n(\mvl,s)ds + \frac{T(\log n)^{2\delta' + K}}{n}\right],
\end{align} 
for some constant $C_1 \geq 0$. Taking expectations on both sides in the above inequality and using the bound in  \eqref{eqn:T_1-condition-bound-mg}, together with Theorem \ref{thm:mean-moment-bound-main}, we conclude that the condition $\text{T}_1$ is satisfied.\\ 

\noindent\textbf{Verifying Condition A:} We now check Condition A in Theorem \ref{thm:AppSemiMartTight} Fix $\epsilon,\eta >0$. Consider a collection  $\set{\tau_n}_{n\in \bN}$ such that, for each $n \in \bN$, $\tau_n$  is a $\cF_n(t)$-stopping time satisfying $\tau_n \le  T-\epsilon$ a.s. By Markov's inequality,  for any $n\geq 1$ and $\beta \leq \epsilon \wedge 1$, \begin{align}
    &\pr\left(\norm{\mvM_n^{c}(\tau_n + \beta) - \mvM_n^{c}(\tau_n)}_{1,\delta} \geq \eta\right)\leq \frac{1}{\eta}\E\left(\norm{\mvM_n^{c}(\tau_n + \beta) - \mvM_n^{c}(\tau_n)}_{1,\delta}\right) \nonumber\\ 
    &= \frac{1}{\eta} \left[\E\left(\left|M_n^{c}(\mv0,\tau_n+\beta) - M_n^c(\mv0,\tau_n)\right|\right) + \sum_{\mvl \in \fT_{M\log n}\setminus\set{\mv0}}\norm{\mvl}^\delta \E\left(\left|M_n^{c}(\mvl,\tau_n + \beta) - M_n^{c}(\mvl,\tau_n)\right|\right)\right]\nonumber\\
    &\leq \frac{1}{\eta} \bigg[\bigg(\E(\la M_n^{c}(\mv0)\ra_{\tau_n+\beta} - \la M_n^c(\mv0)\ra_{\tau_n})\bigg)^{1/2}\nonumber\\ &\hspace{1.5cm}+ \sum_{\mvl \in \fT_{M\log n}\setminus\set{\mv0}}\norm{\mvl}^\delta \bigg(\E(\la M_n^{c}(\mvl)\ra_{\tau_n + \beta} - \la M_n^{c}(\mvl)\ra_{\tau_n})\bigg)^{1/2} \bigg] \label{eq:723n}\end{align}

Note that by \eqref{eqn:21319-qv-m_0} we have \begin{align}\label{eqn:cond-A-mv0}
 \E(\la M_n^{c}(\mv0)\ra_{\tau_n+\beta} - \la M_n^c(\mv0)\ra_{\tau_n})  \leq \beta  \left(\theta_{\kappa_n}(\mu_n,\mu_n) + \norm{\kappa_n}_\infty M^2 \frac{(\log n)^2}{n}\right).
\end{align}

Therefore, using (ii) of Lemma \ref{lem:norm-inequalities}, for $\delta' = \delta+ K+2$, we have
\begin{align}\label{eqn:9378}
    \sum_{\mvl \in \fT_{M\log n}\setminus\set{\mv0}}\norm{\mvl}^\delta \bigg(&\E(\la M_n^{c}(\mvl)\ra_{\tau_n + \beta} - \la M_n^{c}(\mvl)\ra_{\tau_n})\bigg)^{1/2} \nonumber \\ &\leq  C_\delta\left[\sum_{\mvl \in \fT_{M\log n}\setminus\set{\mv0}}\norm{\mvl}^{2\delta'}
    \E(\la M_n^{c}(\mvl)\ra_{\tau_n + \beta} - \la M_n^{c}(\mvl)\ra_{\tau_n})\right]^{1/2}
\end{align}for some constant $C_\delta \geq 0$. Therefore, to verify Condition $A$, one needs to suitably estimate the term in the square braces above.

\noindent To this end, using the arguments similar to those to obtain the bounds in \eqref{eqn:QV-mg-bound}-\eqref{eqn:QV-bound-mg-final}, we have \begin{align}\label{eqn: QV-stopping time-final-bound}
    \sum_{\mvl \in \fT_{M\log n}\setminus\set{\mv0}}\norm{\mvl}^{2\delta'}& 
    (\la M^c_n(\mvl)\ra_{\tau_n + \beta} - \la M^c_n(\mvl)\ra_{\tau_n}) \nonumber\\
    &\leq C_3 \int_{\tau_n}^{\tau_n+\beta}  \left(\sum_{\mvl \in \fT_{M\log n}\setminus\set{\mv0}} \norm{\mvl}^{2\delta'+1}\pi_n(\mvl,s)\right)ds + \frac{C_3'\beta(\log n)^{2\delta'+K}}{n}\nonumber\\
    &\leq C_3\sqrt{\beta} \left[\int_{0}^{T}\left(\sum_{\mvl \in \fT_{M\log n}\setminus\set{\mv0}} \norm{\mvl}^{2\delta'+1}\pi_n(\mvl,s)\right)^2ds\right]^\frac{1}{2}+ \frac{C_3'\beta(\log n)^{2\delta'+K}}{n}.
\end{align} 

Next, we have \begin{align*}
&\E\left[\int_{0}^{T}\left(\sum_{\mvl \in \fT_{M\log n}\setminus\set{\mv0}} \norm{\mvl}^{2\delta'+1}\pi_n(\mvl,s)\right)^2ds\right]^\frac{1}{2} \\
&\leq \left[\E\left(\int_{0}^{T}\left(\sum_{\mvl \in \fT_{M\log n}\setminus\set{\mv0}} \norm{\mvl}^{2\delta'+1}\pi_n(\mvl,s)\right)^2ds\right)\right]^\frac{1}{2}  \leq \sqrt{T} \left[\sup_{t\leq T} \E\left(\sum_{\mvl \in \fT_{M\log n}\setminus\set{\mv0}} \norm{\mvl}^{2\delta'+1}\pi_n(\mvl,t)\right)^2\right]^\frac{1}{2}.
\end{align*}  

Using  Lemma \ref{lem:norm-inequalities}(ii)
$$
\left(\sum_{\mvl \in \fT_{M\log n}\setminus\set{\mv0}} \norm{\mvl}^{2\delta'+1}\pi_n(\mvl,s)\right)^2 \leq  C_4 \left(\sum_{\mvl \in \fT_{M\log n}\setminus\set{\mv0}}\norm{\mvl}^{4\delta'+ K+3}\pi_n(\mvl,s) \right),$$ 
for some constant $C_4$ depending only on $\delta$. Thus,
\begin{align*}
\E\left[\int_{0}^{T}\left(\sum_{\mvl \in \fT_{M\log n}\setminus\{\mv0\}} \norm{\mvl}^{2\delta'+1}\pi_n(\mvl,s)\right)^2ds\right]^\frac{1}{2} \leq \sqrt{C_4 T} \left[\sup_{t\leq T}\sum_{\mvl \in \fT_{M\log n}\setminus\{\mv0\}}\norm{\mvl}^{4\delta'+ K+3}\E(\pi_n(\mvl,t))\right]^{\frac{1}{2}}.
\end{align*} 
Hence, for some constant $C_5 \geq 0$, we have \begin{align}\label{eqn:202421}
    \sum_{\mvl \in \fT_{M\log n}\setminus\set{\mv0}}\norm{\mvl}^\delta &\bigg(\E\la M_n^{c}(\mvl,\tau_n + \beta) - M_n^{c}(\mvl,\tau_n)\ra\bigg)^{1/2} \nonumber \\
    &\leq 
    C_5\beta^{1/4} \left[\Big(\sup_{t\leq T}\sum_{\mvl \in \fT_{M\log n}\setminus\{\mv0\}}\norm{\mvl}^{4\delta'+ K+3}\E(\pi_n(\mvl,t))\Big)^{1/2} + \frac{(\log n)^{2\delta'+K}}{n}\right]^{1/2}
\end{align} 
From \eqref{eq:723n}, \eqref{eqn:cond-A-mv0}, \eqref{eqn:9378}, and \eqref{eqn:202421} and applying Theorem \ref{thm:mean-moment-bound-main}, we have that for a sufficiently small $\zeta > 0$ \begin{align*}
    \sup_{n}\sup_{\beta \leq \zeta}\pr\left(\norm{\mvM_n^{c}(\tau_n + \beta) - \mvM_n^{c}(\tau_n)}_{1,\delta} \geq \eta \right) \leq \epsilon.
\end{align*} 
This shows that the Condition $A$ in Theorem \ref{thm:AppSemiMartTight} holds and completes the proof of the lemma.
The asserted $\bC$-tightness in the statement follows since $\|\Delta M_n^c\|_{1,\delta} \le C n^{-1/2}$, a.s., for some $C>0$.
\end{proof}

\begin{rem} The above calculations, together with and Doob's $\cL^2$-inequality, also show that, for all $\nu\ge 0$, \begin{align}\label{eqn:doob-mg-sub}
   \sup_{n \in \bN}\E\left( \left[\sup_{t\leq T}(M^c_n(\mv0,t))^2 + \sum_{\mvl\in \fT_{M\log n}\setminus\set{\mv0}}\norm{\mvl}^{\nu} \sup_{t\leq T}(M_n^c(\mvl,t))^2\right] \right) < \infty .
\end{align}
\end{rem}

\noindent  We next focus on tightness of $X_n$. For this we will use  the second-moment bound established in Theorem \ref{thm:variance-moment-bound-main} and the estimate in the following lemma.
\begin{lem}
\label{lem:705}
Fix $T < t_c$, $\delta\ge 0$ and $M\ge 0$. There 
 exists a  $C \geq 0$  such that, for every collection  $\set{\tau_n}_{n\in \bN}$, such that for each $n \in \bN$, $\tau_n$  is a $\cF_n(t)$-stopping time satisfying $\tau_n \le  T-\epsilon$ a.s., and $\beta \le \epsilon \wedge 1$, 
    \begin{align}\label{lem:sufficient-condition-A}
      \sup_{n\geq 1} \sqrt{n} \cdot \left[\sum_{\mvl \in \fT_{M\log n}\setminus\set{\mv0}} \norm{\mvl}^{\delta} \E\left(\int_{\tau_n}^{\tau_n+\beta}  \left|F_{\mvl}(\mvpi_n(s),\kappa_n,\mu_n) - F_{\mvl}(\mvpi(s),\kappa,\mu)\right|ds \right)\right] \leq C \beta^{1/4}.
 \end{align} 
\end{lem}

\begin{proof}
Using the expressions for $F_{\mvl}$ for $\mvl \in \fT_{M\log n}\setminus\set{\mv0}$ from \eqref{eqn:F_l-defn}, we have for some $C'\ge 0$, and all $t\ge 0$ and $n \in \bN$,
\begin{align*}
   &\sqrt{n}\left|F_{\mvl}(\mvpi_n(t),\kappa_n,\mu_n) - F_{\mvl}(\mvpi(t),\kappa,\mu)\right|\\
    &\leq C'  \bigg[\sum_{\mvk_1+\mvk_2 = \mvl} \norm{\mvk_1}\norm{\mvk_2}\bigg(\pi_n(\mvk_1,t)|X_n(\mvk_2,t)| + |X_n(\mvk_1,t)|\pi(\mvk_2,t) + \pi(\mvk_1,t)\pi(\mvk_2,t)\bigg) \\
    &\hspace{8cm}+ \norm{\mvl} \left(|X_n(\mvl,t)| + \pi(\mvl,t)\right)  \bigg].
\end{align*} where $C' = 2 \max\left(\sup_{n} \norm{\kappa_n}, L\right)$. Therefore, using Lemma \ref{lem:all-moments-finite} and Lemma \ref{lem:norm-inequalities}(i), for all $t \leq T$, 
\begin{align*}
    \sqrt{n} \bigg[\sum_{\mvl\in \fT_{M\log n}\setminus\set{\mv0}}\norm{\mvl}^{\delta}\left|F_{\mvl}(\mvpi_n(t),\kappa_n,\mu_n) - F_{\mvl}(\mvpi(t),\kappa,\mu)\right|\bigg] \leq C'' U_n(t)  V_n(t)
\end{align*} 
for some constant $C''$ depending on $M,T$ and $L$, where \begin{align}\label{eqn:defn-cond-A-UV}
U_n(t) := 1 + \sum_{\mvl \in \fT_{M\log n}\setminus\set{\mv0}} \norm{\mvl}^{\delta+1}|\pi_n(\mvl,t)|, \qquad V_n(t) := 1+ \sum_{\mvl \in \fT_{M\log n}\setminus\set{\mv0}} \norm{\mvl}^{\delta+1}|X_n(\mvl,t)| 
\end{align}

\noindent By Lemma \ref{lem:norm-inequalities}(ii) and Theorem \ref{thm:variance-moment-bound-main}, we have $\sup_n \sup_{t\leq T} (\E V_n^2(t))^{1/2} = C_1< \infty$. Thus, for any $A \geq 0$, we have \begin{align}\label{eqn:head12302}
    \sup_n \E\left(\int_{\tau_n}^{\tau_n+\beta}U_n(s)V_n(s) \ind\set{U_n(s)\leq A}ds\right) \leq  A\sup_n \E\left(\int_{\tau_n}^{\tau_n+\beta}V_n(s) ds \right) \leq C_1\sqrt{T} A\sqrt{\beta}.
\end{align}  Also, we have 
\begin{align}\label{eqn:tail90293}
  \E\left(\int_{\tau_n}^{\tau_n+\beta} U_n(s)V_n(s) \ind\set{U_n(s)\geq A}ds\right) &\leq \frac{1}{A} \E\left(\int_{\tau_n}^{\tau_n+\beta} U^2_n(s)V_n(s) \, ds \right)\nonumber\\
  &\leq \frac{T}{A}\sup_{t\leq T}\E\left(U_n^2(t) V_n(t)\right) \nonumber\\
  &\leq \frac{T}{A}\sup_{t\leq T}\left[\E U_n^4(t)\right]^{1/2}  \sup_{t\leq T} \left[\E V_n^2(t)\right]^{1/2}.
\end{align} 
By applying  Lemma \ref{lem:norm-inequalities}(ii) twice and using Theorems \ref{thm:mean-moment-bound-main} and \ref{thm:variance-moment-bound-main}, we have $\sup_n \sup_{t\leq T} \E U_n^4(t) = C_2 < \infty$. Therefore, combining \eqref{eqn:head12302} and \eqref{eqn:tail90293}, we have 
\begin{align*}
    &\sup_n \sqrt{n} \bigg[\sum_{\mvl\in \fT_{M\log n}\setminus\set{\mv0}}\norm{\mvl}^{\delta}\E\left(\int_{\tau_n}^{\tau_n+\beta} \left|F_{\mvl}(\mvpi_n(s),\kappa_n,\mu_n) - F_{\mvl}(\mvpi(s),\kappa,\mu)\right| ds\right)\bigg] \\
    &\leq C''(C_1\sqrt{T} A \sqrt{\beta} + \frac{1}{A} T C_2^{1/2}C_1).
\end{align*} Choosing $A = \beta^{-1/4}$, we get the desired result. This finishes the proof of the lemma.\end{proof}

\begin{prop}\label{lem:tightness-x-sub}
     For any $0< T < t_c$ and $\delta >0$, the collection $\set{\vX_n}_{n\geq 1}$ is a $\bC$-tight sequence in $ \bb{D}([0,T]:\ell_{1,\delta})$.
\end{prop}

\begin{proof}
Similar to the proof of Proposition  \ref{lem:mg-tightness-sub}, we will verify the two conditions of Theorem \ref{thm:AppSemiMartTight}. Fix $\delta \geq 0$ and $T < t_c$. 
Recall $L$ from \eqref{eq:bdonerr}.

\noindent \textbf{Verifying Condition $\mathbf{T_1}$:} Using union bound and Markov inequality as in the proof of Proposition \ref{lem:mg-x-tightness-super}, we have, for $\eta>0$ and
$K_{\eta,\delta}$ as in \eqref{eqn:compact-set},
\begin{align}
\pr\left(\vX_n(t)\not\in K_{\eta,\delta}\right) &\leq \frac{1}{\eta^2} \left[\E\left(X_n(\mv0,t)\right)^2 + \sum_{\mvl \in \fT_{M\log n}\setminus\set{\mv0
}} \norm{\mvl}^{2\delta'} \E\left(X_n(\mvl,t)\right)^2\right].\label{eq:726n}
\end{align} 

 Applying Theorem \ref{thm:variance-moment-bound-main}, we now conclude that Condition $T_1$ is satisfied.

\noindent\textbf{Verifying Condition A:} 
Since by Proposition \ref{lem:mg-tightness-sub} the collection $\set{\mvM_n^c}_{n\geq 1}$ satisfies Condition $A$ , in view of \eqref{eqn:semi-mg-rep-X}, it is enough to verify that Condition $A$ is satisfied by $\mvA^c_n$. Fix $\epsilon,\eta > 0$. Also, fix a collection  $\set{\tau_n}_{n\in \bN}$ such that 
$\tau_n$ is a $\cF_n(t)$-stopping times  bounded by $T-\epsilon$ a.s. for all $n$.
 By Markov's inequality, we have 
\begin{align}
    &\pr\left(\norm{\mvA_n^{c}(\tau_n + \beta) - \mvA_n^{c}(\tau_n)}_{1,\delta} \geq \eta\right) \nonumber\\ 
    &\leq \frac{1}{\eta} \left[\E\left(\left|A_n^c(\mv0,\tau_n+\beta) - A_n^c(\mv0,\tau_n)\right|\right)+ \sum_{\mvl \in \fT_{M\log n}\setminus\set{\mv0}}\norm{\mvl}^\delta \E\bigg(|A_n^c(\mvl,
\tau_n + \beta) - A_n^c(\mvl,\tau_n)|\bigg)\right].\label{eq:956}
\end{align} Note that from \eqref{eqn:centered-defns} and \eqref{eqn:bdd-process-defn}, for all $\mvl \in \fT_{M\log n}$,  
\begin{align}
    &\left|A_n^c(\mvl,\tau_n+\beta) -  A^c_n(\mvl,\tau_n)\right| \nonumber\\
    &\leq \sqrt{n} \left[\int\limits_{\tau_n}^{\tau_n + \beta} |F_{\mvl}(\mvpi_n(s),\kappa_n,\mu_n) - F_{\mvl}(\mvpi(s),\kappa,\mu)|ds\right] + \sqrt{n} \left[\int\limits_{\tau_n}^{\tau_n + \beta} |r_n(\mvl,s)|ds\right].\label{eq:1001}
\end{align} 
By \eqref{eqn:r_n-bound}, $|r_n(\mvl,t)| \leq (K+2)M^2 \norm{\kappa_n}_\infty (\log n)^2/n$ for all $\mvl\in \fT_{M\log n}$ and $t\geq 0$. Therefore, \begin{align}\label{eqn:error-term-condition-A-X}
  \sqrt{n} \left[\int_{\tau_n}^{\tau_n+\beta} |r_n(\mv0,s)|ds+ \sum_{\mvl \in \fT_{M\log n}\setminus \{\mv0\}} \norm{\mvl}^{\delta}  \int_{\tau_n}^{\tau_n+\beta} |r_n(\mvl,s)|ds\right] \leq \frac{C\beta(\log n)^{\delta+K+1}}{\sqrt{n}}
 \end{align}for some constant $C \geq 0$. 
 Further, by \eqref{eqn:F_l-defn},
 \begin{align}\label{eqn:zero-term-condition-A-X}
\sqrt{n} \E\left(\left|F_{\mv0}(\pi_n(t),\kappa_n,\mu_n) - F_{\mv0}(\pi(t),\kappa,\mu)\right|\right) \leq  L(1+2K\norm{\kappa}_\infty).
\end{align} 

\noindent  Combining \eqref{eqn:zero-term-condition-A-X} and \eqref{eqn:error-term-condition-A-X} with Lemma \ref{lem:705}  and \eqref{eq:956},  \eqref{eq:1001},  completes the proof.
\end{proof}

\begin{rem} The above calculations  also show that for all $\nu \geq 0$,  \begin{align}\label{eqn:doob-bdd-pr-sub}
   \sup_{n\in \bN} \E\left(\left[\sup_{t\leq T} |A^{c}_n(\mv0,t)| + \sum_{\mvl\in \fT_{M\log n}\setminus\set{\mv0}}\norm{\mvl}^{\nu} \sup_{t\leq T}|A_n^c(\mvl,t)|\right] \right) < \infty .
\end{align} 
\end{rem}

Combining  Proposition \ref{lem:mg-tightness-sub} and Proposition \ref{lem:tightness-x-sub}, we have that the collection of random variables $\set{X_n,M_n^c,A_n^c}_{n\geq 1}$ is $\bC$-tight in $\bD([0,T]:\ell_{1,\delta}\times\ell_{1,\delta}\times\ell_{1,\delta} )$  for all $T < t_c$. This establishes the desired tightness in the sub-critical regime.

\subsubsection{Tightness in the super-critical regime:}
In this section, we briefly discuss the minor modifications to the above arguments needed to establish tightness in the super-critical regime. Fix $t_c < T_1 < T_2$. From the semi-martingale representation of $\mvpi_n(t)$ in \eqref{eqn:biggggggg-semi-mg}, for all  $t\in [T_1,T_2]$, we  have 
\begin{align}\label{eqn:semi-mg-rep-super}
   \mvpi_n(t) = \mvpi_n(T_1) + \tilde{\mvA\,}_n(t) + \tilde{\mvM}_n(t),
\end{align} 
where
\begin{align}\label{eqn:super-critical-sm-rep-defn}
    \tilde{\mvA\,}_n(t) = \mvA_n(t)-\mvA_n(T_1),\qquad \tilde{\mvM}_n(t) = \mvM_n(t)-\mvM_n(T_1).
\end{align}

    The processes $\tilde{\mvA\,}_n$ and $\tilde{\mvM}_n$ depend on $T_1$, but we suppress $T_1$ in the notation.

Similar to the sub-critical regime, for $t \in [T_1,T_2]$, we have\begin{align}\label{eqn:semi-mg-rep-x-super}
    \vX_n(t) = \vX_n(T_1) + \tilde{\mvA\,}^c_n(t) +  \tilde{\mvM}^c_n(t),& \qquad \text{where} \\
    \tilde{A}^c_n(\mvl,t) = \sqrt{n}\left[\tilde{A}_n(\mvl,t) - \int_{T_1}^t F_{\mvl}(\mvpi(s),\kappa,\mu)ds\right],& \hspace{0.5cm}
    \tilde{M}_n^c(\mvl,t) = \sqrt{n}\tilde{M}_n(\mvl,t),  \hspace{0.5cm} \mvl\in \fT_{M\log n}. \label{eqn:centerd-process-super}
\end{align}

As an immediate consequence of Theorem \ref{thm:variance-moment-bound-main} and the inequality in \eqref{eq:726n} we have the following pointwise tightness result.
\begin{prop}\label{lem:pointwise-tightness-super}
    For every $t \neq t_c$ and $\delta \geq 0$, the collection of random variables $\set{\vX_n(t)}_{n\geq 1}$ is a tight sequence in $\ell_{1,\delta}$.
\end{prop}

The above proposition, along with Theorem \ref{thm:fdd-fclt}, implies that for any fixed $t\neq t_c$ the sequence  $\vX_n(t)$ as a collection of  $\ell_{1,\delta}$ valued random variables converges in distribution to a random variable $\vX(t)$ whose finite dimensional distribution is given by  Proposition \ref{prop:mean-variance}(iii). Furthermore, we have the following result.

\begin{prop}\label{lem:mg-x-tightness-super}
     For $t_c < T_1 < T_2$ as above and $\delta \geq 0$, the collection of random variables $\set{\vX_n,\tilde{\mvM}_n^c}_{n\geq 1}$ is a $\bC$-tight sequence in $\bD([T_1,T_2]:\ell_{1,\delta}\times\ell_{1,\delta} )$.
\end{prop}

The proof of Proposition \ref{lem:mg-x-tightness-super} follows using the decomposition in \eqref{eqn:semi-mg-rep-x-super} together with similar arguments as in Proposition \ref{lem:mg-tightness-sub} and Proposition \ref{lem:tightness-x-sub}, and the tightness result in Proposition \ref{lem:pointwise-tightness-super}. 
The main observation here is that the estimates \eqref{eqn:epsilon_l-bound} and \eqref{eqn:r_n-bound}  hold for all $t\ge 0$ and that Theorem \ref{thm:mean-moment-bound-main}, Lemma \ref{lem:all-moments-finite}, and Theorem \ref{thm:variance-moment-bound-main}, which were the main ingredients of the proof in the sub-critical regime, can be applied with $\cI = [T_1, T_2]$.

We omit the details.

\subsection{Convergence in \texorpdfstring{$\ell_{1,\delta}$}{l1d}}\label{subsec:convergence}
The previous section proved tightness of $\vX_n$ as well as associated processes in the sub and supercritical regime.  In this section, we show uniqueness (in distribution) of limit points thus completing the proof of the Theorem. 

Recall the space $\tilde\ell_{1}$ and $\ell_2$ defined in \eqref{eqn:l1-subspace} and \eqref{eqn:l_2,delta-space} respectively. The Banach space $\ell_{1,\delta}$ is contained in the Hilbert space $\ell_{2}$ for any $\delta \geq 0$. 
Recall the operators $\Gamma(t)$ and $\Phi(t)$ defined in \eqref{eqn:sde-gamma-defn} and \eqref{eqn:sde-phi-defn} respectively; and the operator $G(t)$ defined as the nonnegative square-root of  $\Phi(t)$. The next lemma collects properties of these operators which will be used in showing the well-definedness of the SDE in \eqref{eqn:sde-vector}.

\begin{lem}\label{lem:operator-properties} The following properties of above operators hold. \begin{enumeratei}
        \item For $\mvx\in \tilde{\ell}_1$ and for $t \neq t_c$, we have $\Gamma(t)(\mvx)$ is an element of $\tilde{\ell}_1$.
        \item For each $t \neq t_c$,   $\Phi(t)$ is a well defined map on $\ell_2$. Furthermore, it is a non-negative symmetric bounded operator on $\ell_2$.
        \item Furthermore, if $\cI$ is a compact interval not containing $t_c$, $$\sup_{t\in \cI} \left[\la e_{\mv0}, \Phi(t) e_{\mv0} \ra + \sum_{\mvl\in \fT\setminus\set{\mv0}} \norm{\mvl}^{\delta} \la e_{\mvl}, \Phi(t) e_{\mvl} \ra \right] < \infty. $$ 
        In particular, $\Phi(t)$ is a trace class operator on $\ell_2$ for all $t \in \cI$,
        $\sup_{t \in \cI} \|\Phi(t)\|_{\mbox{\tiny{tr}}}<\infty$ and $\sup_{t \in \cI} \|G(t)\|_{\mbox{\tiny{HS}}}<\infty$, where $\|\cdot\|_{\mbox{\tiny{tr}}}$ and $\|\cdot\|_{\mbox{\tiny{HS}}}$ denote the trace class and Hilbert-Schmidt norms, respectively.
    \end{enumeratei}
\end{lem}

\begin{proof}
    (i) Fix $t\neq t_c$, and $\mvx \in \tilde{\ell}_{1}$. Let $\mvy = \Gamma(t)(\mvx)$. For all $\mvl\in \fT$, we have \begin{align*}
        |y_{\mv0}| = 0,\qquad \abs{y_{\mvl}} \leq \norm{\kappa}_\infty  \left[\sum_{\mvk_1+\mvk_2 = \mvl}\pi(\mvk_1,t)|x_{\mvk_2}|\norm{\mvk_1}\norm{\mvk_2} + |x_{\mvl}| \norm{\mvl}\right], \hspace{0.2cm} \mvl \neq \mv0.
    \end{align*}  By Lemma \ref{lem:all-moments-finite}, we have $\norm{\mvpi(t)}_{1,\nu} < \infty$ for all $\nu \geq 0$. Therefore, by (i) of Lemma \ref{lem:norm-inequalities}, we have \begin{align*}
        \norm{\mvy}_{1,\delta} &=\sum_{\mvl \in \fT\setminus\set{\mv0}} \norm{\mvl}^{\delta} |y_{\mvl}| \leq   \norm{\kappa}_\infty  \left[\norm{\mvx}_{1,\delta+1} + 2^{\delta} \norm{\mvx}_{1,1}\norm{\mvpi(t)}_{1,\delta+1} + 2^{\delta} \norm{\mvpi(t)}_{1,1}\norm{\mvx}_{1,\delta+1}\right].
        \end{align*} Therefore, $\mvy \in \tilde{\ell}_1$. This proves (i).

     (ii) Fix $t \neq t_c$. Denote the first (resp. second) infinite sum in \eqref{eqn:sde-phi-defn} by $\Phi_1(t)$ (resp.  $\Phi_2(t)$). Note that, for each $\mvl \in \fT$,
     $\Phi_i(t) e_{\mvl}$ is well defined (as discussed below \eqref{eqn:sde-phi-defn}) as an element of $\ell_2$.
     To show that $\Phi_i(t)$ are well defined maps on $\ell_2$ (namely, for $\Phi_1(t)$,  the sum in \eqref{eq:ssse} converges for every $x \in \ell_2$, and similarly for $\Phi_2(t)$)
       and they are bounded linear operators on $\ell_2$ it  suffices to show $\sum_{\mvl\in \fT} \la \Phi_i(t)e_{\mvl}, \Phi_i(t)e_{\mvl}\ra < \infty.$ 

     For $i=1$, note that we have $$\sum_{\mvl\in \fT}\la \Phi_1(t)e_{\mvl}, \Phi_1(t)e_{\mvl} \ra \leq  \left(\sum_{\mvl\in \fT\setminus\set{\mv0}}\norm{\mvl} \pi(\mvl,t)\right)^2 < \infty.$$ The last inequality is an immediate consequence of Lemma \ref{lem:all-moments-finite}. Similar calculations, and using Lemma \ref{lem:all-moments-finite}, show that $\Phi_2(t)$ also has this property. The non-negative definiteness of the operator $\Phi(t)$ follows from the discussion preceding \eqref{eqn:fdd-sde}. In particular, the terms within the two sets of braces in \eqref{eqn:fdd-fclt-bm} are nonnegative if $\fT_N$ there is replaced by $\fT\setminus\{0\}$.
     
    (iii) Fix $\delta \geq 0$, and a compact interval $\cI$ not containing $t_c$. From the definition of $\Phi(t)$ in \eqref{eqn:sde-phi-defn}, we have \begin{align*}
  \la e_{\mv0},\Phi(t)e_{\mv0}\ra &= \frac{1}{2} \theta_{\kappa}(\mu,\mu) \leq \norm{\kappa}_\infty, \\
  \la e_{\mvl},\Phi(t)e_{\mvl}\ra &= \frac{1}{2}\left[\sum_{\mvk_1+\mvk_2 = \mvl} \pi(\mvk_1,t)\pi(\mvk_2,t)\theta_{\kappa}(\mvk_1,\mvk_2)\right] + \frac{1}{2}\pi(\mvl,t) \theta_{\kappa}(\mvl,\mu) + \pi(\mvl,t)^2 \theta_\kappa(\mvl,\mvl) \\
  &\leq 2\norm{\kappa}_\infty \left[\sum_{\mvk_1+\mvk_2 = \mvl}\norm{\mvk_1}\norm{\mvk_2}\pi(\mvk_1,t)\pi(\mvk_2,t) + \norm{\mvl}\pi(\mvl,t)\right],  \ \ \mvl \neq 0.\\
    \end{align*} Therefore, by (i) of Lemma \ref{lem:norm-inequalities}, and Lemma \ref{lem:all-moments-finite}, we have \begin{align}\label{eqn:9825}
        \sup_{t\in \cI} \left[\la e_{\mv0}, \Phi(t) e_{\mv0}\ra + \sum_{\mvl \in \fT\setminus\set{\mv0}} \norm{\mvl}^{\delta} \la e_{\mvl}, \Phi(t) e_{\mvl}\ra \right] \leq C_1\left[1 + \sup_{t\in \cI} \sum_{\mvl \in \fT\setminus\set{\mv0}}\norm{\mvl}^{\delta+1} \pi(\mvl,t)\right] < \infty
    \end{align} for some constants $C_1>0$. 
    This finishes the proof of the lemma.\end{proof}

We note the following  integrability property of the process $\vX_n$.

\begin{lem}\label{lem:ui-property}
For any compact interval $\cI$ not containing $t_c$ and $\nu \geq 0$, we have  \begin{align*}
    \sup_n \E\left( \left[\sup_{t \in \cI}|X_n(\mv0,t)| + \sum_{\mvl \in \fT_{M\log n}\setminus\set{\mv0}}\norm{\mvl}^{\nu} \sup_{t \in \cI}|X_n(\mvl,t)|\right]\right) < \infty
\end{align*}
\end{lem} 
\begin{proof}
In the sub-critical regime, this is an immediate consequence of  \eqref{eqn:doob-mg-sub} and \eqref{eqn:doob-bdd-pr-sub}, along with the semi-martingale representation in \eqref{eqn:semi-mg-rep-X}. For the super-critical regime one can proceed in a similar fashion by observing that the estimates in \eqref{eqn:doob-mg-sub} and \eqref{eqn:doob-bdd-pr-sub} continue to hold if $\sup_{t\le T}$ is replaced with $\sup_{t \in [T_1, T_2]}$ for some $t_c< T_1 <T_2$ and $\mvM_n^c$, $\mvA_n^c$ replaced with $\tilde \mvM_n^c$ and $\tilde{\mvA\,}_n^c$, respectively, where the latter quantities are as introduced in \eqref{eqn:centerd-process-super}. 
\end{proof}
\begin{proof}[Proof of (a) in Theorem \ref{thm:fclt-irg}:]
We mainly focus on the proof of (i) in Theorem \ref{thm:fclt-irg} i.e functional central limit theorem in the sub-critical regime. The proof in the super-critical regime follows from similar arguments.
 Fix $T < t_c$. Recall the processes $\vX_n(t)$ and $\mvM^c_n(t)$ defined in \eqref{eqn:semi-mg-rep-X} and \eqref{eqn:centered-defns} respectively on the interval $[0,T]$. By Proposition \ref{lem:tightness-x-sub}, and Proposition \ref{lem:mg-tightness-sub}, the collection of random variables $\set{\vX_n(\cdot),\mvM^c_n(\cdot)}_{n\geq 1}$ is $\bC$-tight in $\bD([0,T]:\ell_{1,\delta}\times\ell_{1,\delta})$ for any $\delta \geq 0$.

For the rest of the proof, fix $\delta \geq 0$. Let $\set{\vX,\mvM}$ be a weak limit of a sub-sequence $\set{\vX_n,\mvM^c_n}_{n\geq 1}$ (also indexed by $\set{n}$) in $\bD([0,T]:\ell_{1,\delta}\times\ell_{1,\delta})$, given on some probability space $(\Omega, \cF, \bP)$. We first characterize the limit of the martingale term $\mvM$, and then $\vX$ in the following lemma.

 Fix $m \in \bN$, and a continuous and bounded function, $R:(\ell_{1,\delta}\times \ell_{1,\delta})^m \to \bR$.  For any $0\leq s\leq t\leq T$ and $0\leq t_1\leq t_2 \leq \dots\leq t_m \leq s$, let $\mvxi_i^n = (\vX_n(t_i),\mvM_n^{c}(t_i))$ and $\mvxi_i = (\vX(t_i),\mvM(t_i))$. Then, for all $\mvl \in \fT$, we have \begin{align*}
    \E\left(R(\mvxi_1,\mvxi_2,\dots, \mvxi_m) [M(\mvl,t) - M(\mvl,s)]\right) &= \lim_{n\to \infty}\E\left(R(\mvxi^n_1,\mvxi^n_2,\dots, \mvxi^n_m) [M_n^{c}(\mvl,t) -M_n^{c}(\mvl,s)]\right) =0
\end{align*} The first equality follows from the  integrability property of $\mvM_n^{c}(t)$, noted in \eqref{eqn:doob-mg-sub}, and the second follows from the martingale property of $\mvM_n^c(t)$. Thus, we have $\mvM(\cdot)$ is a $\cF_t$-adapted martingale where $\cF_t = \sigma\left(\set{\vX(s),\mvM(s):s\leq t}\right)$. Furthermore, by \eqref{eqn:doob-mg-sub} and Fatou's Lemma, the following holds 
\begin{align}\label{eqn:0007}
     \E \left[\sup_{t\in [0,T]}(M(\mv0,t))^2 + \sum_{\mvl \in \fT\setminus\set{\mv0}} \norm{\mvl}^{\nu} \sup_{t\in [0,T]}(M(\mvl,t))^2\right] < \infty
\end{align}
for all $\nu \geq 0.$ Therefore, we have $\mvM(t) = (M(\mvl,t):\mvl\in \fT)$ is a collection of $\cF(t)$-square-integrable martingale processes on $[0,T]$. Also, \eqref{eqn:0007} implies that $\mvM(t) \in \tilde{\ell}_1$, a.s., for all $t\in [0,T]$. 
Furthermore, from \eqref{eqn:21319-qv-m_0}, \eqref{eqn:qv-centered-mg-l}, the error estimates \eqref{eqn:epsilon_l-bound},  \eqref{eqn:r_n-bound}, and the fact that, from the tightness established in Proposition \ref{lem:tightness-x-sub}, $\mvpi_n(\cdot) \convp \mvpi(\cdot)$ in $\bD([0,T]: \ell_{1,\delta})$, we see that $$\sup_{t\leq T}\left|\la M^c_n(\mvl)\ra_t - \int_0^t \la e_{\mvl},\Phi(s)e_{\mvl} \ra ds\right| \convp 0.$$ From similar computations, using \eqref{eqn:fdd-pi-mg-term-infinite}, it follows that for any $\mvl,\mvk \in \fT$ \begin{align}\label{eqn:qv-convergence-lk}
\sup_{t\leq T}\left|\la M^c_n(\mvl), M^c_n(\mvk)\ra_t - \int_0^t \la e_{\mvl},\Phi(s)e_{\mvk}\ra ds\right| \convp 0.  
\end{align}
Furthermore, using estimates analogous to \eqref{eqn:QV-mg-bound}, it can be checked that  
$\set{M_n^c(\mvl,t)M_n^v(\mvk,t)}_{n\geq 1}$ and
$\set{\la M_n^c(\mvl),M_n^v(\mvk)\ra_t}_{n\geq 1}$ are uniformly integrable collection of random variables for all $\mvl,\mvk \in \fT$, and $t\leq T$. Therefore, similar to above arguments, for any $s < t < T$, we have \begin{align}
     &\E\left(R(\mvxi_1,\mvxi_2,\dots, \mvxi_m) [ M(\mvl,t)M(\mvk,t)-  M(\mvl,s)M(\mvk,s) -\int_s^t \la e_{\mvl}, \Phi(u) e_{\mvk}\ra du]\right) \nonumber\\
     &=\lim_{n\to \infty}\E\left(R(\mvxi_1,\mvxi_2,\dots, \mvxi_m) [M^c_n(\mvl,t)M^c_n(\mvk,t)-  M^c_n(\mvl,s)M^c_n(\mvk,s) -\int_s^t \la e_{\mvl}, \Phi(u) e_{\mvk}\ra du]\right) \nonumber\\
     &= \lim_{n\to \infty}\E(R(\mvxi^n_1,\mvxi^n_2,\dots, \mvxi^n_m) [\la M^c_n(\mvl),M^c_n(\mvk) \ra_t- \la M^c_n(\mvl),M^c_n(\mvk)\ra_s -\int_s^t \la e_{\mvl}, \Phi(u) e_{\mvk}\ra du]) =0.\nonumber
\end{align} This shows that the quadratic variation of the martingale process $\mvM(\cdot)$ is given by \begin{align}
    \la M(\mvl,t),M(\mvk, t) \ra = \int_0^t \la e_{\mvl}, \Phi(s) e_{\mvk}\ra ds, \qquad \mvl,\mvk \in \fT.
\end{align} 
Since $\bC([0,T]:  \ell_1) \subset \bC([0,T]: \ell_2)$, we see that $\{\mvM(t): 0 \le t \le T\}$ is a continuous $\ell_2$-valued martingale with quadratic variation given by the above expression in terms of the trace class operators $\{\Phi(s)\}$ for which $\sup_{t \in [0,T]} \|\Phi(t)\|_{\mbox{\tiny{tr}}}<\infty$.
It follows by \cite{Da_Prato_Zabczyk_2014}*{Theorem 8.2}  that on some extension $\left(\bar{\Omega}, \bar{\cF}, \{\bar{\cF}_t\}, \bar{\bP}\right)$ of the space $\left(\Omega, \cF, \{\cF_t\}, \bP\right)$, we are given an iid sequence of standard Brownian motions $\{B_{\mvk}, \mvk \in \fT\}$, independent of $\vX(0)$, such that
 \begin{align}\label{eqn:mg-is-bm}
    \mvM(t) = \int_0^t G(s) d\mvB(s),\end{align} 
    where recall that $G(s)$ is the nonnegative square root of $\Phi(s)$ and $\sup_{t \in [0,T]} \|G(t)\|_{\mbox{\tiny{HS}}}<\infty$.
For details on the definition and properties of   the above stochastic integral, see \cite{Da_Prato_Zabczyk_2014}*{Chapter 4}.

We next focus on characterizing the process $\vX(\cdot)$. Recall the semi-martingale representation of $\vX_n(\cdot)$ in \eqref{eqn:semi-mg-rep-X}. We first show that the process $\mvA_n^{c}(t)$ (jointly with $(\vX_n,\mvM_n^{c})$) converges in distribution to $\int_0^t \mvH(\vX(s),s)ds$ (and  $(\vX,\mvM))$, along the  chosen subsequence, where the function $\mvH:\tilde{\ell}_1 \times [0,T] \to \tilde{\ell}_1$ is defined as 
$$\mvH(\mvx,t) = \left[\frac{1}{2}\theta_{\Lambda}(\mu,\mu) + \theta_{\kappa}(\Psi,\mu)\right] e_{\mv0} + \sum_{\mvl \in \fT\setminus\set{\mv0}} \left[ F_{\mvl}(\mvpi(t),\Lambda,\mu) - \pi(\mvl,t)\theta_{\kappa}(\mvl,\Psi)\right]e_{\mvl} + \Gamma(t)\mvx.$$  
Recall from \eqref{eqn:bdd-process-defn} and \eqref{eqn:centered-defns}  that for any $\mvl\in \fT_{M\log n}$ we have
$$
A^c_n(\mvl,t)= \sqrt{n}\int_0^t \left[F_{\mvl}(\mvpi_n(s),\kappa_n,\mu_n) - F_{\mvl}(\mvpi(s),\kappa,\mu) \right] ds + \sqrt{n} \int_0^t r_n(\mvl,t).$$ 
From \eqref{eqn:r_n-bound}, we have \begin{align*}
   \sup_{t\leq T} \sqrt{n} \left[ |r_n(\mv0,t)| + \sum_{\mvl \in \fT_{M\log n}\setminus\set{\mv0}} \norm{\mvl}^{\delta}|r_n(\mvl,t)| \right] = O_{\pr}((\log n)^{\delta+K+1}/\sqrt{n}).
\end{align*}
Therefore, we have \begin{align}\label{eqn:bdd-pr-error-928}
    \sup_{t\leq T}\norm{\mvA_n^{c}(t) - \int_0^t \mvQ_n(s) ds }_{1,\delta} \convp 0,\qquad & \text{where}\\
    \mvQ_n(t) = \bigg(\sqrt{n} \left[ F_{\mvl}(\pi_n(t),\kappa_n,\mu_n) -F_{\mvl}(\pi(t),\kappa,\mu) \right] & : \mvl  \in \fT_{M\log n}\bigg).
\end{align}  
Define the $\ell_{1,\delta}$-valued process $\mvH_n(t) = (H_{\mvl}(\vX_n(t),t):\mvl \in \fT_{M\log n})$. We next claim that 
\begin{equation}\label{lem:bdd-var-approx}
 \int_0^T \norm{\mvQ_n(s) - \mvH_n(s)}_{1,\delta} ds \convp 0.
\end{equation}
We now finish the proof of the theorem using the claim in  \eqref{lem:bdd-var-approx}. Using \eqref{eqn:bdd-pr-error-928} and the claim, it is sufficient to show that along the chosen subsequence \begin{align}\label{eqn:sufficient-weak-convergence}
    \left(\vX_n,\mvM_n^{c}, \int_0^\cdot \mvH_n(s) ds\right) \convd \left(\vX,\mvM, \int_0^\cdot \mvH(\vX(s),s)ds\right)
\end{align} in $\bD([0,T]:\ell_{1,\delta}\times \ell_{1,\delta} \times \ell_{1,\delta})$.
Furthermore, since the tightness in Proposition \ref{lem:tightness-x-sub} holds for every $\delta\ge 0$, we can assume without loss of generality that, along the chosen subsequence, the above convergence in fact holds in $\bD([0,T]:\ell_{1,1+\delta}\times \ell_{1,\delta} \times \ell_{1,\delta})$.
By Skorohod representation theorem, we can assume $(\vX_n,\mvM_n^{c})$ converges almost surely to $(\vX,\mvM)$ in $\bD([0,T]:\ell_{1,1+\delta}\times\ell_{1,\delta})$. Also, from (ii) of Lemma \ref{lem:norm-inequalities} and Lemma \ref{lem:ui-property}, and Fatou's lemma we have that $\vX(t) \in \tilde{\ell}_1$. In fact, we also have $\sup_{t\leq T} \norm{\vX(t)}_{1,\nu} < \infty$ for all $\nu \geq 0$ almost surely. Also, for all $\mvl \in \fT_{M\log n}\setminus\set{\mv0}$, we see from \eqref{eqn:sde-gamma-defn} that  \begin{align*}
    |H_{\mvl}(\vX_n(s),s) - H_{\mvl}(\vX(s),s)|&\\
    &\hspace{-1cm}\leq C\left[\sum_{\mvk_1 + \mvk_2 = \mvl} \norm{\mvk_1}\norm{\mvk_2} \pi(\mvk_1,s) |X_n(\mvk_2,s) - X(\mvk_2,s)| + \norm{\mvl}|X_n(\mvl,s) - X(\mvl,s)|\right].
\end{align*} 
 
Therefore, using (ii) of Lemma \ref{lem:norm-inequalities}, we have \begin{align}\label{eqn:91}
    B_n^{(1)} =  \sup_{0\le s\leq T}\sum_{\mvl\in\fT_{M\log n}\setminus\set{\mv0}} \norm{\mvl}^{\delta} |H_{\mvl}(\vX_n(s),s) - H_{\mvl}(\vX(s),s)| \leq C \sup_{0\le s\leq T}\norm{\vX_n(s)-\vX(s)}_{1,\delta+1} \convp 0
\end{align} for some constant $C$ depending on $\delta,M$ and $T$. Also, we have 
\begin{align}\label{eqn:90}
    B_n^{(2)} &=\sup_{0\le s\leq T}\sum_{\mvl\in\fT\setminus\fT_{M\log n}}\norm{\mvl}^{\delta} |H_{\mvl}(\vX(s),s)|\\
    &\leq C' \left[\sup_{0\le s\leq T}\sum_{\mvl\in \fT\setminus\fT_{M\log n}}\norm{\mvl}^{\delta+1}|X(\mvl,s)|+ \sup_{0\le s\leq T}\sum_{\mvl\in \fT\setminus\fT_{M\log n}}\norm{\mvl}^{\delta+1}\pi(\mvl,s)\right] \nonumber\\
    &\leq \frac{C''}{M\log n} \left[\sup_{0\le s\leq T}\norm{\vX(s)}_{1,\delta+2} + 1\right]\convp 0.
\end{align} 

Therefore, by \eqref{eqn:90} and \eqref{eqn:91}, we have \begin{align*}
    \int_0^T \norm{\mvH_n(s) - \mvH(\vX(s),s) }_{1,\delta} ds \leq T\left[B^{(1)}_n +  B^{(2)}_n\right] \convp 0.
\end{align*} Hence, \eqref{eqn:sufficient-weak-convergence} holds. Therefore, by \eqref{eqn:inital-point-x}, \eqref{eqn:mg-is-bm} and the above, we have $$\vX(t) = \vX(0) + \int_0^t H(\vX(s),s)ds + \int_0^t G(s)d\mvB(s)$$ which is the SDE in \eqref{eqn:sde-vector}. Finally, the pathwise uniqueness of the solution is immediate from the observation that the operator $\Gamma(t)$ is lower-triangular for all $t\neq t_c$. This finishes the proof of part (a), under the assumption that the claim \eqref{lem:bdd-var-approx} holds. We now prove this claim.\\

\noindent{\bf Proof of \eqref{lem:bdd-var-approx}.}
    Let $\Lambda_n =\sqrt{n}(\kappa_n-\kappa)$ and $\Psi_n = \sqrt{n}(\mu_n-\mu)$. We have $\Lambda_n \to \Lambda$ and $\Psi_n \to \Psi$. Observe that \begin{align*}
        &\abs{Q_n(\mv0,t) - H_0(\vX_n(t),t)} \\
        &\leq \left[\norm{\Lambda_n-\Lambda}_\infty + 2 K \norm{\kappa}_\infty \norm{\Psi_n-\Psi}_\infty \right] +  \frac{K^2}{\sqrt{n}} \left[ \norm{\Lambda_n}_\infty\norm{\Psi_n}_\infty+ \norm{\Lambda}_\infty \norm{\Psi_n}_\infty + \norm{\kappa}_\infty\norm{\Psi}_\infty^2 \right].
    \end{align*} Therefore, we have $\sup_{t\leq T}\abs{Q_n(\mv0,t) - H_0(\vX_n(t),t)} \to 0$ as $n\to \infty$. Also, for $\mvl \in \fT_{M\log n}\setminus\set{\mv0}$, we have \begin{align}
    &\abs{Q_n(\mvl,t) - H_{\mvl}(\vX_n(t),t)}\nonumber\\
    &\leq \frac{C_1}{\sqrt{n}}\left[\sum_{\mvk_1 + \mvk_2 = \mvl} \norm{\mvk_1}\norm{\mvk_2} \bigg(|X_n(\mvk_1,t)||X_n(\mvk_2,t)| + \pi(\mvk_2,t)|X_n(\mvk_1,t)|\bigg) + \norm{\mvl}|X_n(\mvl,t)|\right] \nonumber\\
    &  + C_2 \left[\sum_{\mvk_1+ \mvk_2 = \mvl} \norm{\mvk_1} \norm{\mvk_2} \pi(\mvk_1,t) \pi(\mvk_2,t) \norm{\Lambda_n-\Lambda}_\infty + \norm{\mvl}\pi(\mvl,t) \left(\norm{\Lambda_n-\Lambda}_\infty + \norm{\Psi_n-\Psi}_\infty \right)\right]
\end{align} for some constants $C_1,C_2$. Therefore, using Lemma \ref{lem:all-moments-finite} and the inequalities in Lemma \ref{lem:norm-inequalities}, we have \begin{align*}
    \sum_{\mvl\in \fT_{M\log n}\setminus\set{\mv0}}\norm{\mvl}^{\delta}\abs{Q_n(\mvl,t) - H_{\mvl}(\vX_n(t),t)} &\leq \frac{C'_1}{\sqrt{n}} \left[\sum_{\mvl\in \fT_{M\log n}}\norm{\mvl}^{\delta+k+1} |X_n(\mvl,t)|^2\right] \\
    &\hspace{0.5cm}+ C_2' \left(\norm{\Lambda_n-\Lambda}_\infty + \norm{\Psi_n-\Psi}_\infty \right)
\end{align*} for some constants $C_1',C_2'$ depending on $M,T$. Therefore, as an immediate consequence of Theorem \ref{thm:variance-moment-bound-main}, we have $$\int_0^T \sum_{\mvl\in \fT_{M\log n}\setminus\set{\mv0}}\norm{\mvl}^{\delta}\abs{Q_n(\mvl,s) - H_{\mvl}(\vX_n(s),s)} ds \convp 0.$$ The above, along with \eqref{eqn:90}, completes the proof of the claim in \eqref{lem:bdd-var-approx}. 

The proof of part (b) in the theorem follows along similar lines. We omit the details.
\end{proof}

\section{Proofs: Macroscopic functionals in the supercritical regime}
\label{sec:proofs-macro-supercrit}
In this section, we apply Theorem \ref{thm:fclt-irg} to obtain  functional central limit theorems for the various macroscopic functionals for the IRG model, such as the number of connected components, the size of the giant component, the surplus in the giant, etc. We begin with the following lemma.
\begin{lem}\label{lem:auxillary-irg}
   Suppose Assumption \ref{ass:irg} holds. Then: \begin{enumerate}
        \item For any $T > t_c$, there exists $\beta = \beta(T)$ such that whp the second largest component in $\cG_n(t)$ is at most of size $\beta \log n$ for all $t\geq T$.
        \item For any $T, M\geq 0$,  the number of non-tree components of size at most $M \log n$ is uniformly, for $t\in [0,T]$, bounded by $(\log n)^2$ whp.
    \end{enumerate}
\end{lem}

\begin{proof} Without loss of generality, assume that there exists $a_* > 0$ such that $\kappa_n(x,y) \geq a_*$ for all $x,y\in \cS = [K]$ and $n\geq 1$.

\noindent(1) Fix $T > t_c$. For any $\alpha,\beta > 0$ and $t \ge T$, define the event \begin{align*}
        \cA_{n}^{\alpha,\beta}(t) = \set{\text{$\cG_n(t)$ has  two  components, one with  size at least $\alpha n$  and the other  at least }  \beta \log n}.
    \end{align*} Let $\cA_n^{\alpha,\beta} = \cup_{t\geq T}\cA_{n}^{\alpha,\beta}(t)$. Recall the Poisson processes $\cP_{i,j}^n(\cdot)$ as in the Definition \ref{def:irg} for $1 \leq i < j \leq n$. 
    Also recall that $\cC_n(i,t)$ is the component containing the vertex $i$ in $\cG_n(t)$.
    Let  \begin{align*}\cN &= \sum_{1\leq i < j \leq n} \int_T^{\infty} \ind\set{|\cC_n(i,t-)| \geq \alpha  n, |\cC_n(j,t-)| \geq \beta \log n, \cC_n(i,t-) \neq \cC_n(j,t-) } d\cP_{i,j}^n(t)\\
    &\hspace{0.2cm}+\sum_{1\leq i < j \leq n} \int_T^{\infty} \ind\set{|\cC_n(i,t-)| \geq \beta\log n , |\cC_n(j,t-)| \geq \alpha n, \cC_n(i,t-) \neq \cC_n(j,t-) } d\cP_{i,j}^n(t)\end{align*}
    
    Then, we have \begin{align}\label{eqn:second-largest-component-probability-bound}
    &\pr\left(A_n^{\alpha,\beta}\right) \leq \pr(\cN \ge 1) \le \E(\cN) \nonumber\\
        &\leq \frac{\norm{\kappa_n}_\infty}{n} \sum_{1\leq i\neq j \leq n} \int_T^{\infty} \pr\left(|\cC_n(i,t)| \geq \alpha n, |\cC_n(j,t)| \geq \beta \log n, \cC_n(i,t) \neq \cC_n(j,t) \right) dt \nonumber\\
        &\leq \frac{\norm{\kappa_n}_\infty}{n} \sum_{1\leq i\neq j \leq n} \int_T^{\infty} \exp(-t\alpha \beta a_*\log n) dt \leq \frac{n\norm{\kappa_n}_\infty}{\alpha\beta a_*\log n} \exp(-T\alpha\beta a_* \log n).\end{align} Therefore, for every $\alpha > 0$ with $\beta(\alpha,T) = \frac{2}{T\alpha a_*}$, we have 
        \ab{\begin{equation}\label{eq:107n}
        \pr\left(\cA_n^{\alpha,\beta(\alpha,T)}\right) = \frac{\norm{\kappa_n}_\infty}{\alpha a_*\beta(\alpha,T)}\frac{1}{n\log n} \to 0, \; \mbox{ as } n\to \infty.
        \end{equation}}
         Also, by Theorem \ref{thm:irg-boll}, there exists $\alpha_* = \alpha_*(T) > 0$ such that there exists a component of size $\geq \alpha_* n$ for all $t\geq T$ whp. Therefore, we have $$\pr\left(\text{ second largest component } \leq \beta(\alpha_*,T) \log n \text{ for all } t\geq T\right) \to 1.$$
        (2) Fix $M\geq 0$. Let $N_n^{\text{non-tree}}(t)$ be the number of non-tree components of size at most $M \log n$ in $\cG_n(t)$. Also, let $\tilde{N}_n(t)$ be the number of times an edge is added between vertices in the same components of size at most $M \log n$ up to time $t$. Note that $N_n^{\text{non-tree}}(t) \leq \tilde{N}_n(t)$ for all $t \geq 0$. For vertices $1\leq i\neq j\leq n$, let $E_{i,j}(t) =\set{\cC_n(i,t) = \cC_n(j,t) \in \fT_{M\log n}}$. Recall the Poisson process $\cP_{i,j}^n$ associated with edge $i\neq j$ as in Definition \ref{def:irg}. We then have $$\tilde{N}_n(t) = \sum_{1\leq i\neq j\leq n}\int_0^t \ind\set{E_{i,j}(s-)}d\cP_{i,j}^n(s)$$ and hence \begin{align*}
       \E\left(\sup_{t\in [0,T]} N_n^{\text{non-tree}}(t)\right) \leq \E\left(\tilde{N}_n(T)\right) \leq MT\norm{\kappa_n}_\infty  \log n = O\left(\log n\right).
    \end{align*} The result now follows on applying Markov's inequality.\end{proof}

\begin{proof}[Proof of Theorem \ref{thm:fclt-giant-surplus} ]
    Fix $t_c < T_1 < T_2$. By (i) of Lemma \ref{lem:auxillary-irg}, there exists $A_1 = A_1(T_1)$ such that the second largest component in $\cG_n(t)$ is at most of size $A_1\log n$ for all $t\in [T_1,T_2]$ with high probability. Also, from Remark \ref{rem:exp-decay}, there exists $\epsilon > 0$ such that, for all $M >0$,
    \begin{equation}\label{eq:448}
    \sup_{t\in [T_1,T_2]} \sqrt{n} \sum_{\mvl\in \fT \setminus\fT_{M\log n}}\norm{\mvl}\pi(\mvl,t) \leq C\sqrt{n}\sum_{k\geq M\log n} e^{-\epsilon k} \leq C'\sqrt{n}e^{-M \epsilon \log n}.
    \end{equation}
    Therefore, for a sufficiently large choice of $A_2$, we have $$\sup_{t\in [T_1,T_2]} \sqrt{n} \sum_{\mvl\in \fT \setminus\fT_{A_2\log n}}\norm{\mvl}\pi(\mvl,t) \to 0$$

For the rest of the proof, let $A = \max\{A_1,A_2\}$. Therefore, number of components and size of the largest component are given, whp, by   \begin{align}\label{eqn:21312312-giant}
        N_n(t) = n \left[\sum_{\mvl \in \fT_{A\log n}\setminus\set{\mv0}}\pi_n(\mvl,t)\right], \hspace{1cm} L_n(t) = n \left[1 - \sum_{\mvl \in \fT_{A\log n}\setminus\set{\mv0}} \norm{\mvl}\pi_n(\mvl,t)\right].
    \end{align} We next focus on the surplus of the giant. Let $\gamma_n(t)$ be the number of edges in components of size at most $A \log n$. Then, by (ii) of Lemma \ref{lem:auxillary-irg}, we have  \begin{align}\label{eqn:21312321932193218-surplus}
        \frac{\gamma_n(t)}{n} = \left[\sum_{\mvl\in \fT_{A\log n}\setminus\set{\mv0}}(\norm{\mvl}-1)\pi_n(\mvl,t)\right] + O_{\pr}\left(\frac{(\log n)^4}{n}\right).
    \end{align} The second term in the above accounts for multiple edges and number of non-tree components of size at most $A\log n$. Also, number of edges in the giant, whp, is given by $$
    \frac{|\cE(\cC_{(1),n}(t))|}{n}= \pi_n(\mv0,t) - \frac{\gamma_n(t)}{n} = \sum_{\mvl \in \fT_{A\log n}}\pi_n(\mvl,t) - \sum_{\mvl \in \fT_{A\log n}}\norm{\mvl}\pi_n(\mvl,t).$$ Therefore, for all $t\in [T_1,T_2]$, the surplus of the giant, $S_n(t)$, is given, whp, by \begin{align*}
       \frac{S_n(t)}{n} =  \frac{|\cE(\cC_{(1),n}(t))|}{n} - \frac{L_n(t)}{n} = \sum_{\mvl \in \fT_{A\log n}}\pi_n(\mvl,t) - 1.
    \end{align*}  
    The above representations, together with the definitions in \eqref{eq:837}, along with  Theorem \ref{thm:fclt-irg}(b) and \eqref{eq:448}, complete the proof of the first assertion of the theorem. Similar arguments, along with  (c) of Theorem \ref{thm:irg-boll}, yield FCLT for the number of components in the sub-critical regime as stated in \eqref{eq:843bn}. \end{proof}

\begin{proof}[Proof of Theorem \ref{thm:359}]
    Let $L_n^{(i)}(t)$ be the number of vertices of type $i \in \cS$ in the giant component at time $t\geq 0$. Fix $t_c < T_1 < T_2$. By  Lemma \ref{lem:auxillary-irg}(1), there exists $A \geq 0$ such that with high probability \begin{align*}
        L_n^{(i)}(t) = n\left[\mu_n(i) - \sum_{\mvl\in \fT_{A\log n}\setminus\set{\mv0}} l_i \pi_n(\mvl,t)\right],\qquad i\in \cS \text{ and } t\in [T_1,T_2].
    \end{align*} Now, proceeding similarly to the proof of Theorem \ref{thm:fclt-giant-surplus}, we get the desired result.
\end{proof}

\section{Proofs:  Weight of the MST on dense graphs}
\label{sec:proofs-dense-mst}

Fix $\kappa_n, \mu$ and recall the corresponding graphon $\thkappa_n$  and the model in Definition \ref{def:mst-dense}. We write $\cW_n^{\fano}$ for the weight of the MST; we use ``$\fano$''  to distinguish this from a related model, written as $\fanz$, on the complete graph,  that turns out to be directly related to the results proven in Sections \ref{sec:fdd-fclt} - \ref{sec:proofs-macro-supercrit}.

\begin{defn}[MST Model $-\fanz$]
    \label{def:model-0}
    Start with the complete graph $\cK_n$ with vertex set $[n]$. \abb{Let $\vU_n:=\set{U_i:1\leq i\leq n}$ be i.i.d. Uniform $[0,1]$ random variables as in the  MST Model $-\fano$ in Definition \ref{def:mst-dense}.} Conditional on $\vU_n$, generate conditionally independent exponential edge weights but with heterogeneous weights for all edges in $\cK_n$ with edge $e = \set{i,j}$ having weight $\xi_e \sim \exp(\thkappa_n(U_i, U_j))$. Let $\cK_n^{\fanz}$ denote the weighted complete graph with the above edge weights and let $\cW_n^{\fanz}$ denote the corresponding weight of the MST on $\cK_n$ with the above edge weights. 
\end{defn}

We will prove Theorem \ref{thm:clt-mst-weight} with $\cK_n(\thkappa) = \abb{\E(\cW_n^{\fanz}\mid \vU_n)}$. The main goal of this section is to show the following.

\begin{thm}
    \label{thm:mst-wn0-convg}
     The \ab{first} assertion of Theorem \ref{thm:clt-mst-weight} holds for $\cW_n^{\fanz}$ namely $$\sqrt{n}\left(\cW_n^{\fanz} -\cK_n(\thkappa)\right) \convd \fN_\infty,$$ 
    where $\fN_\infty$ is as in Theorem \ref{thm:clt-mst-weight} and $\cK_n(\thkappa) = \abb{\E(\cW_n^{\fanz}\mid \vU_n)}$. 

\end{thm}
Proof of Theorem \ref{thm:mst-wn0-convg} is provided in Section \ref{sec:8.2}.
We now complete the proof of Theorem \ref{thm:clt-mst-weight} using Theorem \ref{thm:mst-wn0-convg}. 
\subsection{Proof of Theorem \ref{thm:clt-mst-weight}:}
We begin with the following fundamental identity relating the weight of the MST on a connected weighted graph \ab{to the number of components in the graph obtained by only retaining all edges up to a certain weight threshold.}

\begin{lem}[{\cite{janson1995minimal}*{equation 3.1}}]
    \label{lem:mst-comp-janson}
  Consider a weighted connected graph $\cG = (\cV,\cE, \cW = (w_e:e\in \cE))$ where $w_e$ is the weight assigned to the edge $e \in \cE$, let $\cW(\cG)$ be the weight of the minimal spanning tree on $\cG$. For any fixed $t\geq 0$ let  $N(\cG,t)$ be the number of components \ab{in the graph obtained from $\cG$ by retaining only edges $e$ with $w_e \leq t$}. Then 
   \[ \cW(\cG) = \int_0^\infty [N(\cG,t) - 1]dt = \int_0^{\max_{e\in \cE} w_e} [N(\cG,t) - 1]dt. \]
\end{lem}

The following proposition completes the proof of \ab{the first assertion in Theorem \ref{thm:clt-mst-weight}}. \sa{The second assertion in the theorem  is an immediate consequence of \cite{hladky2023random}*{Theorem 5} and we omit its proof.} 
\begin{prop}
    \label{prop:coupl-wn0-wn1}
    There exists a common probability space on which $\cW_n^{\fano}$ and $ \cW_n^{\fanz}$ can be constructed such that $\pr(|\cW_n^{\fano} - \cW_n^{\fanz}| > \ab{(\log n)^4/n}) \to 0$ as $n\to\infty$. 
\end{prop}
\begin{proof}
\ab{We begin by constructing a coupling that will define the MST Model $-\fanz$
and the  MST Model $-\fano$ on a common probability space.
 
Let $\vU_n = \{U_1, \ldots, U_n\}$ denote a collection of iid Uniform $[0,1]$ random variables given on some probability space. Assume without loss of generality that on this space we are also given a collection of independent rate one Poisson processes 
 $\set{\cP_{e}: e\in \cK_n}$, independent of $\vU_n$.   Now we consider the following two models on this probability space:}
\begin{enumeratea}
    \item {\bf Model $\fanz$:} For each edge $e = \set{i,j}$, mark each point of the corresponding point of $\cP_{e}$ independently with probability $\thkappa_n(U_i, U_j)$. Let the edge weight $S_e$ be the value \ab{(i.e. the time instant)} of the first marked point. Then standard properties of Poisson processes imply that, \ab{conditionally on $\vU_n$,} $S_e \sim \exp(\thkappa_n(U_i, U_j))$. Write $\cG_n^{\fanz}$ for the resulting complete graph with edge weights $(S_e: e\in \cK_n)$. \ab{Note that this weighted graph has the same distribution as the MST Model $-\fanz$.}
    \item {\bf Model $\fano$:} For every edge $e$, let $T_e> 0$ denote the first point of $\cP_e$. If the first point is marked then let the edge weight be $T_{e}$ while if the first point is {\bf not} marked, then set the edge weight to be $\infty$ (i.e. the edge is not placed in the graph). Let $\cG_n^{\fano}$ denote the resulting random graph with edge weights.   \ab{This weighted graph has the same distribution as the MST Model $-\fano$.}
\end{enumeratea}
By construction the edge weights of the edges {\bf present} in $\cG_n^{\fano}$ coincide with that it $\cG_n^{\fanz}$; we write this relation as $\cG_n^{\fano}\subseteq \cG_n^{\fanz}$. Fix $n_0 \in \bN$, and $\alpha_{*} > 0$  such that,
\begin{equation}
    \label{eqn:def-alpha*}
    \min_{i,j \in [K]} \kappa_n(i,j) \geq  \alpha_* \mbox{ for all } n \ge n_0.
\end{equation}
   Let $A = 3/\alpha_{*} $.  \ab{We will assume without loss of generality that $n_0$ is large enough so that $\exp\{-A\log n_0/n_0\} \le 2/3$.}
   We now introduce a {\bf pruning step} that results in an intermediate weighted graph $\cG_n^{\star}$ as follows: the vertex set as before is $[n]$; the edge set of $\cG_n^{\star}$ are the collection of edges $e$ (along with their corresponding edge weights) such that:
\[S_e = T_e \qquad \text{\bf and } \qquad T_e \leq \frac{A\log{n}}{n}. \]
By construction, $\cG_n^{\star} \subseteq \cG_n^{\fano}$. Let $\cW_n^{\star}$ denote the weight of the MST of $\cG_n^{\star}$ with $\cW_n^{\star} = \infty$ if $\cG_n^{\star}$ is not connected. The following claim complets the proof of the Proposition and thus the Theorem. \\

\noindent {\bf Claim 1:}
With the choice of $A$ above, $\cG_n^{\star}$ is connected whp. In particular \begin{enumeratea}
    \item  $\pr(\cW_n^{\star} = \cW_n^{\fano}) \to 1$.
    \item  $\pr\left(|\cW_n^{\star} - \cW_n^{\fanz}| > \sa{\frac{(\log n)^4}{n}}\right) \to 0$     as $n\to\infty$.
\end{enumeratea}

The proof of the requires following observation about the change in the weight of a minimal spanning tree when a new edge is added. Consider a weighted \emph{connected} graph $\cG = (\cV,\cE, \cW = (w_e:e\in \cE))$ where $w_e$ is the weight assigned to the edge $e \in \cE$, let $\cW(\cG)$ be the weight of the minimal spanning tree in $\cG$. Let $e'$ be an edge not in $\cE$, and associated with it a weight $w_{e'}$ Define $\cW(\cG')$ to be the weight of the minimal spanning tree of the graph $\cG' = (\cV,\cE' = E \cup \set{e'},\cW' = \cW \cup \set{w_{e'}} )$.
\begin{equation}\label{lem:optima-bound}
     \cW(\cG) - \cW(\cG') \leq \left(\max_{e \in \cE} w_e - w_{e'}\right)_+
\end{equation} 
The above inequality says that when we add an edge whose weight is larger than
\ab{the maximal edge weight in the existing graph},  then the weight of the minimal spanning tree does not change, and it also gives an upper bound on the change in the weights when that is not the case. Before proving \eqref{lem:optima-bound} we prove Claim 1 using this inequality.

  Fix $n \ge n_0$.  Note that for any edge $e=(i,j)\in \cK_n$, the probability this edge is present in $\cG_n^*$ in the above construction equals $ \kappa_n(U_i, U_j) (1-\exp(-A\log{n}/n)) \geq 2\log{n}/n$  by choice of the $A$. Also, it follows from known results that an \erdos random graph with connection probability $2\log{n}/n$ is connected whp \cite{boll-book}. Therefore, $\cG_n^*$ is connected whp.
    
(a) One can obtain $\cG_n^{\fano}$ from $\cG_n^*$ by sequentially  adding edges  (in ascending order of weights) whose weights are in $[A \log n/n, \infty)$ . Then by Lemma \ref{lem:optima-bound}, we have $\cW_n^* = \cW_n^{\fano}$ when $\cG_n^*$ is connected. This proves part (a).

(b) Let $R_n$ denote the number of edges $e$ in $\cG_n^{\fanz}$ where the first point in $\cP_e$ is not marked, but $S_e \leq A \log n/n$. Again, one can obtain $\cG_n^{\fanz}$ from $\cG_n^*$ by adding these $R_n$ edges (sequentially). When $\cG_n^*$ is connected, by \eqref{lem:optima-bound}, we then have $\cW_n^* - \cW_n^{\fanz} \leq R_n A\log{n}/n$. We also have \[\E(R_n) \leq \sum_{e\in \cK_n} \pr(\cP_e[0,A\log{n}/n]\geq 2) = O([\log{n}]^2).\]
Thus, $\abs{\cW_n^{\star} - \cW_n^{\fanz}} = O_P((\log{n})^3/n)$. This completes the proof of {\bf Claim 1}.

\ab{Finally we prove \eqref{lem:optima-bound}.
This inequality is trivially true when $w_{e'} \ge \max_{e \in \cE} w_e$. Consider now the case when $w_{e'} < \max_{e \in \cE} w_e$ and so
$\max_{e \in \cE} w_e = \max_{e \in \cE'} w_e$.}

Using Lemma \ref{lem:mst-comp-janson} for $\cG$ and $\cG^\prime$ and using the fact that $\cG$ is connected gives, 

    \ab{\begin{align*}
    \cW(\cG) - \cW(\cG') &= \int_0^{\max_{e\in \cE} w_e}[N(\cG,t) - N(\cG',t)]dt.\\
    &= \int_{w_{e'}}^{\max_{e\in \cE} w_e}[N(\cG,t) - N(\cG',t)]dt \leq \max_{e\in \cE} w_e - w_{e'},
    \end{align*}
    where the second equality uses the fact that $N(\cG,t) = N(\cG',t)$ when
    $t <w_{e'}$ and
     the last inequality follows from observing that having the extra edge $e'$ can merge at most two different components, i.e. $0\le N(\cG,t) - N(\cG',t)\le 1$. }

\ab{This completes the proof of \eqref{lem:optima-bound} and therefore of the proposition and consequently also of the first assertion in Theorem \ref{thm:clt-mst-weight}.}
\end{proof}
\subsection{Proof of Theorem \ref{thm:mst-wn0-convg}:}
\label{sec:8.2}
\ab{For $t\ge 0$, denote the graph obtained from $\cK_n^{\fanz}$ by retaining only edges with weights $\xi_e \le t/n$ as $\cG^n(t)$. We denote the number of components in $\cG^n(t)$ by $\tilde N_n(t)$. \abb{Define for $i \in [n]$, $\tilde x_i = \sum_{x \in [K]} x \ind\{U_i \in \cI_x\}$. Then $\{\tilde x_i, i \in [n]\}$ are i.i.d. distributed as $\mu$.
It is easy to check that, with $\tilde \mvx^n = \{\tilde x_i, i \in [n]\}$,
$\{\{\cG^n(t), t\ge 0\}: \tilde \mvx^n\}$ has the same distribution as $\{\{\cG_n(t,\mvx^n,\kappa_n): t \ge 0\}: \mvx^n\}$ in Definition \ref{def:irg}, where $\mvx^n = \{x_i, i \in [n]\}$ and $\{x_i, i \in \bN\}$ are iid distributed as $\mu$ independent of the driving Poisson processes in Definition \ref{def:irg}. Consequently, 
\begin{equation}\label{eq:samed}(\{\tilde N_n(t), t\ge 0\},\tilde \mvx^n)  \stackrel{d}{=}
(\{ N_n(t), t\ge 0\}, \mvx^n), \mbox{ and } \E(\tilde N_n(t) \mid \vU_n) = 
\E(\tilde N_n(t) \mid \tilde \mu_n) = \E( N_n(t) \mid  \mu_n),
\end{equation}
where $N_n(\cdot)$ are as defined in Section \ref{sec:res-macro},
and $\mu_n$ (resp. $\tilde \mu_n$) are the empirical measures defined as in the statement of Theorem \ref{thm:irg-boll} using $\mvx^n$ (resp. $\tilde \mvx^n$).} Also note that, with $N(\cdot, \cdot)$ as defined in Lemma \ref{lem:mst-comp-janson},
$N(\cK_n^{\fanz}, t/n) = \tilde N_n(t)$. Thus, from Lemma \ref{lem:mst-comp-janson},
\begin{align*}
\cW_n^{\fanz} &= \int_0^\infty [N(\cK_n^{\fanz},t) - 1]dt\\
&= \frac{1}{n} \int_0^\infty [N(\cK_n^{\fanz},t/n) - 1]dt = 
\frac{1}{n} \int_0^\infty [\tilde N_n(t) - 1]dt
\end{align*}
From \eqref{eq:samed} we now have that
\abb{$$\left(\cW_n^{\fanz}, \tilde \mu_n\right) \stackrel{d}{=} \left(\frac{1}{n} \int_0^\infty [ N_n(t) - 1]dt, \mu_n\right).
$$} Thus in order to prove Theorem \ref{thm:mst-wn0-convg} it suffices to prove the result with $\cW_n^{\fanz}$ replaced with
$$\bar\cW_n^{\fanz} = \frac{1}{n} \int_0^\infty [ N_n(t) - 1]dt
$$
\abb{and $\cK_n(\thkappa) = \E(\cW_n^{\fanz}\mid \vU_n) = \E(\cW_n^{\fanz}\mid \tilde \mu_n)$
replaced with $\E(\bar\cW_n^{\fanz} \mid \mu_n)$.}
}
 
 Also fix $n_0$ and $\alpha_*$ as in \eqref{eqn:def-alpha*} for the rest of this section.  Throughout we only consider $n \ge n_0$. 
 \ab{The conditional expectation with respect to $\mu_n$ will be denoted as $\E_{\mu_n}$ throughout this section.}

\begin{lemma}\label{lem:tail-tail}
There exists $A \geq 0$ such that 
    \begin{enumeratei}
         \item With high probability,  the following holds $$\bar\cW_n^{\fanz} = \frac{1}{n}\int_0^{A\log n} (N_n(t) - 1)dt.$$
        \item $$\lim_{n\to \infty}\frac{1}{\sqrt{n}} \int_{A\log n}^{\infty} \ab{\E_{\mu_n}}(N_n(t) -1) dt = 0.$$
    \end{enumeratei} 
\end{lemma}
\begin{proof}
    As argued in the proof of Proposition \ref{prop:coupl-wn0-wn1} (see proof of {\bf Claim 1} there), there exists $A \geq 0$ such that $\cG_n(t,\mvx^n,\kappa_n)$ is connected for all $t\geq A \log n$ whp. This proves (i). 

    To prove (ii), observe that,  $N_n(t)$ is bounded above by number of components in $\ER(n,\alpha_*t/n)$ which we denote  by ${N}^*_n(t)$. Using \cite{janson1995minimal}*{(iii) of Lemma 2.2, Lemma 3.3 and Lemma 3.4}, it follows that for suitably large $A$, we have $$
    \lim_{n\to \infty}\frac{1}{\sqrt{n}} \int_{A\log n}^{\infty} \E_{\mu_n}(N_n(t) -1) dt \leq \lim_{n\to \infty}\frac{1}{\sqrt{n}} \int_{A\log n}^{\infty} \E(N^*_n(t) -1) dt = 0.$$ 

 This concludes the proof the lemma.\end{proof}

By the above lemma, to prove the result, it is sufficient to show that \begin{align}\label{eqn:sufficient-mst}
    \gamma_n =  \int_0^{A\log n} \frac{N_n(t) - \E_{\mu_n}(N_n(t))}{\sqrt{n}} dt  = \int_0^{A\log n} \sum_{\mvl \neq 0} Y_n(\mvl,t) dt\convd \fN_\infty
\end{align} where $Y_n(\mvl,t) = \sqrt{n}\left(\pi_n(\mvl,t) - \E_{\mu_n}(\pi_n(\mvl,t))\right)$ for $\mvl \in \fT$. 

\ab{As noted in the discussion at the beginning of Section \ref{sec:fdd-fclt}, it is sufficient to prove the above result for the collection $\mvx$ that is nonrandom, as a consequence of Lemma \ref{lem:weakcgce}. For the rest of the section, we assume that the empirical distribution of type $\mu_n$ is deterministic and $\Psi_n := \sqrt{n}(\mu_n-\mu)$ converges to a nonrandom vector $\Psi$. We also assume without loss of generality that for all $n \ge n_0$, $\mu_n(j)$ and $\mu(j)$ are bounded below by $\alpha_*$ for all $j \in [K]$
and for such $n$ the inequality in \eqref{eq:bdonerr} is satisfied.
}

For fixed $k\geq 0$, define the ``$k$-th level truncation'': \begin{align*}
    \gamma_{n,k} = \int_0^k \sum_{\mvl \in \fT_k\setminus\set{\mv0}} Y_n(\mvl,t) dt.
\end{align*}

The proof of the following two lemmas are provided at the end of the section. 
\begin{lem}\label{lem:mst-wc-lemma}
\begin{enumeratei}
   \item For any $k\geq 0$, as $n\to \infty$, we have $$\gamma_{n,k} \convd \gamma_k \sim N(0,\sigma_k )$$ where $\sigma_k = \int_0^k\int_0^k \sum_{\mvl\in \fT_k\setminus\set{\mv0}}\sum_{\mvk\in \fT_k\setminus\set{\mv0}}\Sigma(\mvl,\mvk,s,t)ds dt$, and $\Sigma(\mvl,\mvk,s,t)$ is as defined in \eqref{eqn:two-point-covariance}.

   \item As $k\to \infty$, we have $\sigma_k  \to \sigma_\infty$ and $\sigma_\infty < \infty$. Hence, as $k\to \infty, \gamma_k \convd N(0,\sigma_\infty)$.
\end{enumeratei}
\end{lem}

\begin{lem}\label{lem:mst-tail-lemma}
    For any $\epsilon > 0$, we have $$\limsup_{k\to \infty} \limsup_{n\to \infty} \pr\left(\abs{\gamma_n -\gamma_{n,k}} \geq \epsilon\right) = 0.$$
\end{lem}

The above lemmas, along with a triangle inequality, completes the proof \ab{of \eqref{eqn:sufficient-mst} and therefore complete the proof of Theorem \ref{thm:mst-wn0-convg}.} 
\qed

The remainder of the section focuses on the proofs of Lemma \ref{lem:mst-wc-lemma} and Lemma \ref{lem:mst-tail-lemma}.
\begin{proof}[Proof of Lemma \ref{lem:mst-wc-lemma} ]
(i) Fix $k\geq 1$. Let $\mvY_n^{\fT_k}(t) = (Y_n(\mvl,t):\mvl \in \fT_k)$ and recall the  process $\vX_n^{\fT_k}(t)$ as introduced below \eqref{eqn:xl-def}. For $\mvl \in \fT_k$, we have  $$X_n(\mvl,t) = Y_n(\mvl,t) + R_n(\mvl,t) + Z_n(\mvl,t)$$ where,
$$R_n(\mvl,t) = \sqrt{n}\left(\ab{\E_{\mu_n}}(\pi_n(\mvl,t)) - \pi(\mvl,t;\kappa_n,\mu_n) \right), \qquad Z_n(\mvl,t) = \sqrt{n}\left(\pi(\mvl,t;\kappa_n,\mu_n) - \pi(\mvl,t;\kappa,\mu) \right).$$ 
As observed in \eqref{eqn:908908}, we have \begin{align*}
    Z_n(\mvl,t) &= \int_0^1 \frac{d}{d u}\pi(\mvl,t;\kappa_n^{u},\mu_n^{u}) 
    d u
\end{align*} where $\kappa_n^u = u\kappa_n + (1-u)\kappa$ and $\mu_n^u =u\mu_n + (1-u)\kappa$. From the observations in the proof of Lemma \ref{lem: MBP-approx}, since the entries of $\kappa_n,\kappa,\mu_n$ and $\mu$ are all bounded below by $\alpha_* > 0$ for large $n$, by dominated convergence using \eqref{eqn:gradient-trees-kappa} and \eqref{eqn:gradient-trees-mu},
we have \begin{align}\label{eqn:mst-zeeeeeeeeeeee}
   \lim_{n\to \infty} Z_n(\mvl,t) =  \nabla_{\kappa}\pi(\mvl,t;\kappa,\mu)\cdot \Lambda + \nabla_{\mu}\pi(\mvl,t;\kappa,\mu)\cdot \Psi,\qquad \mvl \in \fT_k, \  t\leq k.
\end{align} 
\ab{Note that the right side of the above equation equals $m(\mvl, t)$ defined in Proposition \ref{prop:mean-variance}(i). Also, from the SDE for $X^{\fT_k}$ given in \eqref{eqn:fdd-sde}, and with $\mvm^{\fT_k}(t) := (m(\mvl, t), \mvl \in \fT_k)$, we have that $\mvm^{\fT_k}$ solves the ODE
$$\dot{\mvm}^{\fT_k}(t) = a^{\fT_k}(t) + \Gamma^{\fT_k}(t)\mvm^{\fT_k}(t), \; \mvm^{\fT_k}(0) = X^{\fT_k}(0),$$
where $\Gamma^{\fT_k}$, and $a^{\fT_k}$  are as defined in Section \ref{sec:res-micro}.
}
Also, by Lemma \ref{lem:TV-approximation} we have $\sup_{t\leq k}R_n(\mvl,t) \to 0$ as $n\to \infty$ for all $\mvl \in \fT_k$. Since, by Theorem \ref{thm:fdd-fclt}, $\vX_n^{\fT_k}(\cdot) \convd \vX^{\fT_k}(\cdot)$ in $\bD([0,k]:\bR^{\fT_k})$, the above discussion implies $\mvY_n^{\fT_k}(\cdot) \convd \mvY^{\fT_k}(\cdot)$ in $\bD([0,k]:\bR^{\fT_k})$ \ab{where $\mvY^{\fT_k}(\mvl,t) = \vX^{\fT_k}(\mvl,t) - m(\mvl, t)$ and 
$\mvY^{\fT_k}$ is the solution to the SDE \begin{align*}
    d\mvY^{\fT_k} = \Gamma^{\fT_k}(t)\mvY^{\fT_k} dt + G^{\fT_k}(t)d\mvB^{\fT_k}(t), \qquad \mvY^{\fT_k}(0) = 0,
\end{align*}
where $G^{\fT_k}$ and $\mvB^{\fT_k}$ are as defined in Section \ref{sec:res-micro}.
Furthermore, recall from \eqref{eqn:two-point-covariance} that  $ \cov(Y(\mvl,s),Y(\mvk,t)) = \Sigma(\mvl,\mvk,s,t)$}  for all $\mvl,\mvk \in \fT_k$ and $s,t \leq k$.  Therefore, we have $$\gamma_{n,k} \convd \gamma_k = \int_0^k \sum_{\mvl\in \fT_k\setminus\set{\mv0}} Y(\mvl,t) dt \sim N(0,\sigma_k).$$ 

(ii)
For $s,t \in \bR_+$ and $k\geq 1$, let   $$f_k(s,t) = \sum_{\mvl,\mvk\in \fT_k\setminus\set{\mv0}} \Sigma(\mvl,\mvk,s,t).$$
 Let $\lambda$ be the Lebesgue measure on $\bR_+^2$. Let $\cJ = \set{(s,t)\in \bR_+^2: s=t_c \text{ or } t=t_c}$. Note that $\lambda(\cJ) = 0$. Also, for any $(s,t) \not\in \cJ$, by Remark \ref{rem:exp-decay}, there exists $\delta_1 = \delta_1(s) >0 ,\delta_2 = \delta_2(t) >0$ such that $\pi(\mvl,s) \leq \exp(-\delta_1 \norm{\mvl})$ and $\pi(\mvl,t) \leq \exp(-\delta_2 \norm{\mvl})$ for all $\mvl\in \fT\setminus\set{\mv0}$. Also note that by (ii) of Proposition \ref{prop:mean-variance}, we have  $$\Sigma(\mvl,\mvk,s,t)^2 \leq \Sigma(\mvl,\mvl,s,s)\Sigma(\mvk,\mvk,t,t) \leq (1+\norm{\kappa_n}_\infty s)(1+\norm{\kappa_n}_\infty t)\norm{\mvl}\norm{\mvk}\pi(\mvl,s)\pi(\mvk,t), \hspace{0.15cm} \mvl,\mvk \in \fT\setminus\set{\mv0}.$$
\ab{Consequently, for all $(s,t) \notin \cJ$, $\sum_{\mvl,\mvk\in \fT\setminus\set{\mv0}} |\Sigma(\mvl,\mvk,s,t)| < \infty$ and so
$f_\infty(s,t) = \sum_{\mvl,\mvk\in \fT\setminus\set{\mv0}} \Sigma(\mvl,\mvk,s,t)$ is well defined 
and  $f_k(s,t) \to f_\infty(s,t)$, as $k \to \infty$, for all $(s,t) \notin \cJ$.}

Also, for all $s,t\in \bR_+^2$, we have \begin{align*}
   |f_k(s,t)| &= \bigg|\cov\left(\sum_{\mvl\in \fT_k\setminus\set{\mv0}} Y(\mvl,s), \sum_{\mvk\in \fT_k\setminus\set{\mv0}} Y(\mvk,t),\right)\bigg|\\
&\leq \sqrt{\var\left(\sum_{\mvl\in\fT_k\setminus\set{\mv0}}Y(\mvl,s)\right)\var\left(\sum_{\mvk\in\fT_k\setminus\set{\mv0}}Y(\mvk,t)\right)}
\end{align*}
Using Proposition \ref{prop:mean-variance}(ii), we have 
\begin{align}
\var\left(\sum_{\mvl\in\fT_k\setminus\set{\mv0}}Y(\mvl,t)\right) &\leq (1+\norm{\kappa}_\infty t) \sum_{\mvl \in \fT_k\setminus\set{\mv0}} \norm{\mvl}\pi(\mvl,t)\nonumber\\
&\leq (1+\norm{\kappa}_\infty t) \pr(|\MBP_{\mu}(t\kappa,\mu)| < \infty)\nonumber\\
&\leq (1+\norm{\kappa}_\infty t) \pr(|\BP(\alpha_* t)| < \infty).
 \end{align} Now, note that using \cite{remco-book-1}*{Theorem 3.16}, we have \begin{align}\label{eqn:bp-mbp-survival bound}
     \pr\left(|\BP(\lambda)| < \infty\right) = e^{-\lambda} \sum_{j=1}^\infty \frac{j^{j-1}}{j!} \left(\lambda\ e^{-\lambda}\right)^j \leq C e^{-\lambda} \sum_{j=1}^\infty \left(\frac{j^{j-1}}{j^{j+1/2}e^{-j}}\right) \left(e^{-1}\right)^j  \leq C'e^{-\lambda},\qquad \lambda\geq 0
  \end{align} for some constants $C$ and $C'$ independent of $\lambda$. The second inequality follows because of the Stirling's approximation and the inequality $xe^{-x} \leq e^{-1}$ for all $x\geq 0$. 
\ab{ Thus
 $$\ind_{\{(s,t) \in [0,k]^2\}} |f_k(s,t)| \le C'(1+\norm{\kappa}_\infty t)^{1/2}(1+\norm{\kappa}_\infty s)^{1/2}
 \exp(-\alpha_* (s+t)/2).
 $$
 Since the left side above converges to $\ind_{\{(s,t) \in [0,\infty)^2\}} |f_{\infty}(s,t)|$ a.e. $\lambda$ and, for the right side,
 $$\int_{[0,\infty)^2} C'(1+\norm{\kappa}_\infty t)^{1/2}(1+\norm{\kappa}_\infty s)^{1/2}
 \exp(-\alpha_* (s+t)/2) ds \, dt <\infty,$$
 we have  by dominated convergence, }
 $\sigma_k = \int_0^k\int_0^k f_k(s,t)dsdt \to \sigma_\infty = \int_0^\infty\int_0^\infty f_\infty(s,t)dsdt$ as $k\to \infty$ and $\sigma_\infty < \infty$. This completes the proof of the lemma.
\end{proof} 

\ab{Before proceeding with the proof of Lemma \ref{lem:mst-tail-lemma} we present the following lemma which will be a key ingredient in the proof. The proof of the latter lemma is provided after the proof of Lemma \ref{lem:mst-tail-lemma}.}

\ab{\begin{lem}\label{lem:mst-supercritical-bounds} 
Fix $\eta >0$. For any $\alpha,\beta > 0$, define the following events \begin{align*}
       F_n^{\alpha} = \set{\abs{\cC_{(1),n}(t_c+\eta)} \leq \alpha n}, \hspace{0.1cm} E_n^{\alpha,\beta} = \set{\abs{\cC_{(1),n}(t)} \geq \alpha n, \abs{\cC_{(2),n}(t)} \geq \beta \log n \text{ for \ab{some} } t\geq t_c+\eta}.
   \end{align*}
 There exists $B_1,B_2 > 0$ such that $\sqrt{n}\pr\left(F_n^{B_1} \cup E_n^{B_1,B_2} \right) \to 0$ as $n\to \infty$.
   \end{lem} 
We now provide the proof of Lemma \ref{lem:mst-tail-lemma}.}
\begin{proof}[Proof of Lemma \ref{lem:mst-tail-lemma}]
Note that $\gamma_n - \gamma_{n,k} = \xi_n^{(1)}(k) + \xi_n^{(2)}(k)$ where \begin{align*}
    \xi_n^{(1)}(k) = \int_k^{A\log n} \sum_{\mvl \in \fT_k\setminus\set{\mv0}} Y_n(\mvl,t) dt, \qquad \xi_n^{(2)}(k) = \int_0^{A\log n} \sum_{\mvl \in \fT\setminus\fT_k} Y_n(\mvl,t) dt.
\end{align*}
We first argue that \begin{align}\label{eqn:xi-1}
 \limsup_{k\to\infty} \limsup_{n\to \infty} \E\left(|\xi_n^{(1)}(k)|\right) = 0.
\end{align} By Cauchy-Schwartz inequality, we have $$\E\left(|\xi_n^{(1)}(k)|\right) \leq \int_k^{A\log n} \sqrt{n\var\left(\sum_{\mvl\in \fT_k\setminus\set{\mv0}}\pi_n(\mvl,t)\right)}dt.$$ Let $U_n$ is the uniformly picked vertex in $[n]$. 
\sa{Using Lemma \ref{lem:variance-sum-bound} (for the choice $\cA = \fT_{k}\setminus\set{\mv0}$) we have,} \ab{for a suitable $m \in \bN$, depending only on $K$, }
 \begin{align*}
n\var\left(\sum_{\mvl\in \fT_k\setminus\set{\mv0}}\pi_n(\mvl,t)\right) &\leq (1+\norm{\kappa_n}_\infty t) \pr\left(|\cC_n(U_n,t)|\leq k\right)\\
&\leq C\frac{(1+\norm{\kappa_n}_\infty t) t\ab{k^m}}{n} + (1+\norm{\kappa_n}_\infty t) \pr(|\MBP_{\mu_n}(t\kappa_n,\mu_n)| \leq  k)
\end{align*} where the last line follows from the branching process approximation in Lemma \ref{lem:TV-approximation} or \cite{bollobas2007phase}*{Lemma 9.6}. Also, as argued in \eqref{eqn:bp-mbp-survival bound}, we have $\pr\left(|\MBP_\mu(t\kappa,\mu)| < \infty \right) \leq \pr\left(|\BP(\alpha_*t)| <\infty \right) \leq C'e^{-\alpha_*t}$. Therefore $$
\E\left(|\xi_n^{(1)}(k)|\right) \leq \int_k^{A\log n} \left[C\frac{(1+\norm{\kappa_n}_\infty t) t\ab{k^m}}{n} + C'(1+\norm{\kappa_n}_\infty t) e^{-\alpha_*t}\right]^{\frac{1}{2}} dt.$$ From the above, it follows that \eqref{eqn:xi-1} holds.  We next focus on $\xi_n^{(2)}(k)$. There are three parts to this term, the integrals away from the critical regime, namely for some fixed $\eta$, the integral over $(0, t_c -\eta)$, the  integral from $(t_c+\eta, A\log{n} )$ and finally the contribution over the critical regime namely $(t_c-\eta, t_c+\eta)$. Let us now deal with each of these integrals. 

\ab{ \sa{Again using Lemma \ref{lem:variance-sum-bound} ( with $\cA = \fT\setminus\fT_k$)}, for any $\eta > 0$, and $t \in [0, t_c-\eta]$, we have
\begin{equation}\label{eq:insideterm}
\E\abs{\sum_{\mvl\in \fT\setminus\fT_k}Y_n(\mvl,t)} \le
\sqrt{n\var\left(\sum_{\mvl\in \fT\setminus\fT_k}\pi_n(\mvl,t)\right)} \le \sqrt{(1+\norm{\kappa_n}_\infty t) \pr\left(\abs{\cC_n(U_n,t)} \geq k\right)}
\end{equation}
where $U_n$ is the uniformly picked vertex from $[n]$.
Consequently,
\begin{align}\label{eqn: 2313123123123213}
     \int_0^{t_c -\eta} \E\abs{\sum_{\mvl\in \fT\setminus\fT_k}Y_n(\mvl,t)}dt 
     &\leq \int_0^{t_c -\eta} \sqrt{(1+\norm{\kappa_n}_\infty t) \pr\left(\abs{\cC_n(U_n,t)} \geq k\right)} dt. 
\end{align}}  Now, note that, \ab{for a suitable $m \in \bN$,}
\begin{align}\label{eqn:090909090}
\pr\left(\abs{\cC_n(U_n,t)} \geq k\right) &\leq \frac{C\ab{k^m}t}{n} + \pr\left(\abs{\MBP_{\mu_n}(t\kappa_n,\mu_n)}\geq k\right) \nonumber\\
&\leq \frac{C\ab{k^m}t}{n} + \frac{C'\ab{k^m}t}{\sqrt{n}} + \pr\left(\abs{\MBP_{\mu}(t\kappa,\mu)}\geq k\right)
\end{align} where the first line follows from the branching process approximation as earlier. The second line follows from Lemma \ref{lem:mbp-tv-approx}, since the event branching process has size greater than $k$ depends only the off-spring sizes of first $k$ individuals. Finally, by Remark \ref{rem:exp-decay}, we know that there exists $\delta = \delta(\eta) > 0$ such that $\pr(\abs{\MBP_{\mu}(t\kappa,\mu)} \geq k) \leq Ce^{-\delta k}$ for all $t\leq t_c-\eta$. Using all these observations in the \eqref{eqn: 2313123123123213}, we get for all $\eta > 0$,  \begin{align}\label{eqn:xi-2-subcritical}
   \limsup_{k\to \infty} \limsup_{n\to \infty} \int_0^{t_c -\eta} \E\left(\abs{\sum_{\mvl\in \fT\setminus\fT_k}Y_n(\mvl,t)}\right)dt = 0.
\end{align} 
Next, from the estimate in \eqref{eq:insideterm}, we  have \begin{align}\label{eqn:xi-2-critical}
   \limsup_{\eta \to 0}\limsup_{k\to \infty}\limsup_{n\to \infty} \int_{t_c-\eta}^{t_c +\eta} \E\left(\abs{\sum_{\mvl\in \fT\setminus\fT_k}Y_n(\mvl,t)}\right)dt &\leq\lim_{\eta \to 0}\lim_{n\to \infty} \int_{t_c-\eta}^{t_c + \eta} \sqrt{(1+\norm{\kappa_n}_\infty t} dt = 0.
\end{align}

Finally we consider the integral of the tail sum in the supercritical regime.
\ab{We will argue that, for any $\eta, \epsilon  > 0$, 
\begin{equation}
\label{lem:xi-2-supercritical}
\limsup_{k\to \infty}\limsup_{n\to \infty}\pr\left(\abs{\int_{t_c+\eta}^{A\log n} \sum_{\mvl\in \fT\setminus\fT_k} Y_n(\mvl,t) dt} \geq \epsilon\right) =0.
\end{equation}}
\ab{For the fixed choice of $\eta$, let $B_1$ and $B_2$ be as in Lemma \ref{lem:mst-supercritical-bounds} and let $M=B_2$. }

   By the above lemma, in order to prove \eqref{lem:xi-2-supercritical}, it is sufficient to prove
   \begin{align}\label{eqn:10001}
       \limsup_{k\to \infty}\limsup_{n\to \infty}\pr\left(\abs{\int_{t_c+\eta}^{A\log n} \sum_{\mvl\in \fT_{M\log n}\setminus\fT_k} Y_n(\mvl,t) dt} \geq \epsilon\right) =0.
   \end{align} 
\ab{Indeed, letting $(F_n^{B_1} \cup E_n^{B_1,B_2})^c = G_n$, on $G_n$, for $t \ge t_c+\eta$
\begin{align*}
\sum_{\mvl\in \fT\setminus\fT_k} Y_n(\mvl,t) &= 
\sum_{\mvl\in \fT_{M\log n}\setminus\fT_k} Y_n(\mvl,t)
+
\sqrt{n} \sum_{\mvl\in \fT\setminus\fT_{M\log n}} \left(\pi_n(\mvl,t) - \E_{\mu_n}(\pi_n(\mvl,t))\right)\\
&= \sum_{\mvl\in \fT_{M\log n}\setminus\fT_k} Y_n(\mvl,t) + \frac{1}{\sqrt{n}} O_{\bP}(1) - \sqrt{n} \E_{\mu_n}( \sa{\sum_{\mvl\in \fT\setminus\fT_{M\log n}}}\pi_n(\mvl,t)),
\end{align*}
where the last inequality follows on noting that on $G_n$, there is exactly one component with type in $\fT\setminus\fT_{M\log n}$.
Clearly, the integral of the second term on the right side above, over the interval $[t_c+\eta, A \log n]$ converges to $0$ as $n \to \infty$.
For the third term, note that
\begin{align}
\sqrt{n} \E_{\mu_n}\left(\int_{t_c+\eta}^{A\log n} \sum_{\mvl\in \fT\setminus\sa{\fT_{M\log n}}}\pi_n(\mvl,t) dt\right) &\le \sqrt{n}\ab{A}\log n \frac{\bP_{\mu_n}(F_n^{B_1} \cup E_n^{B_1,B_2})}{M\log n}\nonumber\\
&+ \sqrt{n} \E_{\mu_n}\left(\ind\{G_n\}\int_{t_c+\eta}^{A\log n} \sum_{\mvl\in \fT\setminus\sa{\fT_{M\log n}}}\pi_n(\mvl,t)dt\right)\nonumber\\
&\le \sqrt{n}\ab{\frac{A}{M}}\bP_{\mu_n}(F_n^{B_1} \cup E_n^{B_1,B_2})
+ \sqrt{n} O_{\bP}(1/n),\label{eq:133n}
\end{align}
where the first inequality follows on observing that
$$\sum_{\mvl \in \fT\setminus \fT_{M\log n}} \pi_n(\mvl,t) = \E\left(\frac{\ind\set{\abs{\cC_n(U_n,t)} \geq M\log n}}{\abs{\cC_n(U_n,t)}}| \cG_n(t)\right)\leq \frac{1}{M\log n}$$
and the last inequality once more uses the fact that, on $G_n$,
there is exactly one component with type in $\fT\setminus\fT_{M\log n}$, for every $t \ge t_c+\eta$. The right side in \eqref{eq:133n} converges to $0$ from Lemma \ref{lem:mst-supercritical-bounds}.
Combining the above observations we conclude that it to prove \eqref{lem:xi-2-supercritical} it suffices to show \eqref{eqn:10001}.}

   We now prove \eqref{eqn:10001}. Note that \sa{ by Lemma \ref{lem:variance-sum-bound} (with $\cA = \fT_{M\log n}\setminus\fT_k$)}\begin{align}\label{eqn:10000}
       \int_{t_c + \eta}^{A\log n} \E\left(\abs{\sum_{\mvl\in \fT_{M\log n}\setminus\fT_k}Y_n(\mvl,t)}\right)dt &\leq \int_{t_c + \eta}^{A\log n} \sqrt{n\var\left(\sum_{\mvl\in \fT_{M\log n}\setminus\fT_k}\pi_n(\mvl,t)\right)}dt \nonumber\\
     &\leq \int_{t_c + \eta}^{A\log n} \sqrt{(1+\norm{\kappa_n}_\infty t) \pr\left(k \leq \abs{\cC_n(U_n,t)} \leq M\log n\right)} dt.
   \end{align} where $U_n$ is the uniformly picked vertex from $[n]$. Now, note that using the arguments to obtain \eqref{eqn:090909090}, we have, \ab{for some $m \in \bN$}
   \begin{align}\label{eqn:109109}
\pr\left(k \leq \abs{\cC_n(U_n,t)} \leq M\log n\right) &\leq \frac{C\ab{(\log n)^m}t}{n} + \frac{C'\ab{(\log n)^m}t}{\sqrt{n}} +  \pr\left(k \leq \abs{\MBP_{\mu}(t\kappa,\mu)} \leq M\log n\right).
\end{align} 
\ab{Next, by Cauchy-Schwarz inequality,
\begin{equation}\label{eq:cascn}
\pr\left(k \leq \abs{\MBP_{\mu}(t\kappa,\mu)} \leq M\log n\right)
\le \left(\pr\left( \sa{k \leq \abs{\MBP_{\mu}(t\kappa,\mu)} < \infty}\right)\right)^{1/2}
\left(\pr\left(\abs{\MBP_{\mu}(t\kappa,\mu)} < \infty\right)\right)^{1/2}.
\end{equation}
Recall $\rho_t(x)$ the survival probability of $\MBP_x(t\kappa,\mu)$.  Since the entries of $\kappa$ are bounded below by $\alpha_*$, we have $1-\rho_t(x) \leq \pr\left(\abs{\BP(\alpha_* t)} < \infty \right) \leq C_1 e^{-\alpha_*t}$. Therefore, by the definition of dual operator $\hat{T}_{t\kappa,\mu} = T_{t\kappa,\hat{\mu}_t}$, it follows that for $\norm{\hat{T}_{t\kappa,\mu}} \leq C_2 t e^{-\alpha_*t}$. Therefore, there exists $t_0 > t_c$ such that  $\norm{\hat{T}_{t\kappa,\mu}} < 1$ \sa{for all $t \geq t_0$.} By Lemma \ref{lem:dual-norm} and Lemma \ref{lem:mgf-mbp}, it follows that there exists $\delta > 0 $ such that 
\begin{equation*}
\pr\left( \sa{k \leq \abs{\MBP_{\mu}(t\kappa,\mu)} < \infty}\right) \le \pr\left(k \leq \abs{\MBP_{\hat{\mu}_t}(t\kappa,\hat{\mu}_t)}\right) \leq C''e^{-\delta k}, \mbox{ for all } t\geq t_c+\eta.
\end{equation*}
Also, using \eqref{eqn:bp-mbp-survival bound},
$$
\pr\left(\abs{\MBP_{\mu}(t\kappa,\mu)} < \infty\right) \le C' e^{-\alpha_*t}.
$$
Combining the last two displays with \eqref{eq:cascn}, we have from \eqref{eqn:10000} that \eqref{eqn:10001} holds.}
 This finishes the proof of \eqref{eqn:10001}, and consequently of \eqref{lem:xi-2-supercritical}.

Finally,  \eqref{eqn:xi-1}, \eqref{eqn:xi-2-subcritical},\eqref{eqn:xi-2-critical} and  \eqref{lem:xi-2-supercritical} complete the proof of Lemma \ref{lem:mst-tail-lemma}. 
\end{proof}

We now prove Lemma \ref{lem:mst-supercritical-bounds}.
\begin{proof}[Proof of Lemma \ref{lem:mst-supercritical-bounds}]
Fix $\eta > 0$ as in the statement of the lemma. We first argue that for some  $B_1 \geq 0$ $\sqrt{n} \pr(F_n^{B_1}) \to 0$ as $n\to \infty$. 
\ab{Let $\delta>0$ be such that with  the kernel $\tilde \kappa$ defined as $\tilde \kappa(i,j) = \kappa(i,j) - \delta$, $i,j \in [K]$, $\tilde \kappa$ is irreducible and 
$\MBP_{\mu}\left((t_c+\eta)\tilde \kappa,\mu\right)$ is a supercritical branching process (existence of such a $\delta$ is guaranteed by continuity of norm of integral operators).} Since $\cS$ is finite, there exists $n_0 \in \bN$ such that for all $n\geq n_0$, 
\ab{\begin{equation}\label{eq:100n}
1-\exp\left(-(t_c+\eta)\kappa_n(i,j)/n\right) \geq (t_c+\eta) \frac{\tilde \kappa(i,j)}{n} \text{for all } i,j\in \cS.\end{equation}}  Let $\tilde{\cC}_{(1),n}$ be the largest component in the graph, \ab{in which starting with the type vector $\mvx^n$,  an edge between vertices of type $i,j \in \cS$ is placed with probability $(t_c+\eta)\tilde \kappa(i,j)/n$.} 

Let $\fF_n(\tilde{\cC}_{(1),n}) = \tilde{V}_n$ be the type of the largest component as defined in Definition \ref{defn:irg-comp-type}. Then by \cite{mdp-irg-yu}*{Theorem 1} \ab{(letting $a_n = n^{1/4}$ in that theorem) we have, for a certain vector $\vc \in \bR^K$, $\vc = (c_1, \ldots, c_K)$, $0 < c_i <\mu_i$, $i \in [K]$ (see \cite{mdp-irg-yu}),}
$$\frac{\tilde{V}_n - (\mu-\vc)n}{n^{3/4}}$$
satisfies a moderate deviations principle (MDP) with speed  $n^{1/2}$  and rate
 \begin{align*} \vI(x) = \la (I-D_{\vc}\tilde{\kappa})x, D_{\vv}(I-D_{\vc}\tilde{\kappa})x, \; x \in \bR^K, 
\end{align*} 
 where  $D_{\vv} = \text{Diag}\left(\frac{\mu_i}{c_i(\mu_i-c_i)}\right)$  and  $D_{\vc} = \text{Diag}\left(c_i\right)$, \ab{and for a vector $\vu \in \bR^K$, $\text{Diag}(u_i)$ denotes the $K\times K$ diagonal matrix with diagonal entries $\{u_i, i \in [K]\}$.}

Note that by \cite{mdp-irg-yu}*{Lemma 9} we have $(I-D_{\vc}\tilde{\kappa})D_\vv (I-D_{\vc}\tilde{\kappa})$ is strictly positive definite. Therefore, there exists $\theta > 0$ such that $\vI(x) \geq \theta\norm{x}^2$ for all $x \in \bR^K$. For any closed set $F \in \bR^K$, MDP implies  
$$\limsup_{n\to \infty} \frac{1}{\sqrt{n}} \pr\left(\frac{\tilde{V}_n - (\mu-\vc)n}{n^{3/4}} \in F\right) \ab{\le} -I(F).$$ Let $F = \set{x \in \bR^K : \norm{x} \geq 1}$. Let $a = (\mu-c)\cdot \ind$, \ab{where
$\ind$ is the $K$-dimensional vector of ones.}
Let  $B_1 = a/2$. Note that for $n_1 \in \bN$ such that $an^{\frac{1}{4}}/2K \geq 1$, we have \begin{align*}
    \pr\left(F_n^{B_1}\right) &\leq \pr\left(|\tilde{\cC}_{(1),n}| \leq B_1 n\right)\\
    &\leq \pr\left(\left(\tilde{V}_n - (\mu-\vc)n\right)_i \leq -\frac{an}{2K} \text{ for some } i \in \cS\right)\\
    &\leq \pr\left(\norm{\frac{\tilde{V}_n - (\mu -\vc)n}{n^{\frac{3}{4}}}} \geq \frac{an^{\frac{1}{4}}}{2K}\right) \leq \pr\left(\frac{\tilde{V}_n - (\mu -\cC)n}{n^{\frac{3}{4}}} \in F\right) \text{ for } n\geq n_1,
\end{align*} 
\ab{where the first inequality is a consequence of \eqref{eq:100n}.}
By MDP, there exists  $n_2 \geq n_1$ such that we have $$\pr\left(F_n^{B_1}\right) \leq \exp(-\vI(F)n^{\frac{1}{2}}/2) \leq  \exp(-\theta n^{\frac{1}{2}}/2)\text{ for all } n\geq n_2.$$ 
Therefore, $\sqrt{n}\pr(F_n^{B_1}) \to 0$ as $n\to \infty.$ 
\ab{Also, 
noting that the set $A_n^{\alpha, \beta}$ defined in the proof of Lemma \ref{lem:auxillary-irg} \sa{ contains } $E_n^{\alpha, \beta}$ with $T$ there replaced by $t_{c}+\eta$, we have from \eqref{eq:107n}, that with
$B_2:= \sa{2}\beta(B_1, t_c+\eta)$, 
$\sqrt{n}\pr(E_n^{B_1,B_2}) \to 0$
 as $n\to \infty.$} This completes the proof.
\end{proof}
\section*{Acknowledgements}
 Bhamidi and Sakanaveeti were partially supported by NSF DMS-2113662, DMS-2413928, and DMS-2434559.  Bhamidi and Budhiraja were partially funded by NSF RTG grant DMS-2134107.
 Budhiraja and Sakanaveeti were partially supported by NSF DMS-2152577.
Budhiraja would also like to thank the Isaac Newton Institute for Mathematical Sciences, Cambridge, for support and hospitality during the programme Stochastic systems for anomalous diffusion, where part of the work on this paper was undertaken. This  was supported by EPSRC grant EP/Z000580/1. Bhamidi and Sakanaveeti would like to thank the BIRS-Chennai Math Institute workshop on Network Models in December 2024 where part of this work was undertaken.

\bibliographystyle{plain}
\bibliography{fclt}

\appendix

\section{Conditions for weak convergence}\label{sec:appendix}

The following theorem gives a criterion for the tightness of a sequence of RCLL processes with values in a Polish space, see \cite{kurtz1981approximation}.

\begin{thm}
\label{thm:AppSemiMartTight}
Let $\cS$ be a Polish space and for $T>0$, $\{Y_n(t), 0 \le t \le T\}_{n\in\bb{N}}$ be a sequence of 
$\cS$-valued RCLL processes, with $Y_n$ being $\{\cF^n_t\}$-adapted, given on some filtered probability space
$(\Omega^n, \cF^n, \{\cF^n_t\}, \bP^n)$. Suppose that the following conditions are satisfied.

\begin{enumerate}
\item[$(T_1)$] $\{Y_n(t)\}_{n\in\bb{N}}$ is tight for every $t$ in a dense subset of $[0,T]$.
\item[$(A)$ ] For each  $\eps>0,\ \eta>0$  there exists a $\delta>0$ and $n_0$ with the property that for every collection $(\tau_n)_{n\in\bb{N}}$, where $\tau_n$ is an $\cF^n_t$-stopping time with $\tau_n\leq T-\delta$ for all $n \in \bN$,
\begin{equation*}
\sup_{n\geq n_0}\sup_{0\leq\theta\leq\delta}\pr\left(d(Y_n(\tau_n+\theta),Y_n(\tau_n))\geq\eta\right)\leq \eps,
\end{equation*}
where $d(\cdot,\cdot)$ is the distance on $\cS$.
\end{enumerate}
Then $\{Y_n\}_{n\in\bb{N}}$ is tight in $\bb{D}([0,T]:\cS)$.
\end{thm}

The following result is a minor adaptation of \cite{kurtz-ethier}[Theorem 11.2.1, 11.2.3].

\begin{thm}
\label{thm:fdd-fclt-kurtz}
Fix $N, d \geq 1$. Let $\beta_i: \bR^d \to [0,\infty)$ for $i=1,2,\dots,N$ be a collection of measurable functions, and $\Delta_1,\Delta_2,\dots,\Delta_N \in \bR^d$ be a collection of vectors. 

Consider the process $\set{\mvxi_n(t) \in \bR^d:t\geq 0}$ given as the solution to the following stochastic evolution equation \begin{align*}
    \mvxi_n(t) = \mvxi_n(0) + \frac{1}{n} \sum_{i=1}^N \Delta_i Y_i\left(n\int_0^t \beta_{n,i}(s)ds\right), \; t \ge 0,
\end{align*} where $\set{Y_i(\cdot): 1\leq i\leq N}$ is a collection of independent Poisson processes of rate $1$, and $\beta_{n,i}(s) = \beta_i(\mvxi_n(s)) + \epsilon_{n,i}(s)$ for $i\leq N$ where $\epsilon_{n,i}$ are bounded $\cF_t$-progressively measurable processes, with $\cF_t = \sigma \{\mvxi_n(s), s \le t\}$, $t \ge 0$.

Let $F:\bR^d \to \bR^d$ be defined as $F(\mvx) = \sum_{i=1}^N \Delta_i \beta_i(\mvx)$, $\mvx \in \bR^d$. Suppose the following hold: \begin{enumeratei}
    \item There exists a constant $\gamma \geq 0$ such that, a.s., $\abs{\epsilon_{n,i}(s)} \leq \gamma/n $ for all $s\geq 0$ and $i\leq N$. 

    \item The function $\beta_i$ are differentiable with continuous derivatives for all $1\leq i\leq N$. 

    \item There exists $\mvxi(0) \in \bR^d$ such that $\mvxi_n(0)$ converges in probability to $\mvxi(0)$.
    \item  There exists a function $\mvxi:[0,\infty) \to [0,\infty)$ which solves the  differential equation
    $$\mvx'(t)  = F(\mvx(t)), \qquad \mvx(0) = \mvxi(0). $$

\end{enumeratei} 
Then $\mvxi_n$ converges in probability, in $\bD([0,\infty],\bR^d)$, to $\mvxi$.

Suppose, in addition that, $X_n(0) = \sqrt{n}\left(\mvxi_n(0) - \mvxi(0)\right) \convd X(0)$
for some $\bR^d$ valued random variable. Let $\mvB$ be a $d$-dimensional Brownian motion, given on the same probability space as $X(0)$, and independent of $X(0)$. Then 
$\vX_n(\cdot) = \sqrt{n}\left(\mvxi_n(\cdot) - \mvxi(\cdot)\right) \convd \vX(\cdot)$ in $\bD([0,\infty),\bR^d)$, where $\vX(\cdot)$ is the solution to the following \emph{linear} stochastic differential equation \begin{align}\label{eqn:thm-kurtz-sde}
    d\vX(t) = \partial F(\mvxi(t)) \vX(t) dt + G(t)d\mvB(t)
\end{align} where $ \partial F(\mvx)$ is the $d\times d$ matrix giving the derivative of the function $F$ with respect to $\mvx$, and $G(t)$ is the symmetric square-root of the the non-negative definite $d\times d$ matrix $\Phi(t)$ defined as \begin{align*}
    \Phi(t) = \sum_{i=1}^N \Delta_i \Delta_i^T \beta_i(\mvxi(t)).
\end{align*} 

\end{thm}

\subsection{Proof of Lemma \ref{lem:weakcgce}}

Let $F: S_4 \to \bR$ be a continuous and bounded function. Denote by
$\theta_n$ (resp. $\theta$) the probability distribution of $U_n$ (resp. $U$) on $S_1$.
We need to show that
\begin{equation}\label{eq:1100}
\int E(F(\Phi_n(u, Y_n(u)))) \theta_n(du) \to \int E(F(G(u,Z))) \theta(du) \mbox{ as } n \to \infty.
\end{equation}
Note that $\theta_n \to \theta$ weakly as $n \to \infty$. Thus, from the assumed continuity property of $G$, 
$$
\int E(F(G(u,Z))) \theta_n(du) \to \int E(F(G(u,Z))) \theta(du) \mbox{ as } n \to \infty.
$$
Thus it suffices to show that
$$\int \left|E(F(\Phi_n(u, Y_n(u)))) - E(F(G(u,Z)))\right| \theta_n(du) \to 0 \mbox{ as } n \to \infty.
$$
Since $\{\theta_n\}$ is tight it suffices to show that, for any compact $K^* \in S_1$,
$$\sup_{u \in K^* \cap S^n_1} \left|E(F(\Phi_n(u, Y_n(x)))) - E(F(G(u,Z)))\right| \to 0 \mbox{ as } n \to \infty.
$$
However this is immediate from the assumptions made in the statement of the lemma.

\end{document}